\documentclass[reqno,11pt]{amsart}
\usepackage{amsmath,amssymb,mathrsfs,color}
\usepackage{graphicx}
\usepackage{epstopdf}
\usepackage{hyperref}
\usepackage{enumitem}
\usepackage{soul}
\usepackage{cases}
\hypersetup{hypertex=true,
            colorlinks=true,
            linkcolor=blue,
            anchorcolor=blue,
            citecolor=red}
\let\cal=\mathcal

\def\R{{\mathbb R}}

\def\eps{\varepsilon}

\def\cA{{\mathcal A}}
\def\cB{{\mathcal B}}

\def\cD{{\mathcal D}}

\def\cF{{\mathcal F}}

\def\cK{{\mathcal K}}
\def\cL{{\mathcal L}}

\def\cN{{\mathcal N}}

\def\cP{{\mathcal P}}

\def\cT{{\mathcal T}}

\def\cW{{\mathcal W}}
\def\cX{{\mathcal X}}
\def\cY{{\mathcal Y}}

\def\mE{{\mathbb E}}

\def\mP{{\mathbb P}}

\def\mT{{\mathbb T}}

\def\b1{{\mathbbm 1}}

\def\sA{{\mathscr A}}

\def\sD{{\mathscr D}}

\def\sL{{\mathscr L}}

\def\sS{{\mathscr S}}

\newtheorem{thm}{Theorem}[section]
\newtheorem{cor}[thm]{Corollary}
\newtheorem{lem}[thm]{Lemma}
\newtheorem{prop}[thm]{Proposition}

\theoremstyle{definition}
\newtheorem{de}[thm]{Definition}
\theoremstyle{remark}
\newtheorem{rem}[thm]{Remark}
\newtheorem{exam}[thm]{Example}

\numberwithin{equation}{section}

\newcommand{\rmd}{{\rm d}}
\newcommand{\dist}{{\rm{dist}}}
\newcommand{\rme}{{\rm e}}
\newcommand{\p}{\partial}
\allowdisplaybreaks

\begin{document}

\title[Random Attractors for McKean-Vlasov SDEs]{Random Attractors for McKean-Vlasov SDEs}

\author{Mengyu Cheng}
\address{M. Cheng: School of Mathematical Sciences, Dalian University of Technology, Dalian 116024, P. R. China}
\email{mycheng@dlut.edu.cn; mengyu.cheng@hotmail.com}

\author{Xianjin Cheng}
\address{X. Cheng: School of Mathematical Sciences, Dalian University of Technology, Dalian 116024, P. R. China}
\email{xjcheng1119@hotmail.com}

\author{Zhenxin Liu}
\address{Z. Liu: School of Mathematical Sciences, Dalian University of Technology, Dalian 116024, P. R. China}
\email{zxliu@dlut.edu.cn}

\keywords{Random attractors; McKean-Vlasov equations; Random dynamical systems; Navier-Stokes equations}
\subjclass[2020]{37L55; 37L30; 60H15.}

\begin{abstract}
In this paper, we mainly focus on the existence of random attractors for 
McKean-Vlasov stochastic differential equations on a separable Hilbert space $H$.
A significant challenge arises from the distribution-dependence of the coefficients, 
thereby causing the lack of the stochastic flow property on $H$.
To address this issue, we first transform the original equation into a system 
on the product space $H \times \mathcal{P}(H)$ and
consider the existence of random attractors on this space.
We then analyze cocycles associated with two parametric dynamical systems. 
Within this framework, we define the corresponding pullback random attractor and
develop a general theory for the existence of random attractors for such cocycles.
Finally, we apply our theoretical results to McKean-Vlasov stochastic ordinary differential equations, McKean-Vlasov stochastic reaction-diffusion equations, and McKean-Vlasov stochastic 2D Navier-Stokes equations.
In the case where the attractor reduces to a singleton set $\cA(\omega):=(\xi(\omega),\mu_\infty)$, 
we show that $\xi$ corresponds to the stationary solution for the decoupled SPDE, 
satisfying $\mP\circ[\xi]^{-1}=\mu_\infty$.
\end{abstract}

\maketitle
\setcounter{tocdepth}{3}
\tableofcontents

\section{Introduction}

Consider the following McKean-Vlasov stochastic differential equation (SDE) on separable Hilbert space $H$:
\begin{align}\label{In:maineq1}
\rmd X_t=b(X_t,\sL_{X_t})\rmd t+\sigma(X_t)\rmd W_t,
\end{align}
where $\sL_{X_t}$ is the law or distribution of $X_t$ and $W_t$ is a $Q$-Wiener process.
McKean-Vlasov SDEs, also known as mean-field SDEs, were initially proposed by Kac \cite{Kac1,Kac2}
 in the study of the Boltzmann equation for the density of particles in a diluted monatomic gas and the Vlasov dynamics equation used to describe a stochastic toy model. 
 These equations are closely related to mean-field interacting particle systems. 
 The propagation of chaos \cite{Sznitman} indicates that such equations are the limiting equations for mean-field interacting particle systems as the number of particles tends to infinity. 
 They have a wide range of applications in mean-field game theory, stochastic control, financial mathematics, and other fields \cite{LL1,LL2,LL,C,LM}.
The well-posedness of solutions to McKean-Vlasov SDEs and 
the existence and uniqueness of invariant measures have been proven under various conditions, 
such as in \cite{LM,W1,F,MV,Ma,W2,HHL24}.

Note that the solutions of \eqref{In:maineq1} describe nonlinear diffusion \cite{McKean}, 
where the motion of particles is influenced not only by their own positions 
but also by the effect of all other particles in the system (i.e., the mean-field effect). 
This leads to many interesting phenomena in the motion of particles, 
such as the existence of multiple invariant measures \cite{Dawson} and 
the occurrence of phase transitions \cite{Tugaut} under central interactions. 
This also makes many classical frameworks of dynamical systems inapplicable, 
such as the inability to generate a random dynamical system (RDS)
for \eqref{In:maineq1} as in the classical case. 
Therefore, to study the dynamics of \eqref{In:maineq1}, 
we need to develop new methods and theories.

The attractor is an important concept in dynamical systems, 
which can be used to characterize the long time asymptotic behavior of systems. 
To study the asymptotic behavior of RDS, 
the concept of random attractors was introduced and studied in 
\cite{Crauel2, CF, FS96, Schmalfuss}. 
Since then, it has attracted the attention of many scholars, such as semilinear stochastic partial differential equations (SPDE)
with additive noise \cite{Caraballo,BLW,LG,WangBX,WGLN24}. 
As for the research on non-semilinear SPDEs, see e.g. \cite{BGLR, BZL, GLR11}.

To the best of our knowledge, there is no work on random attractors for McKean-Vlasov SDE.
Existing studies on pullback attractors for McKean-Vlasov SDE mainly focus on measure attractors; see e.g. \cite{SSLW24}.
In contrast to these results, the primary aim of this paper is to establish 
the existence of random attractors for system \eqref{In:maineq1} from a pathwise perspective. 

It is worth noting that when $b(x,\mu)\equiv b(x)$, equation \eqref{In:maineq1} reduces to a classical SDE.
Under standard assumptions, there exists a unique solution $X(s,t,x),t\geq s$ with 
initial condition $X(s,s,x)=x\in H$. 
Moreover, one can show that there exists a subset $\Omega_0\subset\Omega$ of full measure 
such that $X(s,t,x)$ forms a {\it stochastic flow}, that is,
for any $s<r<t$, $x\in H$ and $\omega\in\Omega_0$,
\begin{equation}\label{1111:01}
X(r,t,X(s,r,x)(\omega))(\omega)=X(s,t,x)(\omega),
\end{equation}
which implies that the Chapman--Kolmogorov equation is satisfied.
However, a major difficulty in the McKean-Vlasov SDE framework arises from the distribution-dependent nature of the coefficeints, 
which prevents \eqref{In:maineq1} from satisfying the stochastic flow property $\eqref{1111:01}$ on $H$. 
As a result, the McKean-Vlasov SDE does not generate an RDS on $H$. 
To address this issue, we adopt the method from \cite{BLPR} and \cite{RRW22}
to transform equation~\eqref{In:maineq1}  into
\begin{numcases}
{}
\rmd X_{t}=b(X_t,\sL_{Y_t})\,\rmd t+\sigma(X_t)\rmd W_t,\quad X_s=x\in H,\label{In:maineq2-1}\\
\rmd Y_{t}=b(Y_t,\sL_{Y_t})\,\rmd t+\sigma(Y_t)\rmd W_t,~\sL_{Y_s}=\mu\in \cP_2(H). \label{In:maineq2-2}
\end{numcases}
Specifically, once a solution to the McKean-Vlasov SDE \eqref{In:maineq2-2} is obtained, 
the time marginal law $\sL_{Y_t}$ is determined. 
Fixing this law in \eqref{In:maineq2-1}, it is straightforward that
\eqref{In:maineq2-2} becomes a time-inhomogeneous SDE. 
Define $P_{s,t}^*\mu:=\sL_{Y_t}$ with the initial distribution $\sL_{Y_s}=\mu$.
It can be shown that $P_{s,t}^*$ forms a semi-flow on $\cP_2(H)$, i.e.
for any $0\leq s<r<t$ and $\mu\in\cP_2(H)$, 
$$
P_{s,t}^*\mu=P_{r,t}^* P_{s,r}^*\mu.
$$
Note that \eqref{In:maineq2-1} is non-autonomous,
as its evolution over time depends on $P_{s,t}^*\mu$.
To emphasize the dependence of the solution on the initial distribution
$\mu$ of $\sL_{Y_t}$, we denote it by $X_{s,t}^{x,\mu}$.

To analyze the long-term behaviour of the coupled system
\eqref{In:maineq2-1}-\eqref{In:maineq2-2},
we employ the classical method of skew product,
a technique widely used in the study of non-autonomous dynamical systems.
Then we have the following stochastic flow property on 
the product space $H\times\cP_2(H)$:
\begin{align}\label{flow}
( X _{s,t}^{x,\mu},P_{s,t}^*\mu)=\Big( X _{r,t}^{ X _{s,r}^{x,\mu},P_{s,r}^*\mu},P_{r,t}^*P_{s,r}^*\mu\Big),~s\leq r\leq t.
\end{align}
This structure enables the system to generate an RDS
on $H\times\cP_2(H)$.

We now turn to the long-term behaviour of the system on $H$. Consider the cocycle
\begin{align*}
\varphi:\quad
&
 \R_+ \times \Omega \times \cP_2(H) \times H \rightarrow H,\\
&
 (t,\omega,\mu,x)\mapsto \varphi(t,\omega,\mu,x):=X_{0,t}^{x,\mu}(\omega).
\end{align*}
Since the evolution of $\varphi$ depends not only on $\omega$ but also on $\mu$,
we first develop a general framework for 
random attractors of cocycle defined over two parametric dynamical systems 
$(\Omega,\cF,\mathbb{P},\{\theta_{1,t}\}_{t\in\R})$ and 
$(\cP, \{\theta_{2,t}\}_{t\geq0})$,
where the later one is just a semi-dynamical system with time forward.
To formulate the pullback attractor, we introduce a time backward extension
$\{\Theta_{2,t}\}_{t\leq0}$ of $\theta_{2,\cdot}$ on the global attractor $\cP^*$ of $\{\theta_{2,t}\}_{t\geq0}$
(see e.g. \cite{Hale88, KR2011}).
Since $\{\Theta_{2,t}\}_{t\leq0}$ is a set-valued semi-dynamical system,
we give the definition of random attractors of 
$\varphi:\R_+\times\Omega\times\cP^*\times\cX\rightarrow\cX$ over 
the driving systems in the sense of Definition \ref{deRPA}.
Let
\begin{align*}
\Phi:\quad
&
\R_+\times\Omega\times\cP\times \cX\rightarrow\cP\times \cX\\
&
(t,\omega,(p,x))\mapsto\Phi(t,\omega)(p,x):=\left(\theta_{2,t}p,\varphi(t,\omega,p,x)\right)
\end{align*}
be a continuous RDS over $(\Omega,\cF,\mathbb{P},\{\theta_{1,t}\}_{t\in\R})$.
If $\varphi$ is asymptotically compact and have a closed attracting subset on $\cX$,
then we show that $\Phi$ admits a random attractor 
$\sA:=\left\{\sA(\omega):\omega\in\Omega\right\}$ (Proposition \ref{prop0117}).
Furthermore, we prove that $\varphi$ also admits a random attractor $\cA$, 
where 
$$
\cA(\omega,p):=\left\{x\in \cX: (x,p)\in \sA(\omega)\right\};
$$
see Theorem \ref{thm0330}.

Subsequently, we apply this theoretical result to several representative models:
McKean-Vlasov stochastic ordinary differential equations (SODE), 
McKean-Vlasov stochastic reaction-diffusion equations, 
and McKean-Vlasov stochastic 2D Navier-Stokes equations.
In the case where the attractor $\sA$ reduces to a singleton set, 
denoted as $\{(\xi(\omega),\mu_\infty):\omega\in\Omega\}$, 
it holds that $\sL_{\xi}=\mu_\infty$, where   
$\xi$ is the stationary solution for the decoupled SPDE \eqref{In:maineq2-1} and
$\mu_\infty$ is the unique invariant measure 
of \eqref{In:maineq2-2}; see Theorem \ref{thmRDeq} for more details.

This paper is organized as follows. 
In Section \ref{mainR}, we present our main results.
In Section \ref{secPGR}, we give the proof of the general framework for 
the existence of random attractors. 
In Section \ref{secPODE}, we show the existence of random attractors for
McKean-Vlasov SODEs with multiplicative noise by the general theorem stated in Section \ref{mainR}. 
In Section \ref{secPPDE}, we apply this general framework to McKean-Vlasov stochastic reaction-diffusion equation 
and McKean-Vlasov stochastic 2D Navier-Stokes equation separately.

{\bf Notation.}
Let $\mT^2$ be the 2-dimensional torus. For any constants $a,b>0$, 
set $a\wedge b:=\min\{a,b\}$. Define $\cP(H)$ be the space of all probability 
measures on Hilbert space $H$. And for any $p\geq1$, we set
$$
\cP_p(H):=\left\{\mu\in\cP(H): 
\mu(\|\cdot\|^p):=\int_{H}|x|_H^p\mu(\rmd x)<\infty\right\}
$$ 
and for any $\mu,\nu\in\cP_p(H)$,
$$
\cW_p(\mu,\nu):=\inf_{\pi\in C(\mu,\nu)}\left(\int_{H\times H}\|x-y\|_H^p
\pi(\rmd x,\rmd y)\right)^{\frac{1}{p}},
$$
where $C(\mu,\nu)$ is the set of all couplings for $\mu$ and $\nu$.
Let $\Xi$ be the space of continuous functions 
$f:\R^d\times\cP_2(\R^d)\rightarrow\R$ satisfying the following conditions:
\begin{itemize}
\item[(i)] 
For any $\mu \in \mathcal{P}_2(\mathbb{R}^d)$, the function $\mathbb{R}^d \ni x \mapsto f( x, \mu)$ is of class $C^{2}$, the functions $\partial_x f$ and $\partial^2_{xx} f$ being (jointly) continuous in $(x, \mu)$.

\item[(ii)]  
For any $x \in  \mathbb{R}^d$, the function $\mathcal{P}_2(\mathbb{R}^d) \ni \mu \mapsto f( x, \mu)$ is continuously L-differentiable and, for any $\mu \in \mathcal{P}_2(\mathbb{R}^d)$, we can find a version of the mapping $\mathcal{P}_2(\mathbb{R}^d) \times \mathbb{R}^d \ni (\mu, v) \mapsto \partial_\mu f( x, \mu)(v) \in \mathbb{R}^d$ which is locally bounded and is continuous at any $(x, \mu, v)$ such that $v \in \text{Supp}(\mu)$.

\item[(iii)]  
For the version of $\partial_\mu f$ mentioned above and for any $(x, \mu) \in  \mathbb{R}^d \times \mathcal{P}_2(\mathbb{R}^d)$, the mapping $\mathbb{R}^d \ni v \mapsto \partial_\mu f( x, \mu)(v) \in \mathbb{R}^d$ is continuously differentiable and its derivative, denoted by $\mathbb{R}^d \ni v \mapsto \partial_v \partial_\mu f(x, \mu)(v) \in \mathbb{R}^{d \times d}$, is locally bounded and is jointly continuous in $( x, \mu, v)$ at any point $( x, \mu, v)$ such that $v \in \text{Supp}(\mu)$.

\item[(iv)]
For any compact subset $\cK\subset \mathbb R^d \times \mathcal P_2(\mathbb R^d)$
\begin{align}\label{0331--1}
\sup_{(x,\mu)\in\cK}\int_{\R^d}|\partial_\mu f(x,\mu)(y)|^2\mu(\rmd y)+  \int_{\R^d}|\partial_y\partial_\mu f(x,\mu)(y)|^2\mu(\rmd y)<\infty.
\end{align}
\end{itemize}
Let $(\mathcal{X}_i, d_i)$ for $i = 1, 2$ be two metric spaces. 
In this paper, we consider the product space $\mathcal{X}_1 \times \mathcal{X}_2$ 
equipped with the metric $d = d_1 + d_2$.
Throughout this paper, let $C$ be a positive constant, which may depend on some 
parameters and change from line to line.
Meanwhile, we use $C_\alpha$ to emphasize that $C$ depends on the parameter $\alpha$.

\section{Main results}\label{mainR}
In this section, we present our main results. We begin by establishing a general framework for the existence of random attractors for the cocycle with
two parametric dynamical systems in Section \ref{subsec1}. 
We mention that the second driving system is a semi-dynamical system in our cases.
This framework is then applied to McKean-Vlasov SODEs under the distribution-dependent Lyapunov conditions in Section \ref{subsec2}. We subsequently consider McKean-Vlasov SPDEs, 
first for stochastic reaction-diffusion equations in Section \ref{subsec3} and then for the stochastic 2D Navier-Stokes equations in Section \ref{subsec4}.

\subsection{General results}\label{subsec1}
Let $(\Omega,\cF,\mathbb{P},\{\theta_{1,t}\}_{t\in\R})$ be a metric dynamical system; see Definition \ref{defMDS}.
And assume that  $\theta_2:\R_+\times\cP\rightarrow\cP$ is a {\it semi-dynamical system}, that is to say,
$\theta_{2,0}p:=\theta_{2}(0,p)=p$ and $\theta_{2,t+s}p=\theta_{2,t}\theta_{2,s}p$ for all $p\in\cP$ and $t,s\in\R_+$.  
Let $(\cX,d_{\cX})$ be a Polish space and $(\cP,d_{\cP})$ be a metric space.
Assume that the measurable map
\begin{align*}
\varphi:\R_+\times\Omega\times \cP\times\cX&\rightarrow\cX\\
(t,\omega,p,x)&\mapsto \varphi(t,\omega,p,x)=:\varphi(t,\omega,p)x
\end{align*}
satisfies the property of cocycle over $(\theta_{1,t},\theta_{2,t})_{t\geq0}$, i.e. 
for all $(\omega,p)\in\Omega\times\cP$,
\begin{align*}
 \varphi(0,\omega,p)=Id_{\cX},
\end{align*}
and for all $s,t\in\R_+$, $(\omega,p)\in\Omega\times\cP$,
\begin{align*}
\varphi(s+t,\omega,p)=\varphi(t,\theta_{1,s}\omega,\theta_{2,s}p)\circ\varphi(s,\omega,p),
\end{align*}
where $\circ$ means composition.
It is worth noting that $\varphi$ have two parametric dynamical systems $(\Omega,\cF,\mathbb{P},\{\theta_{1,t}\}_{t\in\R})$ and $(\cP, \{\theta_{2,t}\}_{t\geq0})$, 
and the second parametric dynamical system $\theta_{2}$ is a semi-dynamical system with time forward. 
The method presented in \cite{WangBX12} cannot be directly applied to consider its pullback attractor.

Set $\cY:=\cP\times\cX$. Define
\begin{align*}
\Phi:\quad
&
\R_+\times\Omega\times\cY\rightarrow\cY\\
&
(t,\omega,(p,x))\mapsto\Phi(t,\omega)(p,x)
:=\left(\theta_{2,t}p,\varphi(t,\omega,p)x\right).
\end{align*}
Then $\Phi$ is an RDS on $(\cY,d_{\cY})$ over $(\Omega,\cF,\mathbb{P},\{\theta_{1,t}\}_{t\in\R})$; we refer to Definition \ref{defRDS} for details. And we can consider the 
random attractor on the product space $\cY$, as defined in Definition \ref{DefRA}. 
In this paper, we let $\cD$ be a neighborhood closed collection of some families of
nonempty subsets of $\cX$,
and let $\sS$ be a inclusion closed collection of some nonempty subset of $\cP$; see Definitions \ref{de1103} and \ref{Duniverse} for details. 
These conditions are assumed hereafter without further mention.
Let $\sD$ be a collection of some families of nonempty subsets of $\cY$
which satisfies that if $D\in\sD$ then there exist $B_{\cP}\in\sS$ and $B\in\cD$ such that $D(\omega)\subset B_{\cP}\times B(\omega)$ for all $\omega\in\Omega$.

\begin{de}\label{dePAC}
We say that $\Phi$ is \emph{pullback $\sD$-asymptotically compact} 
if for any $\omega\in\Omega$, $B_{\cP}\times B\in\sD$, $t_n\to+\infty$ and $(p_n,x_n)\in B_{\cP}\times B(\theta_{1,-t_n}\omega)$, 
the sequence 
$$
\{\Phi(t_n,\theta_{1,-t_n}\omega,p_n,x_n)\}
=\{\big(\theta_{2,t_n}p_n,\varphi(t_n,\theta_{1,-t_n}\omega,p_n,x_n)\big)\}
$$ 
possesses a convergent subsequence.
\end{de}
\begin{prop}\label{prop0117}
Assume that $\{\theta_{2,t}\}_{t\geq0}$ is continuous over any $B_{\cP}\in\sS$,
and admits a compact absorbing set $B\in\sS$. 
Then $\{\theta_{2,t}\}_{t\geq0}$ admits a $\sS$-global attractor $\cP^*$.

Suppose further that $\cP$ is a Polish space and $\Phi:\R_+\times\Omega\times \cY\rightarrow\cY$ 
is a continuous RDS over $(\Omega,\cF,\mathbb{P},\{\theta_{1,t}\}_{t\in\R})$.
If $\Phi$ is $\sD$-asymptotically compact and there exists a 
closed $K\in\cD$ such that for any
$\omega\in\Omega$, $B_\cP\in\sS$ and $D\in\cD$, 
\begin{equation}\label{0916:02}
\lim_{t\rightarrow\infty}\sup_{p\in B_\cP}
\dist_{\cX}\left(\varphi(t,\theta_{1,-t}\omega,p)D(\theta_{1,-t}\omega),K(\omega)\right)=0.
\end{equation}
Then $\Phi$ admits a $\sD$-random attractor $\sA$, which is defined by
\begin{equation*}
\sA(\omega)
:=\overline{\bigcup_{D\in\sD}\Omega(D(\omega))}
=\Omega(\cP^*\times K(\omega)),
\end{equation*}
where 
\begin{equation*}
\Omega(D(\omega)):=\bigcap_{s\geq0}\overline{\bigcup_{t\geq s}\Phi(t,\theta_{1,-t}\omega)D(\theta_{1,-t}\omega)}.
\end{equation*}
\end{prop}

\begin{rem}
In the case where $\cP$ is not a Polish space, we consider the restriction of $\theta_2$ to $\cP^*$. Under this restriction, we assume that the limit in \eqref{0916:02}
is uniformly with respect to $p\in\cP^*$. Then
the above theorem remains valid for the mapping $\Phi: \R_+ \times \Omega \times \cP^*\times \cX \to \cP^* \times \cX$.
\end{rem}

Furthermore, we are still concerned with the random attractor  of 
$$
\varphi:\R_+\times\Omega\times\cP\times\cX\rightarrow\cX.
$$
To this end, we define the pullback with respect to $\theta_{2}$. 
Suppose that $\theta_{2}$ has a $\sS$-global attractor $\cP^*$ in $\cP$.
Then by \cite[Theorem 9.19]{KR2011}, we obtain that the backward extension $\Theta_2$ of $\theta_2$ on $\cP^*$,
defined by
\begin{equation*}
\Theta_{2,t}p:=\left\{q\in\cP^*:\theta_{2,|t|}q=p\right\},\quad
\forall (t,p)\in(-\infty,0]\times\cP^*,
\end{equation*}
is a set-valued semi-dynamical system on $\cP^*$ with time backward.

Motivated by the definition of pullback attractor of nonautonomous
semi-dynamical system and RDS (see e.g. \cite{KR2011, Arnold}), 
we define the random attractor of 
$\varphi:\R_+\times\Omega\times\cP^*\times\cX\rightarrow\cX$ over 
$(\Omega,\cF,\mathbb{P},\{\theta_{1,t}\}_{t\in\R})$ 
and $(\cP^*, \{\Theta_{2,t}\}_{t\leq0})$.

\begin{de}\label{deRPA}
Assume that $\varphi(t,\omega,\cdot,\cdot):\cP^*\times\cX\rightarrow \cX$ is continuous for all $t\in\R_+$ and $\omega\in\Omega$.
We say that $\cA_{\cX}\in\cD$ is a {\it $\cD$-random attractor} of $\varphi$ if
\begin{itemize}
\item[(i)] $\cA_{\cX}=\{A(\omega,p):(\omega,p)\in\Omega\times\cP^*\}$ is a family of nonempty and compact subset of $\cX$. 

\item[(ii)] $\omega\mapsto A(\omega,p)$ is measurable for any $p\in\cP^*$.

\item[(iii)] $\cA_{\cX}$ is {\it invariant} in the sense that for all $t\geq0$ and $p\in\cP^*$,
  \begin{equation}\label{1106:02}
  \bigcup_{q\in\Theta_{2,-t}p}\varphi(t,\theta_{1,-t}\omega,q)A(\theta_{1,-t}\omega,q)=A(\omega,p).
  \end{equation}
  
\item[(iv)] $\cA_{\cX}$ is $\cD$-{\it pullback attracting}, i.e. for any $(\omega,p)\in\Omega\times\cP^*$ 
  and $D\in\cD$,
  \begin{equation}\label{0117:01}
  \lim_{t\rightarrow\infty}\sup_{q\in\Theta_{2,-t}(p)}\dist_{\cX}
  \left(\varphi(t,\theta_{1,-t}\omega,q)D(\theta_{1,-t}\omega),
  A(\omega,p)\right)=0.
  \end{equation}
\end{itemize}
\end{de}

\begin{rem}
\begin{itemize}
  \item By Definition \ref{deRPA}, if it exists, the $\cD$-random attractor of $\varphi$ is unique in $\cD$. 
  \item The measurability of $\omega \mapsto A(\omega, p)$ for $p \in \cP^*$ is given in
  the sense of Definition \ref{de0916}.
\end{itemize}
\end{rem}

\begin{de}
$\varphi$ is said to be \emph{pullback $\cD$-asymptotically compact} if for any $\omega\in\Omega$, $D\in\cD$, $t_n\to+\infty$, $p_n\in \cP^*$, $q_n\in \Theta_{2,-t_n}p_n$ and $x_n\in D(\theta_{1,-t_n}\omega)$, the sequence $\varphi(t_n,\theta_{1,-t_n}\omega,q_n,x_n)$ possesses a convergent subsequence.
\end{de}
\begin{rem}\label{1016--1}
 The pullback $\cD$-asymptotically compact property of $\varphi$ implies the pullback $\cD\times \cP^*$-asymptotically compact property of $\Phi$ since $p_n\in\Theta_{2,-t_n}\theta_{2,t_n}p_n$.
\end{rem}

\begin{thm}\label{thm0330}
Assume that $\varphi(t,\omega,\cdot,\cdot):\cP^*\times\cX\rightarrow \cX$ is continuous for all $t\in\R_+$ and $\omega\in\Omega$ and is pullback $\cD$-asymptotically compact.
Suppose further that there exists a closed $K\in\cD$ such that for any $\omega\in\Omega$
and $B\in\cD$,
\begin{equation*}
\lim_{t\rightarrow\infty}\sup_{p\in \cP^*}
\dist_{\cX}\left(\varphi(t,\theta_{1,-t}\omega,p)B(\theta_{1,-t}\omega),K(\omega)\right)=0.
\end{equation*}
Then $\varphi$ admits a $\cD$-random attractor 
$\cA_{\cX}=\{\cA(\omega,p):(\omega,p)\in\Omega\times\cP^*\}$, where 
$$\cA(\omega,p):=\left\{x\in \cX:~(p,x)\in \sA(\omega)\right\}$$
and $\sA$ is the $\sD$-random attractor of $\Phi$.
\end{thm}

\begin{rem}
Let $\cX$ be a Hilbert space.
In Section \ref{subsec2}, we let $\cD$ be the tempered set, i.e. for any $D\in\cD$ and
$\varepsilon>0$,
$$
\lim_{t\rightarrow+\infty}{\rm e}^{-\varepsilon t}\|D(\theta_{-t}\omega)\|_{\cX}^2=0,
$$  
where $\|B\|_{\cX}:=\sup_{x\in B}\|x\|_{\cX}$ for any subset $B\subset \cX$.
And in Section \ref{RMVSPDE}, we will let $\cD_{\epsilon_0}$ be the weakly tempered set with a 
fixed rate $\varepsilon_0>0$, i.e. for any $D\in\cD_{\epsilon_0}$,
$$
\lim_{t\rightarrow+\infty}{\rm e}^{-\varepsilon_0 t}\|D(\theta_{-t}\omega)\|_{\cX}^2=0.
$$
\end{rem}

\subsection{Random attractors for McKean-Vlasov SODE}\label{subsec2}

Let $C_0(\R, \R^{m})$ denote the space of continuous functions $\phi: \R \to \R^{m}$ such that $\phi(0) = 0$, endowed with the topology of uniform convergence on compact subsets of $\R$. 
Denote by $\cB(C_0(\R, \R^{m}))$ the Borel $\sigma$-algebra on $C_0(\R, \R^{m})$, 
and by $\mathbb{P}_W$ the Wiener probability measure on $\cB(C_0(\R, \R^{m}))$.
Let
$$
(\Omega, \mathcal{F}, \mathbb{P},\{\theta_{t}\}_{t\in\R})
:=(C_0(\R, \R^{m}),\cB(C_0(\R, \R^{m})),\mathbb{P}_W,\{\theta_{t}\}_{t\in\R}),
$$
where for $t \in \R$,
\begin{align*}
\theta_{t}: \Omega \to \Omega,~
\omega \mapsto \theta_{t}\omega:= \omega(t + \cdot) - \omega(t).
\end{align*}
Consider the canonical Brownian motion $W_t:C(\R, \R^{m})\to \R^{m}, ~\omega\mapsto \omega(t)$.

In this section, 
we consider the following $d$-dimensional McKean-Vlasov SDE
\begin{align}\label{McKean-Vlasov-SDE}
&\rmd X_t=b(X_t,\sL_{X_t})\,\rmd t+\sigma(X_t)\circ\rmd W_t,
\end{align}
where $b:\R^d\times\cP(\R^d)\rightarrow\R^d$, 
$\sigma:\R^d\rightarrow\R^d\times\R^m$ and
$\circ$ refers to the Stratonovich symmetric integral.
We study the random attractor on the product space $\R^d\times\cP(\R^{d})$ by decomposing the above equation into the following system:
\begin{equation}\label{1111}
\left\{
\begin{aligned}
&\rmd X_{t}=b(X_t,\sL_{Y_t})\,\rmd t+\sigma(X_t)\circ\rmd W_t,~X_s=x\in\R^d,\\
&\rmd Y_{t}=b(Y_t,\sL_{Y_t})\,\rmd t+\sigma(Y_t)\circ\rmd W_t,~\sL_{Y_s}=\mu\in\cP_2(\R^d).
\end{aligned}
\right.
\end{equation}
Under the conditions of unique solvability and differentiability of $\sigma$, system \eqref{1111} can be transformed into the It\^o SDE
\begin{equation}\label{1111-2}
\left\{
\begin{aligned}
&\rmd X_{t}=\overline{b}(X_t,\sL_{Y_t})\,\rmd t+\sigma(X_t) \rmd W_t,~X_s=x,\\
&\rmd Y_{t}=\overline{b}(Y_t,\sL_{Y_t})\,\rmd t+\sigma(Y_t) \rmd W_t,~\sL_{Y_s}=\mu,
\end{aligned}
\right.
\end{equation}
where $\overline{b}(x,\mu):=b(x,\mu)+\frac{1}{2}\sum_{i=1}^m\partial_x\sigma_i(x)\sigma_i(x)$.
\cite{LM} suggests that the existence and uniqueness of solutions to equation \eqref{1111-2} are guaranteed under the following conditions: 
\begin{enumerate}[label=(\bf C\arabic*)]
  \item $\sigma$ belongs to $C^{\infty}(\R^d,\R^{d\times m})$, and its first derivative $\partial_x\sigma$ is bounded. 
  Moreover, $\sigma$ satisfies 
  $$
  [\sigma_i,\sigma_j]_k:=\sum_{l=1}^d\Big(\sigma_{li}(x)\frac{\partial\sigma_{kj}(x)}{\partial x_l}-\sigma_{lj}(x)\frac{\partial\sigma_{ki}(x)}{\partial x_l}\Big)=0
  $$
   for $1\leq i,j\leq m$ and $1\leq k\leq d$.
   
    \item 
    For any \(N\geq 1\), there exists a constant \(C_{N}\geq 0\) such that for any \(|x|,|y|\leq N\) and \(\operatorname{supp}\mu,\operatorname{supp}\nu\subset B(0,N)\),
    \[
     |\overline{b}(x,\mu)-\overline{b}(y,\nu)|\leq C_{N} \big{(}|x-y|+W_{2}(\mu,\nu)\big{)}.
    \]
    Here \(B(0,N)\) denotes the closed ball in \(\mathbb{R}^{d}\) centered at the origin with radius \(N\).

    \item For any bounded sequences \(\{x_{n},\mu_{n}\}\in\mathbb{R}^{d}\times\mathcal{P}_{2}(\mathbb{R}^{d})\) with \(x_{n}\to x\) and \(\mu_{n}\to\mu\) weakly in \(\mathcal{P}_{2}(\mathbb{R}^{d})\) as \(n\to\infty\), we have
    \[
    \lim_{n\to\infty}|b(x_{n},\mu_{n})-b(x,\mu)|=0.
    \]
    
    \item There exists a nonnegative function \(V\in \Xi\) satisfying that there exist $\alpha>0$ and $\mathcal{M}>0$ such that for all \((x,\mu)\in\mathbb{R}^{d}\times\mathcal{P}_{2}(\mathbb{R}^{d})\)
   \begin{align}\label{1009-4}
       &\cL V(x,\mu)  \leq  -\alpha V(x,\mu)+\mathcal{M}, 
   \end{align}
    where 
       \begin{align*}
       \cL V(x,\mu):=
       &
        \big\langle \partial_x V(x,\mu), \overline{b}(x,\mu) \big\rangle+\frac{1}{2} \sum_{i=1}^m\sigma_i(x)\partial^2_{xx}V(x,\mu)\sigma_i^T(x) \\
       &
        + \int_{\R^d}\big\langle \partial_{\mu} V(x,\mu)(y), \overline{b}(y,\mu) \big\rangle+\frac{1}{2} \sum_{i=1}^m\sigma_i(y)\partial_y\partial_{\mu}V(x,\mu)(y)\sigma_i^T(y)\,\mu(\rmd y).
       \end{align*}

    \item There exist $\ell>1$, $p\geq4\ell$ and $K>0$ such that for any $(x,\mu)\in \mathbb R^d \times \mathcal P_2(\mathbb R^d)$,
      \begin{align}\label{1126-1}
         |b(x,\mu)|^{2\ell}+|\sigma(x)|^{4\ell}+|x|^p \leq K(1+V(x,\mu)).
      \end{align}

    \item There exists \(L>0\) such that for any \(\mu,\nu\in\mathcal{P}_{2}(\mathbb{R}^{d}),\pi\in\mathcal{C}(\mu,\nu)\),
 \begin{align*}
    &\int_{\mathbb{R}^{d}\times\mathbb{R}^{d}}2\langle x-y,\overline{b}(x,\mu)-\overline{b}(y,\nu)\rangle\pi(\rmd x,\rmd y) \leq L\int_{\mathbb{R}^{d}\times\mathbb{R}^{d}}(1+\mu(|\cdot|^{2})+\nu(|\cdot|^{2}))|x-y|^{2}\pi(\rmd x,\rmd y).
\end{align*}
\end{enumerate}
The solution of \eqref{1111} is defined as $(X _{s,t}^{x,\mu}(\omega),P_{t-s}^*\mu)$,
where $\{P_t^*\}_{t\geq0}$ is the semigroup of transfer operators generated by 
the second equation in \eqref{1111}.

\begin{rem}\label{rem0929}
\begin{itemize}
  \item
  By adopting the method from \cite{IS}, \textbf{(C1)} enables us to establish the conjugacy $\overline{\cT}(\omega)=Id_{\cP_V(\R^d)}\times \cT(\omega)$ between system \eqref{1111} and a certain random differential equation. 
  See Proposition \ref{thm1} for more details.
  Consequently, we obtain the following RDS
\begin{align*}
\Phi:\R_+\times\Omega\times\cP_V(\R^d)\times \R^d&\to\cP_V(\R^d)\times \R^d\\
(t,\omega,\mu,x)&\mapsto (P_t^*\mu,X^{x,\mu}_{0,t}(\omega))=:(P_{t}^*\mu, X _{t}^{x,\mu}(\omega)).
\end{align*}
and the cocycle 
\begin{align*}
\varphi:\R_+\times\Omega\times  \cP_V(\R^d)\times \R^d&\to\R^d\\
(t,\omega,\mu,x)&\mapsto X _{t}^{x,\mu}(\omega)
\end{align*}
over $(\Omega,\cF,\mathbb{P},\{\theta_{1,t}\}_{t\in\R})$ and 
$(\cP_V(\R^d), \{P_t^*\}_{t\geq0})$.
Here 
$$
\cP_V(\R^d):=\{\mu\in\cP_2(\R^d):\int_{\R^d}V(x,\mu)\mu(\rmd x)<\infty\}.
$$

\item
The diffusion term $\sigma$ in our considered equation does not depend on the distribution variable. 
Otherwise, the conjugate mapping would become distribution-dependent, rendering the construction of the conjugate equation impossible.
\end{itemize}
\end{rem}

We need to introduce the following condition for the random differential equation \eqref{RDE1}. 
 For $(x,\mu,z)\in\R^d\times\cal{P}_2(\R^d)\times\R^m$, let
  \begin{align*}
  \widetilde{\cal{L}}^\eta V(x,\mu)(z):=
  &
  \big\langle \partial_x V(x,\mu), h^{\eta}(z,x,\mu) \big\rangle+ \int_{\R^d}\big\langle \partial_{\mu} V(x,\mu)(y), \overline{b}(y,\mu)\big\rangle\mu(\rmd y)\\
  & 
  +\frac{1}{2}\sum_{i=1}^m\int_{\R^d}\sigma_i(y)\partial_y\partial_{\mu}V(x,\mu)(y)\sigma_i^T(y)\,\mu(\rmd y),
  \end{align*}  
  where
  $$
  h^{\eta}(z,x,\mu):=(\partial_xu)^{-1}(z,x)\Big(b(u(z,x),\mu)+\eta\sum_{i=1}^{m}\sigma_i(u(z,x))z^i\Big).
  $$
  Here
  $u:\R^m\times\R^d\rightarrow\R^d$ is the unique solution to the following equation under \textbf{(C1)}
  \begin{equation*}
  \begin{cases}
   \frac{\p u}{\p z_i}(z,x)=\sigma_i(u(z,x)), \quad 1\leq i\leq m\\
   u(0,x)=x.
  \end{cases}
  \end{equation*}

 Let us revisit the concept of subexponential growth. A function $K:\R^m\to\R$ is said to be \emph{subexponentially growing} if there exists $c>0$ such that $z\mapsto K(z)e^{-c|z|}$ is bounded on $\R^m$.
\begin{itemize}
\item[\textbf{(C7)}] There exist $\eta>0$, two subexponentially growing functions $K:\R^{m}\rightarrow\R$ and $L:\R^{m}\rightarrow\R_+$ such that for any $(x,\mu,z)\in\R^d\times\cal{P}_2(\R^d)\times\R^m$,
  \begin{align}\label{0102-1}
  \widetilde{\cal{L}}^\eta V(x,\mu)(z)\leq K(z)V(x,\mu)+L(z)
  \end{align}
  and
  \begin{align}\label{0902-2}
  \int_{\R^m}K(z)\sL_{Z^I_0}(\rmd z)<0,
  \end{align}
  where 
  $
  Z_t^I=\int_{-\infty}^t\rme^{-\eta(t-s)}\rmd W_s.
  $
\end{itemize}

\begin{rem}
\begin{itemize}
\item 
If the coefficient $\sigma$ is linear, then $\widetilde{\mathcal{L}}^\eta V$ only depends on $\sigma$ rather than $u$, as detailed in \cite{IS}.
\item 
Although no restrictions on $\eta$ are required when establishing the conjugate relationship, a sufficiently large $\eta$ is necessary to guarantee the dissipative condition \eqref{0902-2}; see Lemma \ref{0106-3} for details. 
\end{itemize}
\end{rem}

We say that $B_{\cP}\subset\cP_V(\R^d)$ is {\it bounded} if there exists $R>0$ such that 
$\sup_{\mu\in B_{\cP}}\int_{\R^d}V(x,\mu)\mu(\rmd x)\leq R$.
Define
$$
\sS:=\left\{B_{\cP}\subset \cP_V(\R^d) \text{ is bounded}: B_{\cP} 
\text{ is closed under } \cW_2\right\}.
$$
Now we give our results about McKean-Vlasov SODEs.  

\begin{thm}\label{thmSODE}
Assume that  \textbf{(C4)} and \textbf{(C5)} hold. Then the system $P_t^*$ admits a $\sS$-global attractor $\cP^*$ in $(\cP_V(\R^d),\cW_2)$.

Suppose further that \textbf{(C1)} and \textbf{(C7)} hold. 
Then we have the following conclusions.
   \begin{itemize}
   \item[(i)] 
   $\Phi:\R_+\times\Omega\times\cP^*\times\R^d\to\cP^*\times\R^d$ admits a $\overline{\cT}\sD$-random attractor $\overline{\cT}\sA$, 
   where $\overline{\cT}\sD:=\{\{\overline{\cT}(\omega)D(\omega)\}_{\omega\in\Omega}:D\in\sD\}$ and $\sA$ is the random attractor of the conjugate system of $\Phi$.
   
   \item[(ii)]
   $\varphi:\R_+\times\Omega\times \cP^*\times \R^d\to\R^d$ admits a $\cT\cD$-random attractor $\cA:=\{\cA(\omega,\mu):(\omega,\mu)\in\Omega\times\cP^*\}$,
   where
   $\cT\cD:=\{\{\cT(\omega)D(\omega)\}_{\omega\in\Omega}:D\in\cD\}$ and $\cA(\omega,\mu):=\{x\in\R^d:(\mu,x)\in\overline{\cT}(\omega)\sA(\omega)\}$.
  \end{itemize}
\end{thm}

We apply the above results to the following example.
\begin{exam}\label{ex01}
Consider the following system with nonlinear multiplicative noise:
\begin{numcases}
{}
\rmd X_t=b(X_t,\sL_{Y_t})\,\rmd t+\sigma(X_t)\circ\rmd W_t,
\quad X_s=x\in\R,\label{0929:01}\\
\rmd Y_t=b(Y_t,\sL_{Y_t})\,\rmd t+\sigma(Y_t)\circ\rmd W_t,
\quad \sL_{Y_s}=\mu\in \cP_2(\R). \label{0929:02}
\end{numcases}
Here $W_t$ is a 1-dimensional Brownian motion, 
   $$b(x,\mu):=-x^3\Big(\int_{\R}y^2\mu(\rmd y)\Big)^5-10x,$$
   and
\begin{equation}
\sigma(x):=
\left\{
\begin{aligned}
\nonumber
&\frac{1}{2}x-\frac{1}{2}M,~&&x\in(-\infty,-M-\varepsilon),\\
&\overline{\sigma}_1(x),~&&x\in[- M-\varepsilon, -M+\varepsilon],\\
&x,~&&x\in(- M+\varepsilon, M-\varepsilon),\\
&\overline{\sigma}_2(x),~&&x\in[M-\varepsilon, M+\varepsilon],\\
&\frac{3}{2}x-\frac{1}{2}M,~&&x\in(M+\varepsilon,+\infty),\\
\end{aligned}
\right.
\end{equation}
where $0<\varepsilon<M$, $\overline{\sigma}_1$ and $\overline{\sigma}_2$ are two functions which make $\sigma:\R\to\R$ smooth and strictly increasing. 
By the proof in Section \ref{secex01}, we obtain that 
$\Phi:\R_+\times\Omega\times\cP^*\times\R\rightarrow\cP^*\times\R$
generated by \eqref{0929:01} and \eqref{0929:02} admits a random attractor, 
where $\cP^*$ is the $\sS$-global attractor of $(P_t^*,\cP_8(\R))$.
\end{exam}

\subsection{Random attractors for McKean-Vlasov SPDE}\label{RMVSPDE}

In this section, we consider the random attractor for McKean-Vlasov SPDE, including 
McKean-Vlasov stochastic reaction-diffusion equations and McKean-Vlasov stochastic 2D Navier-Stokes equations. 
In this paper, we address the well-posedness of McKean-Vlasov SPDE within a monotone framework. 
To this end, we introduce the following assumptions.

Let $(H,\langle\cdot,\cdot\rangle)$ be a separable Hilbert space with norm $\|\cdot\|$
and $(V,\|\cdot\|_V)$ be a reflexive Banach space such that $V\subset H$ compactly 
and densely. In particular, there exists a constant $\gamma>0$ such that
$\gamma\|v\|\leq \|v\|_V$ for all $v\in V$.
Denote by $V^*$ the dual space of $V$.
In this section, we always assume that there exist constants $\lambda_1,\lambda_2,C>0$, 
$\kappa\geq2$ and $\beta\geq0$ such that 
$F:H\times \cP(H)\rightarrow H$ satisfies
\begin{itemize}
\item[{\bf(H1)}]
For any $w\in V$, the map
$$
V\times\cP_2(H)\ni(u,\mu)\mapsto ~_{V^*}\langle F(u,\mu),w\rangle_V
$$
is continuous.
\item[{\bf(H2)}]
For any $u,v\in V$ and $\mu\in\cP_2(H)$, $\nu\in\cP_\kappa(H)$,
\begin{align*}
2_{V^*}\langle F(u,\mu)-F(v,\nu),u-v\rangle_V
\leq 
&
C(1+\rho(v)+\nu(\|\cdot\|^\kappa))\|u-v\|^2\\
&
+C(1+\nu(\|\cdot\|^\kappa))\cW_2(\mu,\nu)^2,
\end{align*}
where $\rho:V\rightarrow[0,\infty)$ is locally bounded measurable with
\begin{equation}\label{0416:01}
\rho(v)\leq C(1+\|v\|_V^2)(1+\|v\|_H^\beta),\quad \forall v\in V.
\end{equation} 
\item[{\bf(H3)}]
For all $u\in H$ and $\mu\in\cP_2(H)$,
$$
2\langle F(u,\mu),u\rangle
\leq \lambda_1\|u\|^2+\lambda_2\mu(\|\cdot\|^2)+C.
$$
\item[{\bf(H4)}]
For any $u\in V$, $\mu\in\cP_\kappa(H)$,
$$
\|F(u,\mu)\|_{V^*}^{2}
\leq C\left(\|u\|_V^2+\mu(\|\cdot\|^\kappa)+1\right)\left(1+\|u\|^\beta+\mu(\|\cdot\|^\kappa)\right).
$$
\end{itemize}

Conditions {\bf(H1)}-{\bf(H4)} ensure the existence of solutions to 
a class of McKean-Vlasov SPDEs.

\subsubsection{McKean-Vlasov stochastic reaction-diffusion equation}\label{subsec3}
Let $\Lambda\subset \R^d$ be open and bounded. 
Denote by $L^2:=L^2(\Lambda)$ and $H_0^1:=H_0^1(\Lambda)$ the classical 2-integral and Sobolev spaces 
respectively.
Let $\lambda_*$  be the first eigenvalue of $\Delta$  with the Dirichlet boundary condition.
Now we consider the following system for any $t\geq s$,
\begin{numcases}
{}
\rmd X_t=\left(\Delta X_t+F(X_t,\sL_{Y_t})\right)\rmd t+\rmd W_t, 
\quad X_s=x\in L^2,\label{0327:01-1}\\
\rmd Y_t=\left(\Delta Y_t+F(Y_t,\sL_{Y_t})\right)\rmd t+\rmd W_t,
\quad \sL_{Y_s}=\mu\in \cP_2(L^2), \label{0327:01-2}
\end{numcases}
where $W_t$ is $H_0^1$-valued Wiener process.
Set $V:=H_0^1$ and $H:=L^2$. Hence, $\gamma=\sqrt{\lambda_*+1}$.
To obtain the singleton set attractor for \eqref{0327:01-1} and \eqref{0327:01-2},
we introduce the following assumptions.
\begin{enumerate}
\item[{\bf(H2$'$)}] 
There exist constants $C_1,C_2\geq0$
such that for any $u,v\in V$ and $\mu,\nu\in\cP_2(H)$,
\begin{equation}\label{0327:02}
2_{V^*}\langle F(u,\mu)-F(v,\nu),u-v\rangle_V\leq C_1\cW_2(\mu,\nu)^2+C_2\|u-v\|^2.
\end{equation} 
\item[{\bf(H4$'$)}] 
There exists constant $C\geq0$
such that for any $u\in H$ and $\mu\in\cP_2(H)$,
\begin{equation}\label{0327:04}
\|F(u,\mu)\|^2\leq C\left(1+\|u\|_V^2+\mu(\|\cdot\|^2)\right).
\end{equation} 
\end{enumerate}

\begin{rem}
Consider the following Stokes drag force
\begin{equation}\label{0929:03}
F(u,\mu):=c_0\int_{H}\left(u-v\right)\mu(\rmd v),
\end{equation}
where $c_0>0$ is a positive constant.
It is straightforward that $F$ satisfies {\bf(H1)}, {\bf(H2$'$)}, {\bf(H3)} and {\bf(H4$'$)}.
\end{rem}

Suppose that {\bf(H1)}, {\bf(H2$'$)}, {\bf(H3)} and {\bf(H4$'$)}  hold. 
Then by adopting method from \cite[Theorem 3.1]{HHL24} or \cite[Theorem 1]{AD95},
one sees that for any $Y_s\in L^2(\Omega,\mP;H)$, there exists a unique solution 
$Y_t,t\geq s$ to \eqref{0327:01-2}.
Define $P_{s,t}^*\mu:=\sL_{Y_t}$ with $\sL_{Y_s}=\mu$ and $P_{t}^*\mu:=P_{0,t}^*\mu$.
We rewrite \eqref{0327:01-1} into the following form:
\begin{align*}
X_{s,t}^{x,\mu}=x+\int_s^tA(X_{s,r}^{x,\mu},P_{s,r}^*\mu)\rmd r 
+W_t-W_s,\quad t\geq s,
\end{align*}
where $A(u,\mu):=\Delta u+F(u,\mu)$.
Let
$$
Z_{s,t}^{x,\mu}(\omega):=X_{s,t}^{x,\mu}(\omega)-W_t(\omega),
$$
which satisfies that 
\begin{align}\label{0916:01}
Z_{s,t}^{x,\mu}(\omega)=x-W_s(\omega)+\int_s^tA(Z_{s,r}^{x,\mu}(\omega)
+W_r(\omega),P_{s,r}^*\mu)\rmd r.
\end{align}
Define 
\begin{align*}
\varphi:\R_+\times \Omega\times \cP_2(H)\times H
& 
\longrightarrow H\\
(t,\omega,\mu,x)
& 
\longmapsto 
Z_{0,t}^{x,\mu}(\omega)+W_t(\omega)=:\varphi(t,\omega,\mu)x
\end{align*}
and
\begin{align*}
\Phi:\R_+\times \Omega\times \cP_2(H)\times H
&
\longrightarrow \cP_2(H)\times H\\
(t,\omega,(\mu,x))
&
\longmapsto \left(P_t^*\mu,\varphi(t,\omega,\mu)x\right)=:
\Phi(t,\omega)(\mu,x).
\end{align*}
In view of Lemma \ref{lem0916} and the weak uniqueness of solutions to \eqref{0327:01-2}, we have for any $\omega\in\Omega$, $x\in H$ and $t,s>0$, 
$$
\varphi(t+s,\omega,\mu)x=
\varphi(t,\theta_s\omega,P_s^*\mu)\varphi(s,\omega,\mu)x.
$$

Let $\sS$ be the collection of bounded and closed subset of $\cP_2(H)$
and $\cD_{\lambda_0}$ be the weakly tempered set.
\begin{thm}\label{thmRDeq}
Assume that {\bf(H1)}, {\bf(H2$'$)}, {\bf(H3)} and {\bf(H4$'$)}  hold.
Suppose further that 
$C_1+C_2<2\lambda_*$. Then we have the following conclusions.
\begin{itemize}
\item[(i)] 
$\left(\cP_2(H),\{P_t^*\}_{t\geq0}\right)$ is a continuous semi-dynamical system, 
which admits a $\sS$-global attractor
$$\cP^*:=\{\mu_\infty\}.$$
\item[(ii)] 
$\Phi:\R_+\times \Omega\times \cP_2(H)\times H
\rightarrow \cP_2(H)\times H$ is a continuous RDS.
And there exists $\lambda_0<2\lambda_*-C_1-C_2$ such that $\Phi$
admits a $\sD_{\lambda_0}$-random attractor 
$
\sA(\omega):=(\mu_\infty,\xi(\omega)),
$
where $\xi(\omega):=\xi_0(\omega)$. Here $\xi_t$ satisfies that 
for any $x\in H$ and $\mu\in\cP_2(H)$,
$$
\lim_{s\rightarrow-\infty}X_{s,t}^{x,\mu}=\xi_t,\quad \mP-a.s.
$$
\item[(iii)] 
$\varphi:\R_+\times\Omega\times \{\mu_\infty\}\times H\rightarrow H$ 
admits a $\cD_{\lambda_0}$-random attractor $\cA(\omega):=\{\xi(\omega)\}$.
In particular, 
$$
\mP\circ[\xi]^{-1}=\mu_\infty.
$$
\end{itemize} 
\end{thm}
\begin{rem}
Under the conditions of Theorem \ref{thmRDeq},  
$\left(\cP_2(H),\{P_t^*\}_{t\geq0}\right)$ admits a $\sS$-global attractor consisting of a single point $\mu_\infty$.
Define 
$$\varphi_{\mu_\infty}:\R_+\times\Omega\times H\rightarrow H,\quad(t,\omega,x)\mapsto \varphi_{\mu_\infty}(t,\omega)x:=\varphi(t,\omega,\mu_\infty)x.$$
It is straightforward that $\varphi_{\mu_\infty}$ is an RDS over $(\Omega,\cF,\mP,\{\theta_t\}_{t\in\R})$ with the random attractor $\{\xi_0\}$, 
which is generated by
$$
\rmd X_t=\left(\Delta X_t+F(X_t,\mu_\infty)\right)\rmd t+\rmd W_t.
$$
\end{rem}

\subsubsection{McKean-Vlasov stochastic 2D Navier-Stokes equation}\label{subsec4}

Let
$$
H:=\left\{v\in L^2(\mT^2;\R^2):\nabla \cdot v=0~{\text{in}}~\mT^2,~
\int_{\mT^2}v(x)\rmd x=0\right\}
$$
endowed with the $L^2$-norm $\|\cdot\|$ and the standard scalar product.
Define 
$$
V:=H^1(\mT^2;\R^2)\cap H
$$
with the norm $\|v\|_V^2:=\int_{\mT^2} |\nabla v|^2\rmd x$. 
Define the Helmholz-Leray projection $P_H$ as follows:
$$
P_H:L^2(\mT^2;\R^2)\rightarrow H.
$$
Then set
$$
A: H^2(\mT^2;\R^2)\cap V\rightarrow H,\quad u\mapsto Au:=\nu_c P_H\Delta u
$$
and 
$$
B:V\times V\rightarrow V^*,\quad (u,v)\mapsto B(u,v):=-P_H\left[(u\cdot\nabla)v\right],
$$
where $\nu_c$ is the viscosity constant.

Now we consider the following system for $t\geq s$:
\begin{numcases}
{}
\rmd X_t^{x,\mu}=\left(AX_t^{x,\mu}+B(X_t^{x,\mu},X_t^{x,\mu})+F(X_t^{x,\mu},\sL_{Y_t})\right)\rmd t
+\sum_{i=1}^{d}\phi_i\rmd W_i(t),\quad X_{s}^x=x\in H,\label{0331:03}\\
\rmd Y_t=\left(AY_t+B(Y_t,Y_t)+F(Y_t,\sL_{Y_t})\right)\rmd t
+\sum_{i=1}^{d}\phi_i\rmd W_i(t),\quad \sL_{Y_s}=\mu\in\cP_{p}(H),
\label{0331:01}
\end{numcases}
where $p>4\vee\kappa$, $\phi_1,...,\phi_d\in \cD(A)$, $W_i,1\leq i\leq d$ are independent two-sided Brownian motion
and $F$ satisfies {\bf(H1)}--{\bf(H4)} with $\beta=2$.
Here $\kappa$ is as in {\bf(H2)}.

In order to obtain the uniqueness, we introduce the following condition {\bf(H2$''$)}.
Let $C([0,T];H)$ be the space of all continuous function from $[0,T]$ to $H$
equipped with the uniform norm
$$
\|\varphi\|_T:=\sup_{t\in[0,T]}\|\varphi(t)\|.
$$
Set
$$
\cP_{p,T}:=\left\{\mu\in\cP(C([0,T];H)):\int_{C([0,T];H)}\|\varphi\|_T^p
\mu(\rmd \varphi)<\infty\right\}
$$
and for any $R>0$,
$$
\tau_R^u:=\inf\left\{t\in[0,T]:\|u(t)\|+\int_0^t\|u(s)\|_V^2\rmd s\geq R\right\}.
$$

Define $\xi^t:[0,T]\to H$ by 
$$\xi^t(s)=\xi(t\wedge s),~s\in[0,T],$$
and $\pi_t(\xi):=\xi^t$. Then the marginal distribution before $t$ of a probability measure $\mu\in\cP(C[0,T],H)$ is defined by 
$$\mu^t:=\mu\circ\pi_t^{-1}.$$
For any $\mu,\nu\in\cP_{p,T}$, we recall the ``local" Wasserstein metric 
$$
\cW_{2,T,R,H}(\mu,\nu)=\inf_{\pi\in C(\mu,\nu)}\left(\int_{C([0,T],H)\times C([0,T],H)}\|\xi_{T\wedge\tau^{\xi}_R\wedge\tau^{\eta}_R}-\eta_{T\wedge\tau^{\xi}_R\wedge\tau^{\eta}_R}\|^2_T\pi(\rmd \xi,\rmd \eta)\right)^{1/2}.
$$

Suppose further that the following condition {\bf(H2$''$)} holds:
\begin{enumerate}
  \item[{\bf(H2$''$)}]
There exists $C>0$ such that for any $R>0$, $u,v\in C([0,T];H)\cap L^2([0,T];V)$,
$\mu,\nu\in\cP_{\kappa,T}(H)$ and $t\in[0,T\wedge\tau_R^u\wedge\tau_R^v]$,
\begin{align*}
&2_{V^*}\langle F(u(t),\mu(t))-F(v(t),\nu(t)),u(t)-v(t)\rangle_V\\
&\leq \phi(t)\left(\|u(t)-v(t)\|^2+\cW_{2,T,R,H}(\mu^t,\nu^t)^2\right),
\end{align*}
where
$$
\phi(t):=C+\rho_1(u(t))+\rho_2(v(t))+C\mu(t)(\|\cdot\|^\kappa)
+C\nu(t)(\|\cdot\|^\kappa),
$$
$\rho_1,\rho_2$ are measurable functions and locally bounded in $V$, which satisfy \eqref{0416:01}.
\end{enumerate}
Then for any $t\geq s$, $Y_s\in L^p(\Omega,\mP;H)$ with $p>4\vee\kappa$,
there exists a unique solution $Y_t,t\geq s$ to \eqref{0331:01}; see e.g. \cite[Theorem 3.1]{HHL24}.
And it follows from Lemma \ref{lem0418} that there exists 
a unique solution $X_{s,t}^{x,\mu}$ of \eqref{0331:03}.

\begin{enumerate}
\item[{\bf(H3$'$)}]
Assume that
$F:H\times\cP_2(H)\rightarrow H$ and
there exist $\tilde{\lambda}_1< 1$, 
$\tilde{\lambda}_2,\tilde{\lambda}_3>0$ and $C>0$ 
such that for all $u\in\cD(A)$ and $\nu\in\cP_2(H)$,
\begin{equation}\label{0411:02}
\langle -Au,F(u,\nu)\rangle \leq \tilde{\lambda}_1\|Au\|^2+\tilde{\lambda}_2\|u\|^2+\tilde{\lambda}_3\nu(\|\cdot\|^2)+C.
\end{equation}  
\end{enumerate}

\begin{rem}
\begin{itemize}
\item
Assume that $h:\R\rightarrow\R$ is bounded and Lipschitz continuous and $f$ is defined by
\begin{equation*}
f(v):=
\begin{cases}
\|v\|, & \mbox{if } \|v\|\leq N \\
N, & \mbox{if } \|v\|>N,
\end{cases}
\end{equation*}
where $N>0$ is a given constant.
If 
\begin{equation}\label{0604:01}
F(u,\mu):=h\left(\int_H f(v)\mu(\rmd v)\right)u,
\end{equation}
then {\bf(H2$''$)} holds; see e.g. Theorem 5.12 and Remark 6.2 in \cite{HHL24}.
Note that $F$ given in \eqref{0604:01} also satisfies {\bf(H3$'$)}. 
\item 
Note that $F$ given by \eqref{0929:03} also satisfies {\bf(H2$''$)}
and {\bf(H3$'$)}.
\end{itemize}
\end{rem}

Let $\eta>0$ be fixed. For each $i=1,...,d$, 
consider the following one-dimensional SDE
\begin{equation*}
\rmd z_i=-\eta z_i\rmd t+\rmd W_i(t).
\end{equation*}
Then there exists a unique stationary solution
\begin{equation*}
z_i^I(t)=\int_{-\infty}^t{\rm e}^{-\eta(t-s)}\rmd W_i(s).
\end{equation*}
Defining $z_t^I=\sum_{i=1}^{d}\phi_iz_i(t)$, then we have
\begin{equation*}
\rmd z_t^I=-\eta z_t^I\rmd t+\sum_{i=1}^{d}\phi_i\rmd W_i(t).
\end{equation*}

Let $Z_{s,t}^{x,\mu}=X_{s,t}^{x,\mu}-z_t^I$. Then 
\begin{align}\label{0409:01}
Z_{s,t}^{x,\mu}(\omega)=x-z_s^I(\omega)+\int_s^t\Big(
&
AZ_{s,r}^{x,\mu}(\omega)
+B(Z_{s,r}^{x,\mu}(\omega)+z_r^I(\omega),Z_{s,r}^{x,\mu}(\omega)+z_r^I(\omega))\nonumber\\
&
+F(Z_{s,r}^{x,\mu}(\omega)+z_r^I(\omega),\sL_{Y_r})
+\eta z_r^I(\omega)+Az_r(\omega)\Big)\rmd r,
\end{align}
where $Y_r,r\geq s$ is the solution to \eqref{0331:01} with initial distribution $\sL_{Y_s}=\mu$. 
Define 
$$
\varphi(t,\omega,\mu)x:=Z_{0,t}^{x,\mu}(\omega)+z_t^I(\omega)
$$
and
$$
\Phi(t,\omega)(\mu,x):=\left(P_{t}^*\mu,\varphi(t,\omega,\mu)x\right).
$$

Let $\sS$ be the collection of $B_{\cP}\subset\cP_p(H)$ satisfying that 
$B_{\cP}$ is closed under $\cW_2$ and there exists $R>0$ such that
$$
\sup_{\mu\in B_{\cP}}\int_{H}\|x\|^p\mu(\rmd x)\leq R.
$$
We now present the last main theorem in this paper.
\begin{thm}\label{thmNSeq2}
Suppose that {\bf(H1)}, {\bf(H2)}, {\bf(H2$''$)}, {\bf(H3)}, {\bf(H3$'$)} and {\bf(H4)}  hold with $\beta=2$.
Then we have the following conclusions.
\begin{itemize}
\item[(i)]
$(\cP_p(H),\{P_t^*\}_{t\geq0})$ is continuous over any $B_{\cP}\in\sS$ and
admits a $\sS$-global attractor $\cP^*$.
\item[(ii)]
$\Phi:\R_+\times \Omega\times \cP^*\times H
\rightarrow \cP^*\times H$ is a continuous RDS and
there exists $\lambda_0>0$ such that $\Phi$ admits a $\sD_{\lambda_0}$-random attractor $\sA\in\sD_{\lambda_0}$. 
\item[(iii)] 
$\varphi:\R_+\times\Omega\times \cP^*\times H\rightarrow H$ 
admits a $\cD_{\lambda_0}$-random attractor $\cA\in\cD_{\lambda_0}$, where
$$
\cA(\omega,\mu):=\left\{x\in H:~(\mu,x)\in\sA(\omega)\right\}.
$$
\end{itemize}
\end{thm}

\section{Proof of Proposition \ref{prop0117} and Theorem \ref{thm0330}}\label{secPGR}

In this section, we give the detailed proofs of Proposition \ref{prop0117}
and Theorem \ref{thm0330}.
\begin{proof}[Proof of Proposition \ref{prop0117}]
It follows from Lemma \ref{EGA} that there exists a $\sS$-global attractor $\cP^*$.
Since $\cD$ is neighborhood closed, there exists a $\epsilon_0>0$ such that 
$K_{\epsilon_0}\in\cD$, where
$$
K_{\epsilon_0}(\omega):=\left\{x\in\cX:\dist(x,K(\omega))\leq\epsilon_0\right\}.
$$
Since $B$ is absorbing,
for any $B_{\cP}\in\sS$, there exists $T_1:=T_1(B_{\cP})>0$ such that
\begin{equation}\label{1101:01}
\theta_{1,t}B_{\cP}\subset B,\quad \forall t\geq T_1.
\end{equation}
By \eqref{0916:02}, one sees that for any $B_{\cP}\in\sS$ and $D\in\cD$
there exists $T_2:=T_2(\omega,B_{\cP},D)>0$ such that for any $t\geq T_2$,
\begin{equation*}
\sup_{p\in B_{\cP}}\dist_{\cX}\left(\varphi(t,\theta_{1,-t}\omega,p)D(\theta_{1,-t}\omega),
K(\omega)\right)<\epsilon_0,
\end{equation*}
which implies that
\begin{equation}\label{1101:02}
\varphi(t,\theta_{1,-t}\omega,B_{\cP})D(\theta_{1,-t}\omega)\subset K_{\epsilon_0}(\omega),
\quad t\geq T_2.
\end{equation}
Let $T:=\max\{T_1,T_2\}$. Combining \eqref{1101:01} and \eqref{1101:02}, we obtain that
for any $t\geq T$,
$$
\Phi(t,\theta_{1,-t}\omega,B_{\cP}\times D(\theta_{1,-t}\omega))\subset
B\times K_{\epsilon_0}(\omega).
$$
Hence, Lemma \ref{lemAR} 
implies that $\Phi$ admits a $\sD$-random attractor, which is of the form as mentioned above.
\end{proof}

\begin{proof}[Proof of Theorem \ref{thm0330}]
Remark \ref{1016--1} implies that $\Phi$ is pullback $\cD\times \cP^*$-asymptotically compact. 
Hence, Proposition \ref{prop0117} is valid for $\Phi: \R_+ \times \Omega \times \cP^*\times \cX \to \cP^* \times \cX$.
Now we divide the proof into three steps. 
In step 1, we prove that $\cA(\omega,p)$ is nonempty, compact and measurable.
In step 2, we show that $\cA(\omega,p)$ is invariant in the sense of \eqref{1106:02}.
Finally, we prove that $\cA(\omega,p)$ is pullback attracting in step 3.

{\bf Step 1.}
For any $p\in\cP^*$ and $\omega\in\Omega$, define 
$$
\widetilde{\cA}(\omega,p):=\bigcap_{\tau\geq0}\overline{\bigcup_{t\geq\tau}\bigcup_{q\in\Theta_{2,-t}p}
\varphi(t,\theta_{1,-t}\omega,q,K(\theta_{1,-t}\omega))}.
$$
Note that $x\in\widetilde{\cA}(\omega,p)$ if and only if
there exist $t_n\rightarrow\infty$, $q_n\in\Theta_{2,-t_n}p$ and $x_n\in K(\theta_{1,-t_n}\omega)$ such that
$$
\varphi(t_n,\theta_{1,-t_n}\omega,q_n)x_n\rightarrow x.
$$
For any $t_n\rightarrow\infty$, $q_n\in\Theta_{2,-t_n}p$ and $x_n\in K(\theta_{1,-t_n}\omega)$, by the pullback $\cD$-asymptotically compact property of  $\varphi$,
there exists a subsequence such that
$$
\lim_{k\rightarrow\infty}\varphi(t_{n_k},\theta_{1,-t_{n_k}}\omega,q_{n_k})x_{n_k}=\tilde{x},
$$
which implies that $\widetilde{\cA}(\omega,p)$ is nonempty.
Meanwhile,
\begin{align*}
\Phi(t_{n_k},\theta_{1,-t_{n_k}}\omega,(q_{n_k},x_{n_k}))
&
=\left(\theta_{2,t_{n_k}}q_{n_k},\varphi(t_{n_k},\theta_{1,-t_{n_k}}\omega,q_{n_k},x_{n_k})\right)\\
&
=\left(p,\varphi(t_{n_k},\theta_{1,-t_{n_k}}\omega,q_{n_k},x_{n_k})\right)\rightarrow (p,\tilde{x}),
\end{align*}
which implies that $(p,\tilde{x})\in\sA(\omega)$ with $\tilde{x}\in\cA(\omega,p)$. 
Hence $\widetilde{\cA}(\omega,p)\subset \cA(\omega,p)$.

Let 
$$
\Pi_1:\cP^*\times\cX\rightarrow \cP^*,\quad
(p,x)\mapsto p
$$
be the projection mapping.
Then
\begin{equation}\label{0827:01}
\{p\}\times\cA(\omega,p)=\Pi_1^{-1}(p)\cap\sA(\omega),
\end{equation} 
which implies that $\cA(\omega,p)\subset \cX$ is compact.
Fix $p\in\cP^*$.
For any closed subset $C\subset \cX$, by \eqref{0827:01} one sees that
\begin{align*}
\{\omega:\cA(\omega,p)\cap C\neq\emptyset\}
&
=\{\omega:\{p\}\times\cA(\omega,p)\cap \{p\}\times C\neq\emptyset\}\\
&
=\{\omega:\Pi_1^{-1}(p)\cap\sA(\omega)\cap \{p\}\times C\neq\emptyset\}.
\end{align*} 
Since $\Pi_1^{-1}(p),\{p\}\times C\subset\cY$ are closed, by Lemma \ref{lem0827}
we obtain that 
$$
\{\omega:\cA(\omega,p)\cap C\neq\emptyset\}
=\{\omega:\Pi_1^{-1}(p)\cap\sA(\omega)\cap \{p\}\times C\neq\emptyset\}
$$ is measurable.
Hence, by the arbitrariness of $C$ and Lemma \ref{lem0827} again, one sees that
$\omega\mapsto \cA(\omega,p)$ is measurable.

{\bf Step 2.}
Note that for any $\omega\in\Omega$, $p\in\cP^*$ and $t\geq0$,
\begin{align*}
\left(\theta_{2,t}p,\varphi(t,\omega,p,\cA(\omega,p))\right)
=\Phi(t,\omega,\{p\}\times\cA(\omega,p))
\subset\sA(\theta_{1,t}\omega),
\end{align*}
which yields that
\begin{equation*}
\varphi(t,\omega,p,\cA(\omega,p))\subset\cA(\theta_{1,t}\omega,
\theta_{2,t}p).
\end{equation*}
Hence, for any $t\geq0$ and $p\in\cP^*$,
\begin{equation}\label{0827:02}
\bigcup_{q\in\Theta_{2,-t}p}\varphi(t,\theta_{1,-t}\omega,q,\cA(\theta_{1,-t}\omega,q))\subset \cA(\omega,p).
\end{equation}

On the other hand, the invariance of $\sA(\omega)$ implies that, for any $p\in\cP^*$ and $x\in\cA(\omega,p)$, i.e.
$(p,x)\in\sA(\omega)$,
there exists $(\widetilde{p},\widetilde{x})\in\sA(\theta_{1,-t}\omega)$ 
with $\widetilde{x}\in\cA(\theta_{1,-t}\omega,\widetilde{p})$ such that
\begin{align*}
(p,x)
=\Phi(t,\theta_{1,-t}\omega,(\widetilde{p},\widetilde{x}))
=\left(\theta_{2,t}\widetilde{p},\varphi(t,\theta_{1,-t}\omega,\widetilde{p},\widetilde{x})\right),
\end{align*}
which implies that $\widetilde{p}\in\Theta_{2,-t}p$ and 
$$
x=\varphi(t,\theta_{1,-t}\omega,\widetilde{p},\widetilde{x})
\in\varphi(t,\theta_{1,-t}\omega,\widetilde{p},\cA(\theta_{1,-t}\omega,\widetilde{p})).
$$
Hence, for any $t\geq0$ and $p\in\cP^*$,
$$
\cA(\omega,p)\subset\bigcup_{q\in\Theta_{2,-t}p}
\varphi(t,\theta_{1,-t}\omega,q,\cA(\theta_{1,-t}\omega,q)),
$$ 
which along with \eqref{0827:02} implies that $\cA_{\cX}$ is invariant in the sense of \eqref{1106:02}.

{\bf Step 3.}
It suffices to show that $\cA(\omega,p)$ is pullback attracting $K(\omega)$ in the sense of \eqref{0117:01}.
Assume that there exist $\eps_0>0$ and sequence $t_n\rightarrow\infty$, $q_n\in\Theta_{2,-t_n}p$ and $x_n\in K(\theta_{1,-t_n}\omega)$
such that
\begin{equation*}
\dist_{\cX} \left(\varphi(t_n,\theta_{1,-t_n}\omega,q_n,x_n),\cA(\omega,p)\right)\geq\eps_0.
\end{equation*}
The pullback $\cD$-asymptotically compact property of  $\varphi$ implies that 
$\varphi(t_n,\theta_{1,-t_n}\omega,q_n,x_n)$ possesses a convergent subsequence. The limit point is denoted by $a\in \widetilde{\cA}(\omega,p)$ satisfying
\begin{equation*}
\dist_{\cX} \left(a,\cA(\omega,p)\right)\geq\eps_0.
\end{equation*}
This contradicts $\widetilde{\cA}(\omega,p)\subset \cA(\omega,p)$.
Therefore, $\cA(\omega,p)$ is pullback attracting $K(\omega)$ and the proof is complete.
\end{proof}

\section{Proof of main results for McKean-Vlasov SODE}\label{secPODE}

This section provides the proofs for the main results given in Section \ref{subsec2}.

\subsection{Proof of Theorem \ref{thmSODE}}

As \textbf{(C4)} holds, it follows from Proposition \ref{20250526-1} and \eqref{1009-4} that 
  \begin{align*} 
   \mE V(X _{t}^{\xi,\mu},P_{t}^*\mu)
   &
   =\mE V(\xi,\mu)+ \mE \int_0^t\cL V(X _{r}^{\xi,\mu},P_{r}^*\mu)\rmd r\\
   &
   \leq \mE V(\xi,\mu)+\cal{M}t+\int_{0}^t -\alpha  \mE V(X _{r}^{\xi,\mu},P_{r}^*\mu)\,\rmd r,
  \end{align*}
  where $\mu=\sL_{\xi}\in\cal{P}_V(\R^d)$.
  It should be noted that $\sL_{X _{t}^{\xi,\mu}}=P_{t}^*\mu$.
  The Gronwall inequality implies that 
  \begin{align}\label{1010-1}
   \mE V(X _{t}^{\xi,\mu},P_{t}^*\mu )
   &
   \leq e^{-\alpha t}\mE V(\xi,\mu)+\cal{M}\int_0^te^{-\alpha(t-r)}\rmd r\nonumber\\
   &
   =e^{-\alpha t}\mE V(\xi,\mu)+\frac{\cal{M}}{\alpha}\Big(1-e^{-\alpha t}\Big).
  \end{align} 
  This yields that $P_t^*\mu\in\cal{P}_V(\R^d)$ as $\mu\in\cal{P}_V(\R^d)$.

For any $t\in\R$, almost all $\omega\in\Omega$,
define
\begin{equation}\label{1107:01}
\cT_t(\omega):=u(Z_t^I(\omega),\cdot),
\end{equation}
which is a diffeomorphism of $\R^d$ and is stationary; i.e. $\cT_t=\cT_0\circ\theta_t$. 
See \cite[Section 1]{IS} for more details. 
Define $\cT(\omega):=\cT_0(\omega)$.
 Let $\eta>0$ and $\mu\in\cal{P}_V(\R^d)$. 
 Consider the following random differential equation
  \begin{align}\label{RDE1}
   \rmd Y_t=g(\theta_t\cdot,Y_t,P^*_t\mu)\,\rmd t,~Y_0=x\in\R^d,
  \end{align}
  where for almost all $\omega\in\Omega$,
  \begin{align}\label{g}
  g(\omega,\cdot,\mu):\R^d&\rightarrow\R^d\\
   x&\mapsto\big(\partial_x\cT\big)^{-1}(\omega,x)\Big(b(\cT(\omega)x,\mu)+\eta\sigma(\cT(\omega)x)Z^I_0(\omega)\Big).\nonumber
  \end{align}
If the existence and uniqueness of the solution to the above equation can be ensured, 
denoted as $\chi_t^\mu(x)$, 
then for almost all $\omega\in\Omega$ and $s, t\geq0$, it holds that
 \begin{align}\label{RDEflow}
 \chi_{s+t}^\mu(\omega,x)=\chi_t^{P^*_s\mu}(\theta_s\omega,\chi_s^\mu(\omega,x)),~(\mu, x)\in\cal{P}_V(\R^d)\times\R^d. 
 \end{align}
 This makes us obtain an RDS
  \begin{align*}
  \overline{\chi}:\R_+\times\Omega \times\cal{P}_V(\R^d)\times\R^d&\rightarrow\cal{P}_V(\R^d)\times\R^d\\
  (t,\omega,\mu,x)&\mapsto(P_{t}^*\mu,\chi_t^\mu(\omega,x)).
  \end{align*}

Now we prove that there is a conjugate relationship between $\overline{\chi}$ and $\Phi$.
\begin{prop}\label{thm1}
 Assume that the strong solution to the first equation of \eqref{1111-2} 
 and the weak solution to the second equation of \eqref{1111-2} are uniquely existing, 
 and let $\chi_t^\mu(x)$ be the solution to equation \eqref{RDE1}.
  Then for any $t\geq0$, almost all $\omega$, it holds that
  $$
  \overline{\cT}(\theta_t\omega)\bigr(\overline{\chi}_t\big(\omega,\overline{\cT}^{-1}(\omega)(\mu, x)\big)\bigr)
  =\Phi(t,\omega,\mu, x),~(\mu, x)\in\cal{P}_V(\R^d)\times\R^d,
  $$
  where $\overline{\chi}_t(\omega,\mu,x)=(P_{t}^*\mu,\chi_t^\mu(\omega,x))$ 
  and $\overline{\cT}(\omega)= Id_{\cal{P}_V(\R^d)}\times \cT(\omega)$.
\end{prop}
\begin{proof}
The It\^o formula implies that
\begin{align*}
&
  u(Z^I_t,\chi_t^{\mu}(x))-u(Z^I_0,x)\\
  &=
  \int_0^t\frac{\partial u}{\partial z}(Z^I_r,\chi_r^{\mu}(x))\circ\rmd Z^I_r
  +\int_0^t\frac{\partial u}{\partial x}(Z^I_r,\chi_r^{\mu}(x))\,\rmd \chi_r^{\mu}(x)\\
  &=
  \int_0^t\sigma\big(u(Z^I_r,\chi_r^{\mu}(x))\big)\circ\rmd Z^I_r
  +\int_0^t\partial_x \cT (\theta_r\cdot,\chi_r^{\mu}(x))\,\rmd \chi_r^{\mu}(x)\\
  &=
  \int_0^t\sigma\big( \cT (\theta_r\cdot)\chi_r^{\mu}(x)\big)\circ\rmd Z^I_r
  +\int_0^tb( \cT (\theta_r\cdot)\chi_r^{\mu}(x),P_{r}^*\mu)+\eta\sigma( \cT (\theta_r\cdot)\chi_r^{\mu}(x))Z^I_r\,\rmd r\\
  &=
  \int_0^t\sigma\big( \cT (\theta_r\cdot)\chi_r^{\mu}(x)\big)\circ\rmd W_r
  +\int_0^tb( \cT (\theta_r\cdot)\chi_r^{\mu}(x),P_{r}^*\mu)\,\rmd r.
\end{align*}

The uniqueness of the solution to \eqref{1111-2} and the definition \eqref{1107:01} of $\cT$ indicate that for almost all $\omega\in\Omega$ and any $t\geq0$,
  $$
  \cT(\theta_t\omega)\chi_t^{\mu}(\omega,x)=\varphi(t,\omega,\mu,\cT(\omega)x).
  $$
  Then
  \begin{align}\label{1004--1}
 \Big(P_{t}^*\mu,\cT(\theta_t\omega)\chi_t^{\mu}\left(\omega,\cT^{-1}(\omega)x\right)\Big)
 =\Big(P_{t}^*\mu, \varphi(t,\omega,\mu,x) \Big).
  \end{align}
The proof is complete.
\end{proof}
\begin{rem}
  This theorem and \eqref{RDEflow} yield that $\Phi:\R_+\times\Omega\times\cal{P}_V(\R^d)\times\R^d\to \cal{P}_V(\R^d)\times\R^d$ is an RDS and satisfies that
 for any $s,t\geq0$, almost all $\omega$,
  $$
  (P_{s+t}^*\mu, \varphi(s+t,\omega,\mu,x))=\Big(P_{t}^*P_{s}^*\mu, 
  \varphi(t,\theta_s\omega,P_{s}^*\mu)\varphi(t,\omega,\mu,x)\Big).
  $$
\end{rem}

We will now focus on the continuity of the system. 
\begin{lem}\label{lem1022}
Under conditions \textbf{(C4)} and \textbf{(C5)}, 
for any bounded subset $S$ of $\cal{P}_V(\R^d)$, 
the mapping $P_t^*$ is continuous from $(S,\cW_2)$ to $(\cal{P}_V(\R^d),\cW_2)$.
\end{lem}

\begin{proof}
Let $\{\mu_n\}$ be a sequence in $S$, and suppose $\mu_n$ converges to $\mu$ in the $\cW_2$ metric.
Due to the boundedness of $S$, there exists $R>0$ such that 
\begin{align}\label{0102-11}
\sup_{n}\int_{\R^d}V(x,\mu_n)\mu_n(\rmd x)\leq R.
\end{align}

For any $N>0$, we denote by $\cF_N$ the set of all $\cal{F}$-stopping times that are bounded by $N$. 
Then for $0<\theta<1$, $\tau,\tau'\in \cF_N$ with $0\leq \tau-\tau'\leq \theta$, 
\begin{align*}
\mE|X_{\tau}^{n}-X_{\tau'}^n|^{2\ell}
&=
\mE\Big|\int_{\tau'}^{\tau}b(X_r^n,\sL_{X_r^n})\rmd r
+\int_{0}^{N}1_{[\tau',\tau]}(r)\sigma(X_r^n)\rmd W_r\Big|^{2\ell}\\
&
\leq C_{\ell} \mE\Big[|\tau-\tau'|^{2\ell-1}\int_{0}^{N}|b(X_r^n,\sL_{X_r^n})|^{2\ell}\rmd r\Big]
+C'_{\ell}\mE\Big[\int_{\tau'}^{\tau}|\sigma(X_r^n)|^2\rmd r\Big]^\ell\\
&
\leq C_{\ell}\mE\Big[\int_{0}^{N}|b(X_r^n,\sL_{X_r^n})|^{2\ell}\rmd r\Big]\theta^{2\ell-1}
+C''_{\ell}\mE\Big[\int_{0}^{N}|\sigma(X_r^n)|^{2\ell}\rmd r\Big]\theta^{\ell-1}\\
&
\leq \widetilde{C}_{\ell}\mE\Big[\int_{0}^{N}1+V(X_r^n,\sL_{X_r^n})\rmd r\Big]\theta^{\ell-1}\\
&
\leq \widetilde{C}_{\ell,N,\alpha,\cal{M}}\Big(1+\mE V(X_0^n,\sL_{X_0^n})\Big)\theta^{\ell-1}\\
&
\leq C(\ell,N,\alpha,\cal{M},R)\cdot\theta^{\ell-1},
\end{align*}
where the penultimate inequality holds by \eqref{1010-1},
the last inequality holds by \eqref{0102-11} with $\sL_{X_0^n}=\mu_n$. 
This with the Chebyshev inequality implies that for any $N>0$, $\varepsilon>0$,
\begin{align}\label{20250506-1}
  \lim_{\theta\to0}\limsup_{n\to\infty}\sup_{\tau,\tau'\in\cF_N:\tau'\leq\tau\leq\tau'+\theta}
  \mathbb{P}\bigg(|X_{\tau}^{n}-X_{\tau'}^n|\geq\varepsilon\bigg)=0.
\end{align}
Moreover, for any $N>0$, 
\begin{align}\label{20250508-1}
\mE\Big[\sup_{t\leq N}|X_t^n|^{2\ell}\Big]
&
\leq \mE|X_0^n|^{2\ell}+\mE\sup_{t\leq N}\Big|\int_{0}^{t}b(X_r^n,\sL_{X_r^n})\rmd r
+\int_{0}^{t}\sigma(X_r^n)\rmd W_r\Big|^{2\ell}\nonumber\\
&
\leq  \mE|X_0^n|^{2\ell}+C_{\ell,N} \mE\Big[\int_{0}^{N}|b(X_r^n,\sL_{X_r^n})|^{2\ell}\rmd r\Big]
+C'_{\ell,N}\mE\Big[\int_{0}^{N}|\sigma(X_r^n)|^{2\ell}\rmd r\Big]\nonumber\\
&
\leq \mE|X_0^n|^{2\ell}+ \widetilde{C}_{\ell,N}\mE\Big[\int_{0}^{N}1+V(X_r^n,\sL_{X_r^n})\rmd r\Big] \nonumber\\
&
\leq \mE|X_0^n|^{2\ell}+ \widetilde{C}_{\ell,N,\alpha,\cal{M}}\Big(1+\mE V(X_0^n,\sL_{X_0^n})\Big)\nonumber\\ 
&
\leq C'(\ell,N,\alpha,\cal{M},R).
\end{align}
The Chebyshev inequality implies that 
\begin{align*}
  \mathbb{P}\bigg(\sup_{t\leq N}|X_{t}^{n}|> K\bigg)
  &
  \leq \frac{1}{K^{2\ell}}\mE\Big[\sup_{t\leq N}|X_t^n|^{2\ell}\Big]\\
  &
  \leq \frac{C'(\ell,N,\alpha,\cal{M},R)}{K^{2\ell}}
  \to0,\quad \text{as}~K\to\infty.
\end{align*}
By the Aldous criterion (see \cite[Theorem 16.10]{Billingsley}), this with \eqref{20250506-1} suggests that $\sL_{X_{\cdot}^{n}}$ is tight. 
Hence, there exists a subsequence $\{n_k\}$ such that $\sL_{X_{\cdot}^{n_k}}$ weakly converges to some measure $\mu_{\cdot}\in\cP(C[0,T])$. 
By the Skorokhod representation theorem, there are $C[0, T]$-valued
random variables $\bar{X}_{n_k}$ and $\bar{X}$ on some probability space $(\bar{\Omega},\bar{\cF},\bar{\mathbb{P}})$ such that $\sL_{\bar{X}_{n_k}}=\sL_{X_{\cdot}^{n_k}}$, $\sL_{\bar{X}}=\mu_{\cdot}$ and 
$$
\bar{X}_{n_k}\to\bar{X},~a.s.
$$
Using the method of  \cite[Theorem 3.1]{LM2}, we can prove that $\bar{X}$ is a weak solution to the equation.

Similar to \eqref{20250508-1}, it holds that
$$
\bar{\mE} \sup_{t\in[0,T]}|\bar{X}_{n_k,t}|^{2\ell}\leq C,
$$
where $C>0$ is independent of $n_k$.
The Vitali convergence theorem implies that
$$
\lim_{k\to\infty}\bar{\mE} \sup_{t\in[0,T]}|\bar{X}_{n_k,t}-\bar{X}_{t}|^{2}=0.
$$
Hence, it holds that for any $t\geq0$, $P_t^*\mu_{n_k}\to_{\cW_2} P_t^*\mu$ as $n_k\to\infty$.
The proof is complete.
\end{proof}

\begin{lem}\label{lem0102-2}
Under conditions \textbf{(C1)}, \textbf{(C4)}, \textbf{(C5)} and \textbf{(C7)},   
for all $t\in\R_+$ and $\omega\in\Omega$, the mapping 
$$
\chi_t^{\cdot}(\omega,\cdot):S\times\R^d\rightarrow\R^d,\quad
(\mu,x)\mapsto\chi^{\mu}_t(\omega,x)
$$ is continuous, 
where $S\subset\cal{P}_V(\R^d)$ is as in Lemma \ref{lem1022}.
\end{lem}
\begin{proof}
 Let $\{\mu_n\}$ be a sequence in $S$, and suppose $\mu_n$ converges to $\mu$ in the $\cW_2$ metric. 
 And let $x_n$ converge to $x$ in $\R^d$. 
  It follows from \eqref{1126-1} of \textbf{(C5)} that
  \begin{align}\label{0902-3}
   & 
   |\chi_t^{\mu_n}(\omega,x_n)-\chi_s^{\mu_n}(\omega,x_n)|^{2\ell}\nonumber \\
    &
     \leq |t-s|^{2\ell-1}\int_{s}^t\Big|\big(\partial_x \cT \big)^{-1}(\theta_r\omega,\chi_r^{\mu_n}(\omega,x_n))\Big|^{2\ell}
    \Big|b( \cT (\theta_r\omega)\chi_r^{\mu_n}(\omega,x_n),P_r^*\mu_n)\nonumber\\
    &
    \qquad\qquad\qquad\qquad\qquad\qquad\qquad+\eta\sigma( \cT (\theta_r\omega)\chi_r^{\mu_n}(\omega,x_n))Z^I_0(\theta_r\omega)\Big|^{2\ell} \rmd r\nonumber\\
    &
     \leq C(\omega,T,L,\ell,\eta)|t-s|^{2\ell-1}\int_{s}^t\Big|b( \cT (\theta_r\omega)\chi_r^{\mu_n}(\omega,x_n),P_r^*\mu_n)\Big|^{2\ell}+\Big|\sigma( \cT (\theta_r\omega)\chi_r^{\mu_n}(\omega,x_n))\Big|^{2\ell} \rmd r\nonumber\\
    &
    \leq C(\omega,T,L,\ell,\eta,K)|t-s|^{2\ell-1}\int_{s}^t\Big(1+V\big( \cT (\theta_r\omega)\chi_r^{\mu_n}(\omega,x_n),P_r^*\mu_n\big)\Big)\rmd r,
  \end{align}
  where the second inequality holds since the boundedness of the first derivative of $\sigma$ ensures that 
  $$
  \big(\partial_x \cT \big)^{-1}(\theta_r\omega,\chi_r^{\mu_n}(\omega,x_n))
  $$ 
  is controlled by a constant independent of $n$.

  Set
  \begin{align*}
  b_t:=&\bar{b}(\widetilde{X}_t^{\widetilde{\xi}},P_{-t,0}^*\mu),\quad \sigma_t:=\sigma(\widetilde{X}_t^{\widetilde{\xi}}),
  \end{align*}
  where $\widetilde{X}_t^{\widetilde{\xi}}$ is defined on some probability space $(\widetilde{\Omega},\widetilde{\cF},\widetilde{\mathbb{P}})$ and its law is $P_{-t,0}^*\mu$.
    Then it follows from \eqref{1126-1} of \textbf{(C5)} and \eqref{1010-1} that 
  \begin{align*}
  \widetilde{\mE}\Big(\int_0^T|b_t|^2+|\sigma_t|^4\rmd t\Big)<\infty.
  \end{align*}   
   Hence, apply Proposition \ref{20250526-1} to $\mu\mapsto V(x,\mu)$ and obtain that
  \begin{align*}
  V(x,P_{-t,0}^*\mu) = V(x,\mu) 
  &
  +\int_0^t \widetilde{\mE}[\tilde{b}(\widetilde{X}_r^{\widetilde{\xi}},P_{-r,0}^*\mu)\partial_{\mu}V(x,P_{-r,0}^*\mu,\widetilde{X}_r^{\widetilde{\xi}})]\,\rmd r \\
  &
  +\frac{1}{2}
  \sum_{i=1}^m\int_0^t \widetilde{\mE}[\sigma_i(\widetilde{X}_r^{\widetilde{\xi}})
  \partial_y\partial_{\mu}V(x,P_{-r,0}^*\mu,\widetilde{X}_r^{\widetilde{\xi}})
  \sigma^T(\widetilde{X}_r^{\widetilde{\xi}})]\,\rmd r,
  \end{align*}
 where $P_{-r,0}^*\mu$ is the law of $\widetilde{X}_r^{\widetilde{\xi}}$. Then we apply the chain rule of differentiation to $V(\cdot,P_{-t,0}^*\mu)$ and \eqref{0102-1} of  \textbf{(C7)}, thereby obtaining that
  \begin{align*}
  V(\chi_t^\mu(\omega,x),P_{-t,0}^*\mu)
  & 
    =V(x,\mu)+\int_0^t \big\langle\partial_x V(\chi_r^\mu(\omega,x),P_{-r,0}^*\mu), h^{\eta}(Z_0(\theta_r\omega),\chi_r^\mu(\omega,x),P_{-r,0}^*\mu)\big\rangle\,\rmd r \\
   &
   \qquad\qquad+\int_0^t \widetilde{\mE}\big[\big\langle
   \partial_{\mu}V(\chi_r^\mu(\omega,x),P_{-r,0}^*\mu,\widetilde{X}_r^{\widetilde{\xi}}), \overline{b}(\widetilde{X}_r^{\widetilde{\xi}},P_{-r,0}^*\mu)\big\rangle\big]\,\rmd r  \\
  &
  \qquad\qquad+\frac{1}{2}
  \sum_{i=1}^m\int_0^t \widetilde{\mE}[\sigma_i(\widetilde{X}_r^{\widetilde{\xi}})
  \partial_y\partial_{\mu}V(x,P_{-r,0}^*\mu,\widetilde{X}_r^{\widetilde{\xi}})
  \sigma^T(\widetilde{X}_r^{\widetilde{\xi}})]\,\rmd r\\
  &
  \leq V(x,\mu)+\int_0^t K(Z^I_0(\theta_r\omega))V(\chi_r^\mu(\omega,x),P_{-r,0}^*\mu)+L(Z^I_0(\theta_r\omega))\,\rmd r.
  \end{align*}
The Gronwall inequality implies that
  \begin{align}\label{0102-3}
  V(\chi_t^\mu(\omega,x),P_{-t,0}^*\mu)
  \leq V(x,\mu)e^{\int_{0}^tK(Z^I_0(\theta_{r}\omega))\,\rmd r}+\int_{0}^t e^{\int_{0}^rK(Z^I_0(\theta_{u}\omega))\,\rmd u}\, L(Z^I_0(\theta_{r}\omega))\,\rmd r.
  \end{align}
Hence, there exist $C_1(\omega,T)>0$ and $C_2(\omega,T)>0$ such that
\begin{align*}
V(\chi_t^{\mu_n}(\omega,x_n),P_{t}^*\mu_n)
&
\leq C_1(\omega,T)V(x_n,\mu_n)+C_2(\omega,T).
\end{align*}
Since $\mu_n$ belongs to the bounded set $S$ and $x_n$ converges to $x$, 
there exists a constant $R>0$ such that $\cW_p(\mu_n,\delta_0)\vee |x_n|\leq R$. 
This implies that there exists $C(\omega,T,R)>0$ such that for any $n$,
\begin{align*}
V(\chi_t^{\mu_n}(\omega,x_n),P_{t}^*\mu_n)
&
\leq C(\omega,T,R).
\end{align*}
Since $ \cT (\theta_t\omega,\cdot)$ is a diffeomorphism, 
continuing the computation from \eqref{0902-3} yields that
  \begin{align*}
   & 
   |\chi_t^{\mu_n}(\omega,x_n)-\chi_s^{\mu_n}(\omega,x_n)|^{2\ell} \\
   &
    \leq C(\omega,T,L,\ell,\eta,K)\big(1+ C(\omega,T,R)\big)|t-s|^{2\ell}.
  \end{align*}  
   Hence, $\{\chi^{\mu_n}(\omega,x_n)\}_n\subset C[0,T]$ is equicontinuous. 
   Additionally, the above formula implies that 
  \begin{align*}
   &\sup_{t\in[0,T]} |\chi_t^{\mu_n}(\omega,x_n)|^{2\ell} \leq |x|^{2\ell} +C'(\omega,T,L,\ell,\eta,K).
  \end{align*}
  Hence, $\{\chi^{\mu_n}(\omega,x_n)\}_n\subset C[0,T]$ is uniformly bounded. 
  The Arzel\'a-Ascoli theorem implies that there exist a subsequence $n_k$ and $\chi(\omega)\in C[0,T]$ 
  such that 
  $$\lim_{n_k\to\infty}\sup_{t\in[0,T]}\big|\chi_t^{\mu_{n_k}}(\omega,x_{n_k})-\chi_t(\omega)\big|=0.$$
  Note that 
  \begin{align*}
  \chi_t^{\mu_n}(\omega,x_n)=x_n+\int_{0}^t\big(\partial_x \cT \big)^{-1}(\theta_r\omega,\chi_r^{\mu_n}(\omega,x_n))
    \Big(&b(\cT (\theta_r\omega)\chi_r^{\mu_n}(\omega,x_n),P_r^*\mu_n)\\ 
    &+\eta\sigma( \cT (\theta_r\omega)\chi_r^{\mu_n}(\omega,x_n))Z_0(\theta_r\omega)\Big) \rmd r.
  \end{align*}
  Similar to the approach for proving equicontinuity, 
  we can obtain that the integrand can be controlled by some constant independent of $n$. 
  The Lebesgue dominated convergence theorem implies that $\chi_t(\omega)$ is the solution of equation \eqref{RDE1} with $(x,\mu)$. 
  Hence, $\chi_t(\omega)=\chi_t^{\mu}(\omega,x)$. 
  The proof is completed.
\end{proof}

Now we consider the existence of random attractors for system $\overline{\chi}$. 
In preparation for this, we shall first examine the existence of $\sS$-global attractor of $P_t^*$.

\begin{lem}\label{thmSODE1}
Under conditions \textbf{(C4)} and \textbf{(C5)}, $P_t^*$ admits a $\sS$-global attractor in $(\cP_V(\R^d),\cW_2)$. 
\end{lem}
\begin{proof}
By \eqref{1010-1}, it holds that 
 \begin{align}\label{0605--1} 
  \mE V(X _{t}^{\xi,\mu},P_{t}^*\mu )
   &
   \leq e^{-\alpha t}\mE V(\xi,\mu)+\frac{\cal{M}}{\alpha}\Big(1-e^{-\alpha t}\Big)\nonumber\\
   &
    \rightarrow \frac{\cal{M}}{\alpha}=:\gamma, 
   \quad \text{as}~t\rightarrow+\infty.
  \end{align}
 Condition \textbf{(C2)} implies that 
  \begin{align*} 
  \cB:=\Big\{\mu\in\cal{P}_V(\R^d):\int_{\R^d}V(x,\mu)\mu(\rmd x)\leq\gamma\Big\}
  \subset\Big\{\mu\in\cal{P}_V(\R^d):|\mu|^{p}_{p}\leq K(1+\gamma)\Big\}=:\cB_{\cP}.
  \end{align*} 
  By the continuity of $V$, \cite[lemma 12.8]{ABS} implies that for any $\cW_2$-convergent sequence $\{\mu_k\}\subset\cB$ with the limit $\mu$, it holds that
  $$
  \int_{\R^d}V(x,\mu)\mu(\rmd x)\leq \liminf_{k\to+\infty}\int_{\R^d}V(x,\mu_k)\mu_k(\rmd x)\leq\gamma.
  $$
  This implies that $\cB$ is closed under the metric $\cW_2$. Hence, $\cB\in\sS$.
  Due to $p>2$ and the compactness of $\cB_{\cP}$ under the metric $\cW_2$, $\cB$ is compact under the metric $\cW_2$.
  Moreover, \eqref{0605--1} suggests that $\cB$ absorbs all element in $\sS$.
  By Lemma \ref{lem1022} and Lemma \ref{EGA}, the system admits a $\sS$-global attractor, denoted as $\cP^*$.
\end{proof}

Now, we restrict the RDSs generated by system \eqref{1111} or \eqref{1111-2} to $\cP^*\times\R^d$ for analysis.
Lemma \ref{lem1022}, Lemma \ref{lem0102-2} and Theorem \ref{thm1} imply the continuity of the RDSs.
\begin{cor}\label{0106-4}
Under conditions \textbf{(C1)}, \textbf{(C4)}, \textbf{(C5)} and \textbf{(C7)},
systems $\overline{\chi}:\R_+\times\Omega\times\cP^*\times\R^d\to\cP^*\times\R^d$ and $\Phi:\R_+\times\Omega\times\cP^*\times\R^d\to\cP^*\times\R^d$ are continuous RDSs.
\end{cor}

With continuity established, we now turn to the existence of a random pullback attractor.
\begin{prop}\label{lem1}
Under the hypotheses of Corollary \ref{0106-4}, $\overline{\chi}$ admits a $\sD$-random attractor $\sA$. 
\end{prop}
\begin{proof}
  By \eqref{0102-3}, it holds that for any $t>0$, $(x,\mu)\in\R^d\times \cP^*$, 
  \begin{align}\label{0609--2}
  V(\overline{\chi}_t(\theta_{-t}\omega)(\mu,x))
  &
  =V(\chi_t^\mu(\theta_{-t}\omega,x),P_{-t,0}^*\mu)\\
  &
  \leq V(x,\mu)e^{\int_{-t}^0K(Z^I_0(\theta_{r}\omega))\,\rmd r}+\int_{-t}^0 e^{\int_{r}^0K(Z^I_0(\theta_{u}\omega))\,\rmd u}\, L(Z^I_0(\theta_{r}\omega))\,\rmd r.\nonumber
  \end{align}
  Since $K$ is subexponential and $\theta_t$ is ergodic with respect to $\mathbb{P}$, 
  the Birkhoff ergodic theorem implies that
  $$
  \lim_{t\rightarrow+\infty}\frac{1}{t}\int_{-t}^0K(Z^I_0(\theta_{r}\omega))\,\rmd r=\int_{\R^m}K(z)\sL_{Z^I_0}(\rmd z)<0,
  $$
  which, along with the subexponential property of $L$, 
  indicates that the limit of the first term on the right-hand side of \eqref{0609--2} is zero, 
  while the limit of the second term exists and is strictly positive as $t\rightarrow+\infty$. 
  Therefore, we can define that for any $\mu\in\cP^*$,
  $$
  \cB(\omega,\mu):=\{x\in\R^d:V(x,\mu)\leq 2\gamma(\omega)\},
  $$
  where
  $$
  \gamma(\omega)=\int_{-\infty}^0 e^{\int_{r}^0K(Z^I_0(\theta_{u}\omega))\,\rmd u}\, L(Z^I_0(\theta_{r}\omega))\,\rmd r<+\infty.
  $$
  It follows from \eqref{1126-1} that for any $\mu\in\cP^*$,
  $$
  \cB(\omega,\mu)\subset
  \Big\{x\in\R^d:|x|\leq K(1+2\gamma(\omega))+1\Big\}
  =:\widetilde{K}(\omega).
  $$  
\cite[Theorem 2.12]{IS} implies that $\gamma$ is tempered.
This suggests that $\widetilde{K}:=\{\widetilde{K}(\omega)\}_{\omega\in\Omega}$ belongs to $\cD$. And it follows from \eqref{0609--2} that $\widetilde{K}$ satisfies \eqref{0916:02}. 

For any $\omega\in\Omega$, $\cP^*\times D\in\sD$, $t_n\to+\infty$, and $(\mu_n,x_n)\in \cP^*\times D(\theta_{-t_n}\omega)$, 
it follows from \eqref{0609--2} that there exists $N>0$ such that for any $n>N$, 
$$
\overline{\chi}_{t_n}(\theta_{-t_n}\omega)(\mu_n, x_n)\in \cP^*\times\widetilde{K}(\omega).
$$
Due to the compactness of $\cP^*\times\widetilde{K}(\omega)$, 
$\{\overline{\chi}_{t_n}(\theta_{-t_n}\omega)(\mu_n, x_n)\}_{n>N}$ possesses a convergent subsequence. 
This implies that $\overline{\chi}$ is pullback $\sD$-asymptotically compact.
By Proposition~\ref{prop0117}, one sees that $\overline{\chi}$ admits a $\sD$-random attractor. The proof is complete.
\end{proof}

\begin{proof}[Proof of Theorem \ref{thmSODE}]
(1)
The existence of the $\sS$-global attractor is given in Lemma \ref{thmSODE1}.

(2)
(i) It follows from Lemma \ref{lem0928} and Proposition \ref{lem1}.

(ii) 
Firstly, the homeomorphism of $\cT$ guarantees that $\cT \cD$ is neighborhood closed. Moreover,
For any $\omega\in\Omega$, $ \cT  D\in \cT \cD$, $t_n\to+\infty$, $\mu_n\in \cP^*$, $\nu_n\in \Theta_{2,-t_n}\mu_n\subset\cP^*$ and $x_n\in  \cT (\theta_{-t_n}\omega) D(\theta_{-t_n}\omega)$, 
it follows from \eqref{0609--2} that there exists $N>0$ such that for any $n>N$,
$$
\varphi(t_n,\theta_{-t_n}\omega,\nu_n,x_n)= \cT (\omega)\chi_{t_n}^{\nu_n}\left(\theta_{-t_n}\omega, \cT ^{-1}(\theta_{-t_n}\omega)x_n\right)\in \cT (\omega)\widetilde{K}(\omega). 
$$
Due to the compactness of $ \cT (\omega)\widetilde{K}(\omega)$, $\varphi:\R_+\times\Omega\times \cP^*\times \R^d\to\R^d$ is pullback $ \cT \cD$-asymptotically compact.
The continuity of $\overline{\chi}$ and \eqref{1004--1} implies the continuity of $\varphi$. 
Finally, since \eqref{1004--1} holds and $\widetilde{K}$ satisfies \eqref{0916:02}, it follows that
\begin{align*}
\lim_{t\to+\infty}\sup_{\mu\in\cP^*}
\dist\big(\varphi(t,\theta_{-t}\omega,\mu)(\cT (\theta_{-t}\omega)D(\theta_{-t}\omega)),
 \cT (\omega)\widetilde{K}(\omega)\big)=0.
\end{align*}
This completes the verification of all the conditions of Theorem \ref{thm0330}.
Hence, $\varphi$ admits a $ \cT \cD$-random attractor, which is of the form as mentioned above.
\end{proof}

\subsection{Proof of Example \ref{ex01}}\label{secex01}

To facilitate the verification, we introduce the following condition:
\begin{itemize}
\item [\textbf{(S1)}] 
 There exist two subexponentially growing function $K_1,K_2:\R^m\rightarrow\R_+$ satisfying
   \begin{align}\label{c3}
   \lim_{z\rightarrow0}K_2(z)=0
   \end{align}
   such that
  \begin{align}\label{c1}
   \limsup_{|x|\wedge|\mu|_2\rightarrow\infty}\sup_{z\in\R^m} 
  \frac{\big|\big\langle \partial_xV(x,\mu), L^1(z,x) \big\rangle\big|}{V(x,\mu)K_1(z)} =0
  \end{align}
  and
  \begin{align}\label{c2}
     \limsup_{|x|\wedge|\mu|_2\rightarrow\infty}\sup_{z\in\R^m}\frac{\bigr|\bigr\langle \partial_x V(x,\mu), b(x,\mu)-L^2(z,x,\mu) \bigr\rangle\bigr|}{V(x,\mu) K_2(z)} \leq1,
  \end{align}
  where $$L^1(z,x):=\big(\partial_xu\big)^{-1}(z,x)\,\sum_{i=1}^{m}\sigma_i\big(u(z,x)\big)z^i$$
  and
  $$L^2(z,x,\mu):=\big(\partial_xu\big)^{-1}(z,x)\,b(u(z,x),\mu).$$
\end{itemize}

\begin{rem}
Consider $L^1(z,x)+\eta L^2(z,x,\mu)=:h^{\eta}(z,x,\mu)$ as a perturbation of $b$; 
  therefore, we need to impose additional conditions on the error between $h^{\eta}(z,x,\mu)$ and $b(x,\mu)$, 
  as stated in conditions \eqref{c1} and \eqref{c2}.
Since the second term $\eta\big(\partial_x\Phi\big)^{-1}(\omega,x)\sigma(\Phi(\omega,x))Z_0(\omega)$ of \eqref{g} can be neglected in the sense of \eqref{c1},
   this also ensures that $\eta$ can be chosen to be sufficiently large.  
\end{rem}

\begin{lem}\label{0106-3}
Assume that conditions \textbf{(C1)}, \textbf{(C4)} and \textbf{(C5)} hold. 
If $\sup_{(x,\mu)\in \mathcal{K}_M}|\partial_x V(x,\mu)|<\infty$ for any $M>0$, then \textbf{(S1)} imply \textbf{(C7)}, where $\mathcal{K}_M:=\{(x,\mu);|x|\vee|\mu|_2\leq M\}$.
\end{lem}
\begin{proof}
  It is straightforward that  
    \begin{align}\label{0106-1}
   \cal{L}V(x,\mu)\leq -\alpha V(x,\mu)+\mathcal{M},
  \end{align}
  where 
  \begin{align*}
  \cal{L}V(x,\mu):=\big\langle \partial_x V(x,\mu), b(x,\mu) \big\rangle 
  &+ \int_{\R^d}\big\langle \partial_{\mu} V(x,\mu)(y), \overline{b}(y,\mu) \big\rangle\,\mu(\rmd y)\\
  &+\frac{1}{2} \sum_{i=1}^m\int_{\R^d}\sigma_i(y)\partial_y\partial_{\mu}V(x,\mu)(y)\sigma_i^T(y)\,\mu(\rmd y)
  \end{align*}
for $(x,\mu)\in\R^d\times\cal{P}_2(\R^d)$.

Additionally, as $\eta\rightarrow+\infty$,  $Z^I_0$ approximates $0$ in the $L^2$ sense.
Therefore, under condition \eqref{c3}, we can choose a sufficiently large $\eta$ such that the distance between $\sL_{Z^I_0}$ and $\delta_0$ is sufficiently small and then select a sufficiently small $\delta:=\delta(\eta)$ such that the integral 
\begin{align}\label{0905--1}
\int_{\R^m}\delta\eta K_1(z)+K_2(z)\sL_{Z^I_0}(\rmd z)<\alpha.
\end{align}

Due to 
$$
h^{\eta}(z,x,\mu)=L^2(z,x,\mu)+\eta L^1(z,x),
$$
  it holds that
  \begin{align*}
  &
  \frac{\big|\big\langle\partial_x V(x,\mu), h^{\eta}(z,x,\mu)-b(x,\mu)\big\rangle\big|}{V(x,\mu)(K_2(z)+\eta\delta K_1(z))}\\
  &
  \leq\frac{\big|\bigr\langle \partial_x V(x,\mu), L^2(z,x,\mu)-b(x,\mu)\bigr\rangle\big|}{V(x,\mu)K_2(z)}
  +\frac{1}{\delta } \frac{\big|\big\langle \partial_xV(x,\mu), L^1(z,x) \big\rangle\big|}{V(x,\mu)K_1(z)}.
  \end{align*}
  It follows from \eqref{c1} and \eqref{c2} that
 \begin{align*}
  \limsup_{|x|\wedge|\mu|_2\rightarrow\infty}\sup_{z\in\R^m}\frac{\big|\big\langle\partial_x V(x,\mu), h^{\eta}(z,x,\mu)-b(x,\mu)\big\rangle\big|}{V(x,\mu)(K_2(z)+\eta\delta K_1(z))}
  \leq1.
 \end{align*}
Let $\bar{K}(z):=K_2(z)+\eta\delta K_1(z)$. Hence, for any $\varepsilon>0$, there exists $M>0$ such that $(x,\mu)\in\{(x,\mu);|x|\wedge|\mu|_2\geq M\}$ and $z\in\R^m$,
\begin{align}\label{0106-2}
\big\langle \partial_x V(x,\mu), h^{\eta}(z,x,\mu)- b(x,\mu)\big\rangle\leq V(x,\mu)(\bar{K}(z)+\varepsilon).
 \end{align}
Additionally, the boundedness of $\partial_x\sigma$ implies that the mapping $z\mapsto\sup_{|x|\leq M}|\big(\partial_xu\big)^{-1}(z,x)|$ is subexponentially growing, which guarantees that $z\mapsto\sup_{|x|\vee|\mu|_2\leq M}|h^{\eta}(z,x,\mu)|$ is subexponentially growing.
$$
L(z):=\sup_{(x,\mu)\in \mathcal{K}_M}|\partial_x V(x,\mu)|\Big(\sup_{(x,\mu)\in \mathcal{K}_M}|h^{\eta}(z,x,\mu)|+\sup_{(x,\mu)\in \mathcal{K}_M}|b(x,\mu)|\Big)<+\infty,
$$
which is subexponentially growing.
 This implies that for any $(x,\mu)\in \R^d\times \cal{P}_2(\R^d)$, 
  \begin{align*}
  \widetilde{\cal{L}}^\eta V(x,\mu)(z) 
  &= \cal{L}V(x,\mu)+\big\langle \partial_x  V(x,\mu), h^{\eta}(z,x,\mu)- b(x,\mu)\big\rangle\\
  &\leq (-\alpha+\bar{K}(z)+\varepsilon)V(x,\mu)+L(z)+\mathcal{M}.
  \end{align*}
  By \eqref{0905--1}, we can choose a sufficiently small $\varepsilon>0$ to ensure 
  $$
  \int_{\R^m}(-\alpha+\bar{K}(z)+\varepsilon) \sL_{Z^I_0}(\rmd z)<0,
  $$
  which implies that \textbf{(C7)}.
\end{proof}

\begin{proof}[Proof of Example \ref{ex01}] 
Choose the function
   \begin{align*}
   V(x,\mu)=x^8+\Big(\int_{\R}y^2\mu(\rmd y)\Big)^{45},
   \end{align*}
   which belongs to $\Xi$ and satisfies 
   $\sup_{(x,\mu)\in \mathcal{K}_M}|\partial_x V(x,\mu)|<\infty$ for any $M>0$. 
   And it holds that
   $$\partial_xV(x,\mu)=8x^7,\quad \partial^2_{xx}V(x,\mu)=56x^6,$$
   $$\partial_{\mu}V(x,\mu)(z)=90\Big(\int_{\R}y^2\mu(\rmd y)\Big)^{44}z,\quad \partial_z\partial_\mu V(x,\mu)(z)=90\Big(\int_{\R}y^2\mu(\rmd y)\Big)^{44}.$$

  Through simple calculations,
  \begin{align*}
  &\cL V(x,\mu)\\
  =&
  8x^7\Bigr(-x^3\Big(\int_{\R}y^2\mu(\rmd y)\Big)^{5}-10x+\frac{1}{2}\sigma(x)\sigma'(x)\Bigr)+28x^6\sigma^2(x)\\
  &
  +90\Big(\int_{\R}y^2\mu(\rmd y)\Big)^{44}\int_{\R}z\Big(-z^3\Big(\int_{\R}y^2\mu(\rmd y)\Big)^{5}-10z+\frac{1}{2}\sigma(z)\sigma'(z)\Big)\mu(\rmd z)\\
  &
  +45\Big(\int_{\R}y^2\mu(\rmd y)\Big)^{44}\int_{\R}\sigma^2(z)\mu(\rmd z)\\ 
  \leq&
  -8x^{10}\Big(\int_{\R}y^2\mu(\rmd y)\Big)^5-80x^{8}+4x^7\sigma(x)\sigma'(x)+28x^6\sigma^2(x)\\
  &
  +90\Big(\int_{\R}y^2\mu(\rmd y)\Big)^{44}\int_{\R}-10z^2+\frac{1}{2}z\sigma(z)\sigma'(z)\mu(\rmd z)+\frac{405}{4}\Big(\int_{\R}y^2\mu(\rmd y)\Big)^{45}\\
  \leq
  &
  (-80+9+63)x^8+\Big(-900+\frac{405}{4}+\frac{405}{4})\Big(\int_{\R}y^2\mu(\rmd y)\Big)^{45}\\
  \leq&
  -8V(x,\mu).
  \end{align*}
  This verifies \textbf{(C4)}. Conditions \textbf{(C1)}-\textbf{(C3)}, \textbf{(C5)} and \textbf{(C6)} are straightforward.
  The existence and uniqueness of solutions can be guaranteed; see \cite{MA2} for more details. 
  
  We now verify \textbf{(S1)}.
  Following the method of the second section, we can obtain the conjugate random differential equation:
   \begin{align*}
    \rmd Y_t=g(\theta_t\omega,Y_t,\mu_t)\,\rmd t,
   \end{align*}
   where
\begin{align*}
   g(\omega,x,\mu)&=\big(\partial_x\cT\big)^{-1}(\omega,x)\Big(b(\cT(\omega)x,\mu)
   +\lambda\sigma(\cT(\omega)x)Z_0(\omega)\Big).
\end{align*}
 And we have
   \begin{align*}
   L^1(z,x)&=e^{-\int_0^z\sigma'(u(r,x))\rmd r}\sigma\big(u(z,x)\big)z,
   \end{align*}
    and
  \begin{align*}
  L^2(z,x,\mu)&=e^{-\int_0^z\sigma'(u(r,x))\rmd r}b(u(z,x),\mu),
  \end{align*}
  where $u(z,x)$ satisfies 
  $$|u(z,x)|\leq|x|e^{\frac{3}{2}|z|}.$$
  This implies that the mapping $z\mapsto\sup_{|x|\leq L}|\big(\partial_xu\big)^{-1}(z,x)|$ is subexponentially growing for any $L>0$.  
  We choose $K_1(z)=e^{|z|}(|z|+1)$ to obtain that for any $x\in\R,z\in\R$,
  \begin{align*}
    \frac{\big|\big\langle \partial_xV(x,\mu), L^1(z,x) \big\rangle\big|}{V(x,\mu)K_1(z)}
    &
    =\Big|\frac{8x^7e^{-\int_0^z\sigma'(u(r,x))\rmd r}\sigma\big(u(z,x)\big)z}{\Big(x^8+\Big(\int_{\R}y^2\mu(\rmd y)\Big)^{45}\Big)K_1(z)}\Big|\\
     &
    \leq\Big|\frac{12x^7|x|e^{|z|}}{\Big(x^8+\Big(\int_{\R}y^2\mu(\rmd y)\Big)^{45}\Big)e^{|z|}}\Big|,
    \end{align*}
    where the inequality holds by applying the following inequality
    $$
    e^{-\int_0^z\sigma'(u(r,x))\rmd r}\sigma\big(u(z,x)\big)
    \leq \frac{3}{2} |x|e^{|z|}.
    $$
Hence, 
 \begin{align*}
 \limsup_{|x|\wedge|\mu|_2\rightarrow\infty}\sup_{z\in\R} 
 \frac{\big|\big\langle \partial_xV(x,\mu), L^1(z,x) \big\rangle\big|}{V(x,\mu)K_1(z)}=0.
 \end{align*}

In addition, we choose $K_2(z)=e^{6|z|}-1$ for $z\neq0$, with $K_2(0)=1$. Then
  \begin{align}\label{0903-1}
   \frac{\bigr|\bigr\langle \partial_x V(x,\mu), L^2(z,x,\mu)-b(x,\mu) \bigr\rangle\bigr|}{V(x,\mu) K_2(z)}
   &=
      \frac{\bigr|8x^7\bigr(\big(\partial_xu\big)^{-1}(z,x)\,b(u(z,x),\mu)-b(x,\mu)\bigr)\bigr|}{\Big(x^8+\Big(\int_{\R}y^2\mu(\rmd y)\Big)^{45}\Big) K_2(z)}\\
   &\leq
      \frac{8|x|^7\Big(\int_{\R}y^2\mu(\rmd y)\Big)^5}{x^8+\Big(\int_{\R}y^2\mu(\rmd y)\Big)^{45}}\cdot \Big| \frac{e^{-\int_0^z\sigma'(u(r,x))\rmd r}u^3(z,x)-x^3}{K_2(z)}\Big|\nonumber\\
      &\quad+
      \frac{80|x|^7}{x^8+\Big(\int_{\R}y^2\mu(\rmd y)\Big)^{45}}\cdot \Big| \frac{e^{-\int_0^z\sigma'(u(r,x))\rmd r}u(z,x)-x}{K_2(z)}\Big|.  \nonumber
  \end{align}
Now, we aim to show that there exists a constant $C>0$ such that
  \begin{align}\label{0107-1}
  \sup_{z\in\R}\Big| \frac{e^{-\int_0^z\sigma'(u(r,x))\rmd r}u^3(z,x)-x^3}{K_2(z)}\Big|\leq C|x|^3 
  ~\text{and}~
  \sup_{z\in\R}\Big| \frac{e^{-\int_0^z\sigma'(u(r,x))\rmd r}u(z,x)-x}{K_2(z)}\Big|\leq C|x|.   
  \end{align}
Note that
\begin{align*}
&e^{-\int_0^z\sigma'(u(r,x))\rmd r}u^3(z,x)-x^3\\
&=\int_0^ze^{-\int_0^\tau\sigma'(u(r,x))\rmd r}u^2(\tau,x)\bigr(3\sigma(u(\tau,x))-\sigma'(u(\tau,x))u(\tau,x)\bigr) \rmd \tau
\end{align*}
and 
\begin{align*}
&e^{-\int_0^z\sigma'(u(r,x))\rmd r}u(z,x)-x\\
&=\int_0^ze^{-\int_0^\tau\sigma'(u(r,x))\rmd r}\bigr(\sigma(u(\tau,x))-\sigma'(u(\tau,x))u(\tau,x)\bigr) \rmd \tau.
\end{align*}
Hence, for $z\neq0$, it holds that
\begin{align*}
&\Big|e^{-\int_0^z\sigma'(u(r,x))\rmd r}u^3(z,x)-x^3\Big|\\
&
\leq\int_0^{|z|}e^{-\int_0^\tau\sigma'(u(r,x))\rmd r}u^2(\tau,x)\big(3\big|\sigma(u(\tau,x))\big|+\big|\sigma'(u(\tau,x))u(\tau,x)\big|\big)\rmd \tau\\
&
\leq C\int_0^{|z|}e^{\frac{3}{2}\tau}\Big|u^3(\tau,x)\Big|\rmd \tau\\
&
\leq C'|x|^3\int_0^{|z|}e^{\frac{3}{2}\tau+\frac{9}{2}\tau}\rmd \tau\\
&
= C'\Big(e^{6|z|}-1\Big)|x|^3.
\end{align*}
Similarly, we obtain that 
\begin{align*}
&\Big|e^{-\int_0^z\sigma'(u(r,x))\rmd r}u(z,x)-x\Big|\\
&\leq\int_0^{|z|}e^{\int_0^\tau\sigma'(u(r,x))\rmd r}\Big|\sigma(u(\tau,x))-\sigma'(u(\tau,x))u(\tau,x)\Big|\rmd \tau\\
&\leq C\int_0^{|z|}e^{\frac{3}{2}\tau}\Big|u(\tau,x)\Big|\rmd \tau\\
&\leq C|x|\int_0^{|z|}e^{\frac{3}{2}\tau+\frac{3}{2}\tau}\rmd \tau\\
&\leq C''\Big(e^{6|z|}-1\Big)|x|.
\end{align*}

By \eqref{0107-1}, we obtain that 
\begin{align*}
   \frac{\bigr|\bigr\langle \partial_x V(x,\mu), L^2(z,x,\mu)-b(x,\mu) \bigr\rangle\bigr|}{V(x,\mu) K_2(z)}
   &\leq
      \frac{Cx^{10}\Big(\int_{\R}y^2\mu(\rmd y)\Big)^5}{x^8+\Big(\int_{\R}y^2\mu(\rmd y)\Big)^{45}}
      +
      \frac{Cx^8}{x^8+\Big(\int_{\R}y^2\mu(\rmd y)\Big)^{45}}. 
  \end{align*}
Hence,
\begin{align*}
   &
   \limsup_{|x|\wedge|\mu|_2\rightarrow\infty}\sup_{z\in\R}\frac{\bigr|\bigr\langle \partial_x V(x,\mu), L^2(z,x,\mu)-b(x,\mu) \bigr\rangle\bigr|}{V(x,\mu) K_2(z)}\leq1.
\end{align*}
Having thus verified \textbf{(S1)}, we conclude from Lemma \ref{0106-3} that \textbf{(C7)} holds.

So far, we have verified that the system satisfies all conditions of Theorem \ref{thmSODE}, and consequently established the existence of random attractor.
\end{proof}

\section{Proof of main results for McKean-Vlasov SPDE}\label{secPPDE}

In this section, we present the proofs of the main results stated in Section \ref{RMVSPDE}.
Specifically, the existence of random attractors is established for the following equations: 
the McKean-Vlasov stochastic reaction-diffusion equations in Section \ref{secRDeq}, 
and the McKean-Vlasov stochastic 2D Navier-Stokes equations in Section \ref{secNSeq}.

\subsection{McKean-Vlasov stochastic reaction-diffusion equation}\label{secRDeq}

To prove Theorem \ref{thmRDeq}, we first derive some helpful estimates.

\begin{lem}\label{0302-1}
Assume that {\bf(H3)} holds with $C_1+C_2<2\lambda_*$. Then there exist constants $\eta_1>\eta_2>0$ and $C>0$ such that
for any $u\in V$ and $\mu\in\cP_2(H)$,
\begin{align}\label{1001:01}
2_{V^*}\langle \Delta u+F(u,\mu),u\rangle_V
\leq-\eta_1\|u\|^2+\eta_2\mu(\|\cdot\|^2)+C.
\end{align}
Suppose further that {\bf(H1)}, {\bf(H2$'$)} and {\bf(H4$'$)}  hold.
Then for any $t\geq s$, the solution $Y_t,t\geq s$ to \eqref{0327:01-2} satisfies 
\begin{equation}\label{0911:03}
\mE\|Y_t\|^2\leq \mE\|Y_s\|^2\rme^{-(\eta_1-\eta_2)(t-s)}+C.
\end{equation}
\end{lem}
\begin{proof}
Employing {\bf(H3)}, we have for any $u\in H$ and $\mu\in\cP_2(H)$,
\begin{align}\label{0328:01}
2\,_{V^*}\langle A(u, \mu), u\rangle_V
&
:=2\,_{V^*}\langle \Delta u+F(u, \mu), u\rangle_V\\
&
\leq -2\|u\|_V^2+(2+\lambda_1)\|u\|^2+\lambda_2\mu(\|\cdot\|^2)+C.\nonumber
\end{align}
Then by \eqref{0328:01}, \eqref{0327:02} and the Young inequality, one sees that for all $u\in V$ and $\mu\in\cP_2(H)$,
\begin{align*}
&
2\,_{V^*}\langle A(u, \mu), u\rangle_V \\
&
= 2\,_{V^*}\langle A(u, \mu)-A(0, \delta_0), u\rangle_V+2\,_{V^*}\langle A(0,\delta_0), u\rangle_V\\
&
\leq C_1\cW_2(\mu,\delta_0)^2-(2\lambda_*-C_2)\|u\|^2+\varepsilon\|u\|^\alpha_V+C_\varepsilon\\
&
\leq -(2\lambda_*-C_2)\|u\|^2+C_1\cW_2(\mu,\delta_0)^2+C_\varepsilon
+\frac{(2+\lambda_1)\varepsilon}{2}\|u\|^2\\
&\quad
+\frac{\lambda_2\varepsilon}{2}\cW_2(\mu,\delta_0)^2
-\varepsilon\,_{V^*}\langle A(u, \mu), u\rangle_V\\
&
\leq -(2\lambda_*-C_2-\frac{(2+\lambda_1)\varepsilon}{2})\|u\|^2
+(C_1+\frac{\lambda_2\varepsilon}{2})\cW_2(\mu,\delta_0)^2
-\varepsilon\,_{V^*}\langle A(u, \mu), u\rangle_V+C_\varepsilon.
\end{align*}
Define 
$$
\eta_1:=2\lambda_*-C_2-\frac{(2+\lambda_1)\varepsilon}{2},\quad
\eta_2:=C_1+\frac{\lambda_2\varepsilon}{2}.
$$
Letting 
$\varepsilon\in(0,\frac{2(2\lambda_*-C_1-C_2)}{2+\lambda_1+\lambda_2}
\wedge 1)$,
then there exists $C>0$ such that for any $u\in H_0^1$ and $\mu\in\cP_2(H)$,
\begin{equation}\label{0328:02}
2\,_{V^*}\langle A(u, \mu), u\rangle_V 
\leq -\eta_1\|u\|^2+\eta_2\cW_2(\mu,\delta_0)^2+C.
\end{equation}

By the It\^o formula and \eqref{1001:01}, one sees that
\begin{align}\label{0911:02}
\mE\|Y_t\|^2
&
=\mE\|Y_s\|^2
+\mE\int_s^t\left(2_{V^*}\langle \Delta Y_r+F(Y_r,\sL_{Y_r}),Y_r\rangle_V+C\right)\rmd r\nonumber\\
&
\leq\mE\|Y_s\|^2+\mE\int_s^t\left(-\eta_1\|Y_r\|^2+\eta_2\mE\|Y_r\|^2+C\right)\rmd r\\
&
\leq \mE\|Y_s\|^2+\int_s^t\left(-(\eta_1-\eta_2)\mE\|Y_r\|^2+C\right)\rmd r,\nonumber
\end{align}
which along with the Gronwall lemma implies that \eqref{0911:03} holds.
\end{proof}

By Lemma \ref{0302-1}, one sees that for any $t \geq 0$, $P_t^*$ maps $\mathcal{P}_2(H)$ into itself.
We now show that this mapping is continuous.

\begin{lem}\label{lem03:25}
Suppose that {\bf(H1)}, {\bf(H2$'$)}, {\bf(H3)} and {\bf(H4$'$)}  hold.
Then for any $\mu,\mu'\in\cP_2(H)$,
\begin{align}\label{0325:02}
\cW_2(P^*_{t-s}\mu,P^*_{t-s}\mu')^2
\leq\cW_2(\mu,\mu')^2e^{-(2\lambda*-C_1- C_2)(t-s)}.
\end{align} 
In particular, $\left(\cP_2(H),\{P_t^*\}_{t\geq0}\right)$ is a continuous semi-dynamical system. 

Assume further that  $C_1+C_2<2\lambda_*$. Then there exists a unique invariant measure 
$\mu_\infty$ for $\{P_t^*\}_{t\geq0}$,
and satisfies that for any $t>s$ and $\mu\in\cP_2(H)$,
\begin{equation}\label{0325:06}
\cW_2(P^*_{t-s}\mu,\mu_\infty)
\leq e^{\frac{-(2\lambda_*-C_1-C_2)(t-s)}{2}} \cW_2(\mu,\mu_\infty).
\end{equation} 
\end{lem}
\begin{proof}
For any $s<t$ and $\mu,\mu'\in\cP_2(H)$, take $Y_s,\Tilde{Y}_s\in L^2(\Omega,\mP;H)$ such that $\sL_{Y_s}=\mu$, 
$\sL_{\Tilde{Y}_s}=\mu'$ and 
$$
E\|Y_s-\Tilde{Y}_s\|^2=\cW_2(\mu,\mu')^2.
$$
Let $Y_t$ and $\Tilde{Y}_t$ be two solutions to
\begin{equation}\label{0325:01}
\rmd Y_t=A(Y_t,\sL_{Y_t})\rmd t+\rmd W_t, \quad t>s
\end{equation}
with initial value $Y_s$ and $\Tilde{Y}_s$ respectively. Then we have
\begin{align*}
\mE\|Y_{t}-\Tilde{Y}_t\|^2
&
=\mE\|Y_s-\Tilde{Y}_s\|^2+2\mE\int_s^{t}\,_{V^*}\langle A(Y_r,\sL_{Y_r})-A(\Tilde{Y}_r,\sL_{\Tilde{Y}_r}),Y_r-\Tilde{Y}_r\rangle_{V}\rmd r\\
&
\leq \mE\|Y_s-\Tilde{Y}_s\|^2+\int_s^{t}\left(-2\lambda_*\mE\|Y_r-\Tilde{Y}_r\|^2_{H}
+C_1\cW_2(\sL_{Y_r},\sL_{\Tilde{Y}_r})^2+ 
C_2\mE\|Y_r-\Tilde{Y}_r\|^2_{H}\right)\rmd r\\
&
\leq \mE\|Y_s-\Tilde{Y}_s\|^2-\int_s^{t}\big(2\lambda_*-C_1- C_2\big)
\mE\|Y_r-\Tilde{Y}_r\|_{H}^2\rmd r,
\end{align*}
which along with Gronwall inequality implies that
\begin{align*}
\cW_2(P^*_{t-s}\mu,P^*_{t-s}\mu')^2
&
\leq \mE\|Y_{t}-\Tilde{Y}_t\|^2\nonumber\\
&
\leq \big(\mE\|Y_s-\Tilde{Y}_s\|^2\big){\rm e}^{-(2\lambda*-C_1- C_2)(t-s)}\\
&
=\cW_2(\mu,\mu')^2e^{-(2\lambda*-C_1- C_2)(t-s)}.\nonumber\
\end{align*} 
Hence, $\left(\cP_2(H),\{P_t^*\}_{t\geq0}\right)$ is a continuous semi-DS.

Note that with the help of \eqref{0325:02} and \eqref{0911:03}, we have for any 
$s<\tau<0$ and $t>\tau$,
\begin{equation*}\label{0325:03}
\mE\|Y_{s,t}(0)-Y_{\tau,t}(0)\|^2
\leq \mE\|Y_{s,\tau}(0)\|^2{\rm e}^{-(2\lambda*-C_1- C_2)(t-\tau)}
\leq C{\rm e}^{-(2\lambda*-C_1- C_2)(t-\tau)}.
\end{equation*}
Hence, for any $\tau>s$,
\begin{equation}\label{0325:05}
\lim_{\tau\rightarrow-\infty}\sup_{t\geq0}
\cW_2\left(P_{t-s}^*\delta_0,P_{t-\tau}^*\delta_0\right)=0,
\end{equation}
which implies that there exists $\mu_\infty\in\cP_2(H)$ such that 
\begin{equation}\label{0325:04}
\lim_{t\rightarrow+\infty}\cW_2(P_t^*\delta_0,\mu_\infty)=0.
\end{equation}
Therefore, \eqref{0325:05}, \eqref{0325:04}, \eqref{0911:03} and \eqref{0325:02} yield that for any $s\geq0$,
\begin{align*}
\cW_2(P_s^*\mu_\infty,\mu_\infty)
\leq \lim_{t\rightarrow+\infty}\cW_2(P_{s}^*P_t^*\delta_0,P_t^*\delta_0)
\leq \lim_{t\rightarrow+\infty}\cW_2(P_{s+t}^*\delta_0,P_{t}^*\delta_0)
\leq \lim_{t\rightarrow+\infty}C{\rm e}^{-(2\lambda_*-C_1-C_2)t}=0.
\end{align*}
Then $\mu_\infty$ is an invariant measure. Its uniqueness, as well as \eqref{0325:06}, follow from \eqref{0325:02}.
It is straightforward to show that for any $t,s\in\R_+$ and $\mu\in\cP_2(H)$,
$P_0^*\mu=\mu$ and $P_{t+s}^*\mu=P_{t}^*P_{s}^*\mu$.
Combing \eqref{0325:02}, \eqref{0325:06} and Lemma \ref{EGA},
we complete the proof.
\end{proof}

The following result is a direct consequence of Lemma \ref{lem03:25}.
\begin{cor}
Suppose that {\bf(H1)}, {\bf(H2$'$)}, {\bf(H3)} and {\bf(H4$'$)}  hold.
Assume further that $C_1+C_2<2\lambda_*$. Then
$\left(\cP_2(H),\{P_t^*\}_{t\geq0}\right)$ 
admits a $\sS$-global attractor
$\cP^*:=\{\mu_\infty\}$. 
\end{cor}

With the help of \eqref{0911:03}, for any $x\in H$ and $\mu\in\cP_2(H)$,
there exists a unique solution $X_t,t\geq s$ to \eqref{0327:01-1}
with $X_{s}=x$ and $\sL_{Y_s}=\mu$.
Combining {\bf{(H1)}}, \eqref{0327:02}, \eqref{1001:01}, \eqref{0327:04}
and Lemma \ref{EUsolution}, we have the following lemma.
\begin{lem}\label{lem0916}
Let $\mu\in\cP_2(H)$ and $\mP$-a.s. $\omega\in\Omega$. 
For any $s\in\R$, $T>0$ and $z_0\in H$, there exists a unique solution
$Z(t,s,\omega,\mu)z_0$ to 
\begin{equation}\label{0327:03}
\frac{\rmd Z_t}{\rmd t}=A(Z_t+W_t(\omega),\mu),\quad t\in[s,s+T],\quad Z_s=z_0
\end{equation}
with $Z(\cdot,s,\omega,\mu)z_0\in L^2\left([s,s+T];V\right)\cap C\left([s,s+T];H\right)$.
\end{lem}

In order to show that $\Phi$ is a continuous RDS, we first establish the following lemma.
\begin{lem}\label{lem0917}
Suppose that {\bf(H1)}, {\bf(H2$'$)}, {\bf(H3)} and {\bf(H4$'$)}  hold.
Assume further that $C_1+C_2<2\lambda_*$. Then for any $x,x'\in H$, $\mu,\mu'\in\cP_2(H)$ and
$s_1<s_2<s\leq t$,
\begin{equation}\label{0603:02}
\|X_{s_1,t}^{x,\mu}(\omega) - X_{s_2,t}^{x',\mu'}(\omega)\|^2
\leq\left(\|x-x'\|^2
+\frac{C_1}{2\lambda_*-C_1-C_2}
\cW_2(\mu,\mu')\right){\rm e}^{-(2\lambda_*-C_2)(t-s)}.
\end{equation}
\end{lem}
\begin{proof}
By It\^o's formula, \eqref{0327:02} and \eqref{0325:06}, one sees that for any $x,x'\in H$, $\mu,\mu'\in\cP_2(H)$ and
$s_1 \leq s_2 \leq s \leq t$,
\begin{align}\label{0603:01}
& 
\|X_{s_1,t}^{x,\mu}(\omega) - X_{s_2,t}^{x',\mu'}(\omega)\|^2\\\nonumber
&
= \|X_{s_1,s}^{x,\mu}(\omega) - X_{s_2,s}^{x',\mu'}(\omega)\|^2\\\nonumber
&\quad 
+ 2\int_{s}^{t} 
\,_{V^*}\langle 
A(X_{s_1,r}^{x,\mu}(\omega), P_{r-s_1}^*\mu)  - A(X_{s_2,r}^{x',\mu}(\omega), P_{r-s_2}^*\mu'), X_{s_1,r}^{x,\mu}(\omega) - X_{s_2,r}^{x',\mu'}(\omega)
\rangle_V 
\rmd r\\\nonumber
&
\leq \|X_{s_1,s}^{x,\mu}(\omega) - X_{s_2,s}^{x',\mu'}(\omega)\|^2 
+\int_{s}^{t}\left( -(2\lambda_*-C_2)\| X_{s_1,r}^{x,\mu}(\omega) - X_{s_2,r}^{x',\mu'}(\omega)\|^2 
+C_1\cW_2(P_{r-s_1}^*\mu,P_{r-s_2}^*\mu')^2\right)
\rmd r.
\end{align}
Letting $s_1=s_2=s$ in \eqref{0603:01}, in view of
\eqref{0325:02}, we have
\begin{align*}
& 
\|X_{s,t}^{x,\mu}(\omega) - X_{s,t}^{x',\mu'}(\omega)\|^2\\\nonumber
&
\leq \|x - x'\|^2 
+\int_{s}^{t}\left( -(2\lambda_*-C_2)\| X_{s,r}^{x,\mu}(\omega) - X_{s,r}^{x',\mu'}(\omega)\|^2 
+C_1\cW_2(\mu,\mu'){\rm e}^{-(2\lambda_*-C_1-C_2)(r-s)}
\right)\rmd r\\
&
\leq \|x - x'\|^2 
+\int_{s}^{t} -(2\lambda_*-C_2)\| X_{s,r}^{x,\mu}(\omega) - X_{s,r}^{x',\mu'}(\omega)\|^2 \rmd r
+\frac{C_1}{2\lambda_*-C_1-C_2}\cW_2(\mu,\mu'),
\end{align*}
which along with Gronwall's lemma completes the proof.
\end{proof}

Now it is the position to prove Theorem \ref{thmRDeq}.
\begin{proof}[Proof of Theorem \ref{thmRDeq}]
According to Lemmas \ref{lem0916} and \ref{lem0917} and \eqref{0325:02}, we obtain that 
$$
\Phi:\R\times\Omega\times\cP_2(H)\times H\longrightarrow\cP_2(H)\times H
$$
is a continuous RDS.

Note that \eqref{0328:01}, \eqref{0328:02} and \eqref{0327:04} yield that for any $\varepsilon_1,\varepsilon_2\in(0,1)$, $u\in H$ and $\mu\in\cP_2(H)$,
\begin{align*}
&
2\,_{V^*}\langle A(u+W_r(\omega), \mu), u\rangle_V \\
&
=2\,_{V^*}\langle A(u+W_r(\omega), \mu), u+W_r(\omega)-W_r(\omega)\rangle_V\\
&\leq 
2\varepsilon_1\,_{V^*}\langle A(u+W_r(\omega), \mu), u+W_r(\omega)\rangle_V
+2(1-\varepsilon_1)_{V^*}\langle A(u+W_r(\omega), \mu), u+W_r(\omega)\rangle_V\\
&\quad
+2\|A(u+W_r(\omega), \mu)\|_{V^*}\|W_r(\omega)\|_V\\
&\leq 
-2\varepsilon_1\|u+W_r(\omega)\|_V^2
+\varepsilon_1(2+\lambda_1)\|u+W_r(\omega)\|^2+\varepsilon_1\lambda_2\cW_2(\mu,\delta_0)^2
-(1-\varepsilon_1)\eta_1\|u\|^2 \\
&
\quad +(1-\varepsilon_1)\eta_2\cW_2(\mu,\delta_0)^2+\varepsilon_2\|A(u+W_r(\omega), \mu)\|_{V^*}^2+C_{\varepsilon_2}\|W_r(\omega)\|_V^2
+C_{\varepsilon_1,\varepsilon_2}\\
&
\leq -2\varepsilon_1\|u+W_r(\omega)\|_V^2
+\varepsilon_1(2+\lambda_1)\|u+W_r(\omega)\|^2+\varepsilon_1\lambda_2\cW_2(\mu,\delta_0)^2
-(1-\varepsilon_1)\eta_1\|u\|^2\\
&\quad
+(1-\varepsilon_1)\eta_2\cW_2(\mu,\delta_0)^2+\varepsilon_2 C\|u+W_r(\omega)\|_V^2
+\varepsilon_2 C\cW_2(\mu,\delta_0)^2+C_{\varepsilon_2}\|W_r(\omega)\|_V^2
+C_{\varepsilon_1,\varepsilon_2}\\
&
\leq-(2\varepsilon_1-\varepsilon_2 C)\|u+W_r(\omega)\|_V^2
-((1-\varepsilon_1)\eta_1-2\varepsilon_1(2+\lambda_1))\|u\|^2\\
&\quad
+((1-\varepsilon_1)\eta_2+\varepsilon_1\lambda_2+\varepsilon_2C)\cW_2(\mu,\delta_0)^2
+2\varepsilon_1(2+\lambda_1)\|W_r(\omega)\|^2
+C_{\varepsilon_2}\|W_r(\omega)\|_V^2+C_{\varepsilon_1,\varepsilon_2},
\end{align*}
which, by letting
$$
0<\varepsilon_1<\frac{\eta_1-\eta_2}{\eta_1-\eta_2+4+2\lambda_1+\lambda_2},
\quad 0<\varepsilon_2<\frac{[(1-\varepsilon_1)(\eta_1-\eta_2)
-\varepsilon_1(4+2\lambda_1+\lambda_2)]\wedge(2\varepsilon_1)}{C},
$$ 
and
$$
\eta_1':=(1-\varepsilon_1)\eta_1-2\varepsilon_1(2+\lambda_1)>
\eta_2':=(1-\varepsilon_1)\eta_2+\varepsilon_1\lambda_2+\varepsilon_2C>0,
$$ 
implies that
\begin{equation}\label{0329:01}
2\,_{V^*}\langle A(u+W_r(\omega), \mu), u\rangle_V
\leq -\eta_1'\|u\|^2+\eta_2'\cW_2(\mu,\delta_0)^2
+C\left(1+\|W_r(\omega)\|_V^2\right).
\end{equation}

By the It\^o formula, \eqref{0329:01} and \eqref{0911:03}, one sees that
for any $t>s$
\begin{align}\label{0329:02}
\|Z_{s,t}^{x,\mu}(\omega)\|^2
&
=\|x-W_{s}(\omega)\|^2+\int_{s}^{t}
2_{V^*}\langle 
A(Z_{s,r}^{x,\mu}(\omega)+W_r(\omega), P_{r-s}^*\mu), Z_{s,r}^{x,\mu}(\omega) 
\rangle_V \rmd r\nonumber\\
&
\leq\|x-W_{s}(\omega)\|^2+\int_{s}^{t}\left(
-\eta_1'\|Z_{s,r}^{x,\mu}(\omega)\|^2+\eta_2'\cW_2(P_{r-s}^*\mu,\delta_0)^2+C(r,\omega)\right)
\rmd r\nonumber\\
&
\leq \|x-W_{s}(\omega)\|^2+\int_{s}^{t}\left(-\eta_1'\|Z_{s,r}^{x,\mu}(\omega)\|^2
+\eta_2'\mu(\|\cdot\|^2){\rm e}^{-(\eta_1-\eta_2)(r-s)}+C(r,\omega) \right)
\rmd r,
\end{align} 
where $C(r,\omega):=C\left(1+\|W_r(\omega)\|_V^2\right)$.
Then in view of \eqref{0329:02} and Gronwall's inequality, we have
\begin{equation*}
\|Z_{s,t}^{x,\mu}(\omega)\|^2
\leq \|x-W_{s}(\omega)\|^2{\rm e}^{-\eta_1'(t-s)}
+C\mu(\|\cdot\|^2){\rm e}^{-[\eta_1'\wedge(\eta_1-\eta_2)](t-s)}
+\int_s^t{\rm e}^{-\eta_1'(t-r)}C(r,\omega)\rmd r,
\end{equation*}
which implies that
\begin{align}\label{0329:03}
\|X_{s,t}^{x,\mu}(\omega)\|^2
&
\leq2\|Z_{s,t}^{x,\mu}(\omega)\|^2+2\|W_t(\omega)\|^2\nonumber\\
&
\leq 2\left(\|x-W_{s}(\omega)\|^2{\rm e}^{-\eta_1'(t-s)}
+C\mu(\|\cdot\|^2){\rm e}^{-[\eta_1'\wedge(\eta_1-\eta_2)](t-s)}\right)\nonumber\\
&\quad
+2\int_s^t{\rm e}^{-\eta_1'(t-r)}C(r,\omega)\rmd r
+2\|W_t(\omega)\|^2.
\end{align}

Define $\eta':=2\lambda_*-C_1-C_2$. By \eqref{0603:01} and \eqref{0325:06},
one sees that for any $s_1\leq s_2\leq s\leq t$,
\begin{align*}
& 
\|X_{s_1,t}^{x,\mu}(\omega) - X_{s_2,t}^{x',\mu'}(\omega)\|^2\\
&
\leq \|X_{s_1,s}^{x,\mu}(\omega) - X_{s_2,s}^{x',\mu'}(\omega)\|^2 
-\int_{s}^{t}(2\lambda_*-C_2)\| X_{s_1,r}^{x,\mu}(\omega) - X_{s_2,r}^{x',\mu'}(\omega)\|^2 
\rmd r\\
&\quad
+C \int_{s}^{t} \left(
\cW_2(P_{r-s_1}^*\mu,\mu_\infty)^2+\cW_2(P_{r-s_2}^*\mu',\mu_\infty)^2\right)\rmd r\\
&
\leq \|X_{s_1,s}^{x,\mu}(\omega) - X_{s_2,s}^{x',\mu'}(\omega)\|^2 -\int_{s}^{t}(2\lambda_*-C_2)
\| X_{s_1,r}^{x,\mu}(\omega) - X_{s_2,r}^{x',\mu'}(\omega)\|^2 
\rmd r\\
&\quad+C \int_{s}^{t}\left(
{\rm e}^{-\eta'(r-s_1)}\cW_2(\mu,\mu_\infty)^2
+{\rm e}^{-\eta'(r-s_2)}\cW_2(\mu',\mu_\infty)^2\right)\rmd r\\
&
\leq \|X_{s_1,s}^{x,\mu}(\omega) - X_{s_2,s}^{x',\mu'}(\omega)\|^2 -\int_{s}^{t}(2\lambda_*-C_2)
\| X_{s_1,r}^{x,\mu}(\omega) - X_{s_2,r}^{x',\mu'}(\omega)\|^2 \rmd r\\
&\quad
+C\left({\rm e}^{-\eta'(s-s_1)}\cW_2(\mu,\mu_\infty)^2
+{\rm e}^{-\eta'(s-s_2)}\cW_2(\mu',\mu_\infty)^2\right),
\end{align*}
which along with the Gronwall inequality implies that
\begin{align}\label{0329:04}
\|X_{s_1,t}^{x,\mu}(\omega) - X_{s_2,t}^{x',\mu'}(\omega)\|^2
&
\leq {\rm e}^{-(2\lambda_*-C_2)(t-s)}\bigg(
\|X_{s_1,s}^{x,\mu}(\omega) - X_{s_2,s}^{x',\mu'}(\omega)\|^2\\
&\qquad
+C\left({\rm e}^{-\eta'(s-s_1)}\cW_2(\mu,\mu_\infty)^2
+{\rm e}^{-\eta'(s-s_2)}\cW_2(\mu',\mu_\infty)^2\right)\bigg).\nonumber
\end{align}
Combining \eqref{0329:03}, \eqref{0329:04} with $s=s_2$ and 
the subexponential growth of $\|W_r(\omega)\|_V$, we obtain
\begin{align*}
&
\|X_{s_1,t}^{x,\mu}(\omega) - X_{s_2,t}^{x',\mu'}(\omega)\|^2\\
&
\leq {\rm e}^{-(2\lambda_*-C_2)(t-s_2)}\bigg(
\|X_{s_1,s_2}^{x,\mu}(\omega) - x'\|^2
+C\left({\rm e}^{-\eta'(s_2-s_1)}\cW_2(\mu,\mu_\infty)^2
+\cW_2(\mu',\mu_\infty)^2\right)\bigg)\\
&
\leq C{\rm e}^{-(2\lambda_*-C_2)(t-s_2)}\bigg(\|x-W_{s_1}(\omega)\|^2
{\rm e}^{-\eta_1'(s_2-s_1)}+C\mu(\|\cdot\|^2){\rm e}^{-[\eta_1'\wedge(\eta_1-\eta_2)](s_2-s_1)}
+\|x'\|^2\\
&\qquad
+\int_{s_1}^{s_2}{\rm e}^{-\eta_1'(s_2-r)}C(r,\omega)\rmd r
+\|W_{s_2}(\omega)\|^2
+\cW_2(\mu,\mu_\infty)^2+\cW_2(\mu',\mu_\infty)^2\bigg)\rightarrow 0
\quad {\text{as}}~s_2\rightarrow-\infty. 
\end{align*}
Hence, for any $t\in\R$ and $\omega\in\Omega$, there exists $\xi_t(\omega)$ such that 
$$
\lim_{s\rightarrow-\infty}X_{s,t}^{x,\mu}(\omega)=\xi_t(\omega)
$$
uniformly with respect to $x$ on any bounded subset $H$
and $\mu$ on any bounded subset $\cP_2(H)$.
And it holds that for any $t\in\R$,
$$
\xi_t(\omega)=\lim_{s\to-\infty}X_{s,t}^{x,\mu}(\omega)
=\lim_{s\to-\infty}X_{s-t,0}^{x,\mu}(\theta_t\omega)=\xi_0(\theta_t\omega)
=:\xi(\theta_t\omega).
$$
In particular, by \eqref{0329:03}, we have
\begin{align*}
\|\xi(\theta_{-t}\omega)\|^2
&
=\lim_{s\rightarrow-\infty}\|X_{s,0}^{0,\mu}(\theta_{-t}\omega)\|^2\\
&
\leq\lim_{s\rightarrow-\infty}\left(2\|W_s(\theta_{-t}\omega)\|^2{\rm e}^{\eta_1's}
+C\mu(\|\cdot\|^2){\rm e}^{[\eta_1'\wedge(\eta_1-\eta_2)]s}
+\int_s^0{\rm e}^{\eta'r}C(r,\theta_{-t}\omega)\rmd r
+2\|W_0(\theta_{-t}\omega)\|^2\right)\\
&
\leq \int_{-\infty}^0{\rm e}^{\eta'r}C(r,\theta_{-t}\omega)\rmd r
+2\|W_0(\theta_{-t}\omega)\|^2,
\end{align*}
which implies that $\{\xi\}\in\cD_{\lambda_0}$ with $\lambda_0<2\lambda_*-C_1-C_2$, 
where $\xi(\omega)$ is a singleton for 
each $\omega\in\Omega$.
With the help of \eqref{0329:04}, we have for any $D\in\cD_{\lambda_0'}$ with 
$\lambda_0'=2\lambda_*-C_2$ and bounded subset 
$B_{\cP}\subset \cP_2(H)$,
\begin{align}\label{1106:01}
&
\sup_{\mu\in B_{\cP}}\dist_{H}\left(\varphi(t,\theta_{-t}\omega,\mu)D(\theta_{-t}\omega),
\xi(\omega)\right)\nonumber\\
&
=\sup_{\mu\in B_{\cP}}\sup_{x(\theta_{-t}\omega)\in D(\theta_{-t}\omega)}
\left\|\varphi(t,\theta_{-t}\omega,\mu)x(\theta_{-t}\omega)-\xi(\omega)\right\|\nonumber\\
&
=\sup_{\mu\in B_{\cP}}\sup_{x(\theta_{-t}\omega)\in D(\theta_{-t}\omega)}
\left\|X_{-t,0}^{x(\theta_{-t}\omega),\mu}(\omega)-\xi(\omega)\right\|\nonumber\\
&
\leq\sup_{\mu\in B_{\cP}}\sup_{x(\theta_{-t}\omega)\in D(\theta_{-t}\omega)}
 \left[C{\rm e}^{-(2\lambda_*-C_2)t}\left(\|x(\theta_{-t}\omega)\|^2
 +\|\xi(\theta_{-t}\omega)\|^2+\mu(\|\cdot\|^2)
 +\mu_\infty(\|\cdot\|^2)\right)\right]^{\frac12}\nonumber\\
&
\leq C{\rm e}^{\frac{-(2\lambda_*-C_2)t}{2}}\left[\|D(\theta_{-t}\omega)\|^2
+\|\xi(\theta_{-t}\omega)\|^2+1\right]^{\frac12}
\rightarrow 0\quad {\text{as }} t\rightarrow\infty,
\end{align}
which along with the neighborhood closed property of $\cD_{\lambda_0}$
and $\sS$ implies that there exists $\epsilon_0>0$ such that 
for any $B_{\cP}\in\sS$ and $D\in\cD_{\lambda_0}$, there exists
$T:=(B_{\cP},D,\omega,\epsilon_0)>0$ so that
$$
\Phi(t,\theta_{-t}\omega)B_{\cP}\times D(\theta_{-t}\omega)
=\left(P_{t}^*B_{\cP},\varphi(t,\theta_{-t}\omega,B_{\cP})D(\theta_{-t}\omega)\right)
\subset B_{\cP}^{\epsilon_0}\times \overline{U_{\epsilon_0}(\xi(\omega))},
$$
where 
$$
B_{\cP}^{\epsilon_0}:=\left\{\nu\in\cP_2(H):\cW_2(\nu,\mu_\infty)\leq\epsilon_0\right\},
\quad
U_{\epsilon_0}(\xi(\omega)):=\left\{x\in H:\dist(x,\xi(\omega))<\epsilon_0\right\}.
$$
By Proposition \ref{prop0117}, \eqref{0325:06} and \eqref{1106:01}, 
we obtain that $\Phi$ admits a 
$\sD_{\lambda_0}$-random attractor $\sA$, which is defined by
$$
\sA(\omega)=\Omega(B_{\cP}^{\epsilon_0}\times \overline{U_{\epsilon_0}(\xi(\omega))})
=(\mu_\infty,\xi(\omega)).
$$

Since $\cP^*$ is the $\sS$-global attractor of $\{P_t^*\}_{t\geq0}$,
we can define the backward extension of $P_t^*$ on $\cP^*$, which is defined by 
$P_{-t}^*\mu_\infty=\mu_\infty$ for any $t\geq0$.
By Theorem \ref{thm0330}, we can obtain the pullback attractor $\cA_H$ of
$
\varphi:\R_+\times\Omega\times\{\mu_\infty\}\times H\rightarrow H
$
over $(\Omega,\cF,\mP,\{\theta_t\}_{t\in\R})$ and $(\cP^*,\{P_t^*\}_{t\in\R})$. Moreover,
$$
\cA(\omega,\mu_\infty):=\{\xi(\omega)\}.
$$
Note that by the uniqueness of solutions, we have 
for any $\mu\in\cP_2(H)$,
\begin{align*}
\mu_\infty=\lim_{s\rightarrow-\infty}P_{s,0}^*\mu
=\lim_{s\rightarrow-\infty}\int_H\mP\circ[X_{s,0}^{x,\mu}]^{-1}\mu(\rmd x)
=\int_H\lim_{s\rightarrow-\infty}\mP\circ[X_{s,0}^{x,\mu}]^{-1}\mu(\rmd x)
=\mP\circ[\xi]^{-1},
\end{align*}
which completes the proof.
\end{proof}

\subsection{McKean-Vlasov stochastic 2D Navier-Stokes equation}\label{secNSeq}

Note that there exists $C>0$ such that
for any $u,v,w\in V$,
\begin{equation}\label{0421:06}
~_{V^*}\langle B(u,v),w\rangle_V=-_{V^*}\langle B(u,w),v\rangle_V,
\end{equation}
\begin{equation}\label{0421:07}
~_{V^*}\langle B(u,v),v\rangle_V=0,
\end{equation}
\begin{equation}\label{0911:04}
\left|\,_{V^*}\langle B(u,v),w\rangle_V\right|
\leq C\|u\|^{\frac{1}{2}}\|u\|_V^{\frac{1}{2}}
\|v\|^{\frac{1}{2}}\|v\|_V^{\frac{1}{2}}\|w\|_V,
\end{equation}
and for any $u\in\cD(A)$, $v\in V$ and $w\in H$,
\begin{equation}\label{0421:08}
\left|\,_{V^*}\langle B(u,v),w\rangle_V\right|
\leq C\|u\|^{\frac12}\|Au\|^{\frac12}\|v\|_V\|w\|.
\end{equation}
Note that there exists $C>0$ such that for all $u,v\in V$,
\begin{equation}\label{0424:02}
2_{V^*}\langle A(u-v)+B(u,u)-B(v,v),u-v\rangle_V
\leq -\nu_c\|u-v\|_V^2+C\|v\|_V^2\|v\|^2\|u-v\|^2.
\end{equation}

In what follows, define $B(u):=B(u,u)$ for simplicity.
And assume that $\beta=2$, where $\beta$ is as in {\bf(H2)}
and {\bf(H4)}.

\begin{lem}\label{lem0418}
Assume that {\bf(H3)} holds and $\lambda_1+\lambda_2<2\gamma^2\nu_c$. 
Then for any $\lambda_1'\in[0,2\nu_c-\tfrac{\lambda_1+\lambda_2}{\gamma^2})$
there exist $\lambda_2'>\lambda_2$ and $C>0$
such that for any $u\in V$ and $\mu\in\cP_2(H)$,
\begin{equation}\label{0401:01}
2_{V^*}\langle Au+B(u)+F(u,\mu),u\rangle_V
\leq-\lambda_1'\|u\|_V^2-\lambda_2'\|u\|^2+\lambda_2\mu(\|\cdot\|^2)+C. 
\end{equation}
Moreover, suppose further that {\bf(H1)}, {\bf(H2)}, {\bf(H2$''$)} and {\bf(H4)}  hold. 
Then for any $p\in(4\vee\kappa,+\infty)$ and $\lambda'\in(0,\frac{p(\lambda_2'-\lambda_2)}{2})$ there exists a constant $C>0$ such that
for all $t\geq s$,
\begin{equation}\label{0401:02}
\mE\|Y_t\|^p\leq\mE\|Y_s\|^p{\rm e}^{-\lambda'(t-s)}+C.
\end{equation} 
\end{lem}
\begin{proof}
It follows from {\bf(H3)} and \eqref{0421:07} that 
for any $\lambda_1'\in(0,2\nu_c)$, $u\in V$ and $\mu\in\cP_2(H)$,
\begin{align*}
2_{V^*}\langle Au+B(u)+F(u,\mu),u\rangle_V
&
\leq-2\nu_c\|u\|_V^2+\lambda_1\|u\|^2+\lambda_2\mu(\|\cdot\|^2)+C\nonumber\\
&
\leq -\lambda_1'\|u\|_V^2-\left((2\nu_c-\lambda_1')\gamma^2-\lambda_1\right)\|u\|^2
+\lambda_2\mu(\|\cdot\|^2)+C\nonumber\\
&
\leq -\lambda_1'\|u\|_V^2-\lambda_2'\|u\|^2+\lambda_2\mu(\|\cdot\|^2)+C
\end{align*}
by letting 
$\lambda_1'\in[0,2\nu_c-\tfrac{\lambda_1+\lambda_2}{\gamma^2})$
and $\lambda_2':=(2\nu_c-\lambda_1')\gamma^2-\lambda_1$.

By It\^o's formula, \eqref{0401:01} and Young's inequality, one sees that 
for any $p>4\vee\kappa$ and $0<\lambda'<\frac{p(\lambda_2'-\lambda_2)}{2}$,
\begin{align*}
\mE\|Y_t\|^p
=&
\mE\|Y_s\|^p+\frac{p}{2}\mE\int_s^t\|Y_r\|^{p-2}\left(2_{V^*}\langle AY_r+B(Y_r)+F(Y_r,\sL_{Y_r}),Y_r\rangle_V
+\sum_{i=1}^{d}\|\phi_i\|^2\right)\rmd r\\
&
+\frac{p(p-2)}{2}\sum_{i=1}^d\int_s^t\mE\|Y_r\|^{p-2}\|\phi_i\|^2\rmd r\\
\leq &
\mE\|Y_s\|^p+\frac{p}{2}\int_s^t\left(-\lambda_2'\mE\|Y_r\|^p+\lambda_2\mE\|Y_r\|^{p-2}\mE\|Y_r\|^2
+C\right)\rmd r\\
&
+\frac{p(p-2)}{2}\sum_{i=1}^d\int_s^t\mE\|Y_r\|^{p-2}\|\phi_i\|^2\rmd r\\
\leq&
 \mE\|Y_s\|^p+\int_s^t\left(-\lambda'\mE\|Y_r\|^p+C\right)\rmd r,
\end{align*}
which along with Gronwall's lemma implies that 
\begin{equation*}
\mE\|Y_t\|^p\leq \mE\|Y_s\|^p{\rm e}^{-\lambda'(t-s)}+C.
\end{equation*}
\end{proof}

By the definition of $\{P_t^*\}_{t\geq0}$ and Lemma \ref{lem0418}, one sees that 
for any $\mu\in\cP_p(H)$ and $t,s>0$, $P_t^*\mu\in\cP_p(H)$,
$$
P_0^*\mu=\mu, \quad{\text{and}} \quad P_{t+s}^*\mu=P_t^*P_s^*\mu.
$$
We say that a subset $B\subset\cP_p(H)$ is bounded in this subsection
if there exists a constant $R>0$ such that $\mu(\|\cdot\|^p)\leq R$ for all $\mu\in B$.
Now we show that $P_t^*$ is continuous over any bounded subset of $\cP_p(H)$.
\begin{lem}
Assume that {\bf(H1)}--{\bf(H3)}, {\bf(H2$''$)} and {\bf(H4)}  hold.
Let $B$ be a bounded subset of $\cP_p(H)$ with $p>4\vee\kappa$.
Suppose that $\mu\in B$ and $\{\mu_n\}\subset B$ satisfy
$\lim_{n\rightarrow\infty}\cW_2(\mu_n,\mu)=0$.
Then for any $T>0$,
\begin{align}\label{0423:01}
\lim_{n\rightarrow\infty}\sup_{t\in[0,T]}
\cW_2\left(P_{t}^*\mu_n,P_{t}^*\mu\right)=0.
\end{align}
\end{lem}
\begin{proof}
Let ${\xi_n},\xi\in L^p(\Omega,\mP)$ such that $\sL_{\xi_n}={\mu_n}$, $\sL_\xi=\mu$
and $W_2({\mu_n},\mu)^2=\mE\|{\xi_n}-\xi\|^2$.
Denote by $Y_t^\mu,t\geq0$ the solution to \eqref{0331:01} with the initial value $\xi$.
For any $R>0$, define 
$$
\tau_R^n:=\inf\left\{t\in[0,T]:
\left(\|Y_t^{\mu_n}\|+\int_0^t\|Y_s^{\mu_n}\|_V^2\|Y_s^{\mu_n}\|^2\rmd s\right)\vee
\left(\|Y_t^\mu\|+\int_0^t\|Y_s^\mu\|_V^2\|Y_s^\mu\|^2\rmd s\right)\geq R\right\}.
$$
Set $\nu_n(t):=\sL_{Y_t^{\mu_n}}$ and $\nu(t):=\sL_{Y_t^{\mu}}$.
By It\^o's formula, one sees that for all $t\in[0,T]$,
\begin{align*}
\|Y_{t\wedge\tau_R^n}^{\mu_n}-Y_{t\wedge\tau_R^n}^\mu\|^2
&
=\|{\xi_n}-\xi\|^2+\int_0^{t\wedge\tau_R^n}
2_{V^*}\langle A(Y_r^{\mu_n}-Y_r^\mu)+B(Y_r^{\mu_n})-B(Y_r^\mu),
Y_r^{\mu_n}-Y_r^\mu\rangle_V\rmd r\\
&\quad
+\int_0^{t\wedge\tau_R^n}2_{V^*}\langle F(Y_r^{\mu_n},\nu_n(r))-F(Y_r^\mu,\nu(r)),
Y_r^{\mu_n}-Y_r^\mu\rangle_V\rmd r,
\end{align*}
which along with \eqref{0424:02}, {\bf(H2$''$)} and Young's inequality
implies that for any $0\leq \tau_1\leq \tau_2\leq T\wedge\tau_R^n$,
\begin{align*}
\|Y_{t\wedge\tau_2}^{\mu_n}-Y_{t\wedge\tau_2}^\mu\|^2
&
=\|Y_{\tau_1}^{\mu_n}-Y_{\tau_1}^\mu\|^2
+\int_{\tau_1}^{t\wedge\tau_2}
2_{V^*}\langle A(Y_r^{\mu_n}-Y_r^\mu)+B(Y_r^{\mu_n})-B(Y_r^\mu),
Y_r^{\mu_n}-Y_r^\mu\rangle_V\rmd r\\
&\quad
+\int_{\tau_1}^{t\wedge\tau_2}
2_{V^*}\langle F(Y_r^{\mu_n},\nu_n(r))-F(Y_r^\mu,\nu(r)),
Y_r^{\mu_n}-Y_r^\mu\rangle_V\rmd r\\
&
\leq \|Y_{\tau_1}^{\mu_n}-Y_{\tau_1}^\mu\|^2
-\int_{\tau_1}^{t\wedge\tau_2}
\nu_c\|Y_r^{\mu_n}-Y_r^\mu\|_V^2\rmd r\\
&\quad
+\int_{\tau_1}^{t\wedge\tau_2}
\phi(r)\left(\|Y_r^{\mu_n}-Y_r^\mu\|^2
+\cW_{2,T,R,H}\left(\nu_n^r,\nu^r\right)^2\right)
\rmd r\\
&
\leq \|Y_{\tau_1}^{\mu_n}-Y_{\tau_1}^\mu\|^2
+\int_{\tau_1}^{t\wedge\tau_2}\phi(r)\left(\|Y_r^{\mu_n}-Y_r^\mu\|^{2}
+\mE\left(\sup_{s\in[0,r\wedge\tau_R^n]}\|Y_s^{\mu_n}-Y_s^\mu\|^2\right)\right)\rmd r,
\end{align*}
where 
$$
\phi(r):=C\left[\left(1+\|Y_r^{\mu_n}\|_V^2\right)\left(1+\|Y_r^{\mu_n}\|^2\right)
+\left(1+\|Y_r^\mu\|_V^2\right)\left(1+\|Y_r^\mu\|^2\right)
+\mE\|Y_r^{\mu_n}\|^\kappa+\mE\|Y_r^\mu\|^\kappa\right].
$$
Hence,
\begin{align}\label{0423:02}
\mE\left(\sup_{t\in[\tau_1,\tau_2]}\|Y_{t}^{\mu_n}-Y_{t}^\mu\|^2\right)
\leq 
&
\mE\|Y_{\tau_1}^{\mu_n}-Y_{\tau_1}^\mu\|^2\nonumber\\
&
+\mE\int_{\tau_1}^{\tau_2}\phi(r)\left(\|Y_r^{\mu_n}-Y_r^\mu\|^2
+\mE\left(\sup_{s\in[0,r]}\|Y_s^{\mu_n}-Y_s^\mu\|^2\right)\right)\rmd r.
\end{align}

In view of \cite[Proposition 3.6, Lemmas 3.3 and 3.4]{HHL24}, we obtain that for any $p>4\vee\kappa$,
\begin{align}\label{0423:03}
&
\mE\left(\sup_{t\in[0,T]}\|Y_t^{\mu_n}\|^p\right)
+\mE\left(\sup_{t\in[0,T]}\|Y_t^{\mu}\|^p\right)
+\mE\int_0^T\|Y_t^{\mu_n}\|^{p-2}\|Y_t^{\mu_n}\|_V^2\rmd t
+\mE\int_0^T\|Y_t^{\mu}\|^{p-2}\|Y_t^{\mu}\|_V^2\rmd t\nonumber\\
&
\leq C_T\left(1+\mE\|{\xi_n}\|^p+\mE\|{\xi}\|^p\right).
\end{align}
Hence, 
\begin{equation}\label{0904:01}
\mE\int_0^{T\wedge\tau_R^n}\phi(r)
\mE\left(\sup_{s\in[0,r]}\|Y_s^{\mu_n}-Y_s^{\mu}\|^2\right)\rmd r<\infty.
\end{equation}
By the definition of $\tau_R^n$ and \eqref{0401:02}, we have
\begin{equation}\label{0904:02}
\int_0^{T\wedge\tau_R^n}\phi(r)\rmd r\leq C_{R} \quad \mP-a.s.
\end{equation}

Combining \eqref{0423:02}, \eqref{0904:01}, \eqref{0904:02} and stochastic Gronwall's lemma (see Lemma \ref{sGlem}), we have
\begin{align*}
\mE\left(\sup_{t\in[0,T\wedge\tau_R^n]}\|Y_{t}^{\mu_n}-Y_{t}^\mu\|^2\right)
&
\leq C_{T,R}\left(\mE\|{\xi_n}-\xi\|^2+\mE\int_0^{T\wedge\tau_R^n}\phi(r)
\mE\left(\sup_{s\in[0,r\wedge\tau_R^n]}\|Y_s^{\mu_n}-Y_s^\mu\|^2\right)\rmd r\right)\\
&
\leq C_{T,R}\left(\mE\|{\xi_n}-\xi\|^2+\int_0^{T}\mE \phi(r)
\mE\left(\sup_{s\in[0,r\wedge\tau_R^n]}\|Y_s^{\mu_n}-Y_s^\mu\|^2\right)\rmd r\right),
\end{align*}
which along with Gronwall's lemma and \eqref{0423:03} implies that
\begin{align}\label{0424:01}
\mE\left(\sup_{t\in[0,T\wedge\tau_R^n]}\|Y_{t}^{\mu_n}-Y_{t}^\mu\|^2\right)
&
\leq C_{T,R}\mE\|{\xi_n}-\xi\|^2
{\rm e}^{\int_0^TC_{T,R}\mE \phi(r)\rmd r}\nonumber\\
&
\leq C_{T,R}\mE\|{\xi_n}-\xi\|^2
{\rm e}^{\int_0^TC\mE \phi(r)\rmd r}\nonumber\\
&
\leq C_{T,R}\mE\|{\xi_n}-\xi\|^2.
\end{align}
Fix $t\in[0,T]$. For any $R>0$, with the help of \eqref{0424:01}, \eqref{0401:02}, \eqref{0423:03}
and Chebyshev's inequality, we have for any $t\in[0,T]$,
\begin{align*}
\mE\left(\|Y_{t}^{\mu_n}-Y_{t}^\mu\|^2\right)
&
=\mE\left(\|Y_{t}^{\mu_n}-Y_{t}^\mu\|^2\chi_{\{t\leq\tau_R^n\wedge T\}}\right)
+\mE\left(\|Y_{t}^{\mu_n}-Y_{t}^\mu\|^2\chi_{\{t>\tau_R^n\wedge T\}}\right)\\
&
\leq \mE\left(\sup_{t\in[0,\tau_R^n\wedge T]}\|Y_{t}^{\mu_n}-Y_{t}^\mu\|^2\right)
+\left(\mE\|Y_{t}^{\mu_n}-Y_{t}^\mu\|^p\right)^{2/p}
\left(\mE\chi_{\{t>\tau_R^n\}}\right)^{\frac{p-2}{p}}\\
&
\leq C_{TR}\mE\|{\xi_n}-\xi\|^2+CR^{-\frac{p-2}{p}}.
\end{align*}
Hence, we have for any $R>0$,
\begin{align*}
\lim_{n\rightarrow\infty}\sup_{t\in[0,T]}\cW_2\left(P_t^*\mu_n,P_t^*\mu\right)
&
\leq \lim_{n\rightarrow\infty}\sup_{t\in[0,T]}\mE\|Y_t^{\mu_n}-Y_t^\mu\|^2\\
&
\leq\lim_{n\rightarrow\infty}\left(C_{T,R}\cW_2(\mu_n,\mu)+CR^{-\frac{p-2}{p}}\right)\\
&
\leq CR^{-\frac{p-2}{p}}.
\end{align*}
By the arbitrary of $R$, we complete the proof.
\end{proof}

Hence, $\{P_t^*\}_{t\geq0}$ is a semi-DS on $\left(\cP_p(H),\cW_2\right)$,
which is continuous over bounded subset of $\cP_p(H)$.
Therefore, to prove that $\{P_t^*\}_{t\geq0}$ admits a $\sS$-global attractor, 
Lemma \ref{EGA} shows that it suffices to prove the existence of a compact absorbing set.

\begin{lem}\label{lem0418-2}
Assume that {\bf(H1)}--{\bf(H3)}, {\bf(H2$''$)} and {\bf(H4)}  hold.
There exist $\tilde{\lambda}'>0$ and $C>0$ such that
for any $\tilde{p}\geq 2$, $p>4\vee\kappa$ and $s\leq\tau\leq t$,
\begin{equation}\label{0408:02}
\mE\|X_{s,t}^{x,\mu}\|^{\tilde{p}}
\leq \mE\|X_{s,\tau}^{x,\mu}\|^{\tilde{p}}{\rm e}^{-\tilde{\lambda}'(t-\tau)}
+C(\mE\|Y_s\|^{p})^{\frac{\tilde{p}}{p}}{\rm e}^{-\tilde{\lambda}'(t-s)}+C.
\end{equation}
Suppose further that {\bf(H3$'$)} holds.
Then for any $x\in V$ and $s\leq t$,
\begin{equation}\label{0408:04}
\mE\|X_{s,t}^{x,\mu}\|_V^2
\leq C\left(\|x\|_V^2+\mE\|Y_s\|^p+C_{s,t}\right).
\end{equation}

\end{lem}
\begin{proof}
By It\^o's formula, \eqref{0401:01} and Young's inequality, one sees that
for any $\tilde{p}\geq2$, 
$\tilde{\lambda}_2\in(0,\frac{\lambda_2'\tilde{p}}{2})$, 
$p>4\vee\kappa$ and $s\leq\tau<t$,
\begin{align}\label{0408:01}
\mE\|X_{s,t}^{x,\mu}\|^{\tilde{p}}
=&
\mE\|X_{s,\tau}^{x,\mu}\|^{\tilde{p}}
+\mE\int_\tau^t\frac{{\tilde{p}}}{2}\|X_{s,r}^{x,\mu}\|^{{\tilde{p}}-2}
\bigg(2_{V^*}\langle AX_{s,r}^{x,\mu}+B(X_{s,r}^{x,\mu}),X_{s,r}^{x,\mu}\rangle_V\nonumber\\
&\qquad
+2_{V^*}\langle F(X_{s,r}^{x,\mu},\sL_{Y_r}),X_{s,r}^{x,\mu}\rangle_V
+\sum_{i=1}^{d}\|\phi_i\|^2\bigg)\rmd r\nonumber\\
&
+\mE\int_\tau^t\frac{{\tilde{p}}({\tilde{p}}-2)}{2}\|X_{s,r}^{x,\mu}\|^{\tilde{p}-4}\|
(\phi_1,...,\phi_d)^*X_{s,r}^{x,\mu}\|_{U}^2\rmd r\nonumber\\
\leq &
\mE\|X_{s,\tau}^{x,\mu}\|^{\tilde{p}}+\mE\int_\tau^t\frac{\tilde{p}}{2}\|X_{s,r}^{x,\mu}\|^{{\tilde{p}}-2}\bigg(-\lambda_1'\|X_{s,r}^{x,\mu}\|_V^2
-\lambda_2'\|X_{s,r}^{x,\mu}\|^2+\lambda_2\mE\|Y_r\|^2\nonumber\\
&\qquad
+C+\sum_{i=1}^{d}\|\phi_i\|^2\bigg)\rmd r
+\mE\int_\tau^t\frac{{\tilde{p}}({\tilde{p}}-2)}{2}\|X_{s,r}^{x,\mu}\|^{{\tilde{p}}-2}
\sum_{i=1}^{d}\|\phi_i\|^2\rmd r\nonumber\\
\leq&
\mE\|X_{s,\tau}^{x,\mu}\|^{\tilde{p}}
-\frac{\lambda_1'\tilde{p}}{2}\int_\tau^t
\mE\|X_{s,r}^{x,\mu}\|^{{\tilde{p}}-2}\|X_{s,r}^{x,\mu}\|_V^2\rmd r
-\frac{\lambda_2'\tilde{p}}{2}\int_\tau^t
\mE\|X_{s,r}^{x,\mu}\|^{\tilde{p}}\rmd r\nonumber\\
&
+\frac{\lambda_2\tilde{p}}{2}{\tilde{p}}\int_\tau^t\mE\|X_{s,r}^{x,\mu}\|^{{\tilde{p}}-2}\mE\|Y_r\|^2\rmd r
+\int_\tau^t C_{\tilde{p}}\mE\|X_{s,r}^{x,\mu}\|^{{\tilde{p}}-2}\left(\sum_{i=1}^{d}
\|\phi_i\|^2+1\right)\rmd r\nonumber\\
\leq&
\mE\|X_{s,\tau}^{x,\mu}\|^{\tilde{p}}
-\frac{\lambda_1'\tilde{p}}{2}\int_\tau^t
\mE\|X_{s,r}^{x,\mu}\|^{{\tilde{p}}-2}\|X_{s,r}^{x,\mu}\|_V^2\rmd r\nonumber\\
&
-\tilde{\lambda}_2\int_\tau^t\mE\|X_{s,r}^{x,\mu}\|^{\tilde{p}}\rmd r
+C\int_\tau^t\left((\mE\|Y_r\|^2)^{\frac{\tilde{p}}{2}}+1\right)\rmd r,
\end{align}
which along with Gronwall's inequality and \eqref{0401:02} with
$\lambda'\in(0,\frac{\tilde{\lambda}_2p}{\tilde{p}}\wedge
\frac{p(\lambda_2'-\lambda_2)}{2})$ implies that
\begin{align*}
\mE\|X_{s,t}^{x,\mu}\|^{\tilde{p}}
&
\leq \mE\|X_{s,\tau}^{x,\mu}\|^{\tilde{p}}{\rm e}^{-\tilde{\lambda}_2(t-\tau)}+
C\int_\tau^t{\rm e}^{-\tilde{\lambda}_2(t-r)}\left((\mE\|Y_r\|^2)^{\frac{\tilde{p}}{2}}+1\right)\rmd r\nonumber\\
&
\leq\mE\|X_{s,\tau}^{x,\mu}\|^{\tilde{p}}{\rm e}^{-\tilde{\lambda}_2(t-\tau)}+
C\int_\tau^t{\rm e}^{-\tilde{\lambda}_2(t-r)}
\left((\mE\|Y_r\|^p)^{\frac{\tilde{p}}{p}}+1\right)\rmd r
\nonumber\\
&
\leq \mE\|X_{s,\tau}^{x,\mu}\|^{\tilde{p}}{\rm e}^{-\tilde{\lambda}'(t-\tau)}
+C(\mE\|Y_\tau\|^{p})^{\frac{\tilde{p}}{p}}{\rm e}^{-\tilde{\lambda}'(t-\tau)}+C,
\end{align*}
where $\tilde{\lambda}':=\frac{\lambda'\tilde{p}}{p}$.

Employing \eqref{0408:01} and \eqref{0401:02} with $\tau=s$, we also have
\begin{align}\label{0408:03}
\int_s^t\mE\|X_{s,r}^{x,\mu}\|^{{\tilde{p}}-2}\|X_{s,r}^{x,\mu}\|_V^2\rmd r
&
\leq \|x\|^{\tilde{p}}+C\int_s^t(\mE\|Y_r\|^p)^{\frac{\tilde{p}}{p}}\rmd r+C(t-s)\nonumber\\
&
\leq \|x\|^{\tilde{p}}+C(\mE\|Y_s\|^p)^{\frac{\tilde{p}}{p}}+C(t-s).
\end{align}
It follows from It\^o's formula, \cite[Lemma 2.10]{Kuk06}, \eqref{0411:02},
\eqref{0408:02} and \eqref{0401:02}
that for any $s<\tau< t$ and $p>4\vee\kappa$,
\begin{align}\label{0407:01}
\mE\|X_{s,t}^{x,\mu}\|_V^2
=&
\mE\|X_{s,\tau}^{x,\mu}\|_V^2-\mE\int_\tau^t2\langle  AX_{s,r}^{x,\mu},AX_{s,r}^{x,\mu}+B(X_{s,r}^{x,\mu})\rangle\rmd r\nonumber\\
&\qquad
+\mE\int_\tau^t\left(-2\langle AX_{s,r}^{x,\mu},F(X_{s,r}^{x,\mu},\sL_{Y_r})\rangle
+\sum_{i=1}^{d}\|\phi_i\|_V^2\right)\rmd r\nonumber\\
\leq &
\mE\|X_{s,\tau}^{x,\mu}\|_V^2
+\mE\int_\tau^t\bigg(-\frac{2-2\tilde{\lambda}_1}{\nu_c}\|AX_{s,r}^{x,\mu}\|^2
+2\tilde{\lambda}_2\|X_{s,r}^{x,\mu}\|^2+2\tilde{\lambda}_3\mE\|Y_r\|^2+C\bigg)\rmd r\nonumber\\
\leq&
\mE\|X_{s,\tau}^{x,\mu}\|_V^2+C\int_\tau^t\left(\mE\|X_{s,r}^{x,\mu}\|^2+\mE\|Y_r\|^p+1\right)\rmd r\nonumber\\
\leq &
\mE\|X_{s,\tau}^{x,\mu}\|_V^2+C\left(\mE\|X_{s,\tau}^{x,\mu}\|^2
+\mE\|Y_\tau\|^p+1+(t-\tau)\right)
\end{align}
by using $\tilde{\lambda}_1<1$. 
Integrating \eqref{0407:01} with respect to $\tau$ on $(s,t)$, 
by \eqref{0401:02} and \eqref{0408:03} with $\tilde{p}=2$,
we obtain that 
\begin{align*}
\mE\|X_{s,t}^{x,\mu}\|_V^2
&
\leq C\left[\int_s^{t}\left(\mE\|X_{s,\tau}^{x,\mu}\|_V^2+\mE\|Y_\tau\|^p\right)\rmd \tau
+C_{s,t}\right]\nonumber\\
&
\leq C\left(\|x\|_V^2+\mu(\|\cdot\|^p)+C_{s,t}\right).
\end{align*}
\end{proof}
In the following, we consistently assume that $p>4\vee\kappa$. 
For the sake of simplicity, we will not repeatedly state this assumption.
\begin{lem}\label{thmNSeq01}
Under conditions of Lemma \ref{lem0418-2},
$(\cP_p(H),\{P_t^*\}_{t\geq0})$ admits a $\sS$-global attractor $\cP^*$.
\end{lem}
\begin{proof}
It follows from Lemma \ref{lem0418} that for any $r>0$ there exists $t_0>0$ such that 
for any $t>t_0$ and $\mu\in\cP_p(H)$ with $\mu(\|\cdot\|^p)<r$,
$$
P_t^*\mu\left(\|\cdot\|^p\right)<C+1,
$$
where $C$ is as in \eqref{0401:02}. Hence, there exists a bounded absorbing set $B_{\cP}$
for $P_t^*$. Define $K_{\cP}:=\overline{P_1^*B_{\cP}}^{\cW_2}$. 
It is straightforward that $K_{\cP}$ is absorbing.
Combining Lemma \ref{EGA}, we complete the proof by showing that $K_{\cP}$ is compact. Indeed, we can show that
\begin{equation}\label{0421:01}
{\text{for any bounded subset }} B {\text{ and }} s<t,\quad
\overline{P_{s,t}^*B}^{\cW_2} \text{ is compact.}
\end{equation}
For any $R,N>0$, define 
$$
B_R:=\left\{\mu\in\cP_p(H):\mu(\|\cdot\|^p)\leq R\right\}
$$
and
$$
Q_N:=\left\{x\in H:\|x\|_V\leq N\right\},\quad
Q_N^c:=\left\{x\in H:\|x\|_V> N\right\}.
$$
Since $V\subset H$ is compact, $Q_N$ is compact in $H$.
Fix $R>0$. For any $\mu\in B_R$,
by the uniqueness of solutions to \eqref{0331:01} and \eqref{0331:03},
one sees that for any $s<t$,
\begin{align}\label{0418:01}
P_{s,t}^*\mu=\mP\circ[Y_{s,t}^{\mu}]^{-1}
=\int_H\mP\circ[X_{s,t}^{x,\mu}]^{-1}\mu(\rmd x).
\end{align} 
Combining \eqref{0418:01}, Chebyshev's inequality and \eqref{0408:04}, 
we have for any $N>0$ and $s<t$,
\begin{align*}
P_{s,t}^*\mu(Q_N^c)
&
=\int_H\mP\circ[X_{s,t}^{x,\mu}]^{-1}(Q_N^c)\mu(\rmd x)\\\nonumber
&
\leq \frac{1}{N^2}\int_H\mE\|X_{s,t}^{x,\mu}\|_V^2\mu(\rmd x)\\\nonumber
&
\leq \frac{1}{N^2}\int_HC\left(\|x\|^2+\mu(\|\cdot\|^p)+C_{s,t}\right)\mu(\rmd x)
\\\nonumber
&
\leq \frac{ C_{R,s,t}}{N^2},
\end{align*}
which implies that $P_{s,t}^* B_{R}$ is tight.
Hence, for any $\{\mu_n\}\subset P_{s,t}^* B_{R}$, there exists a subsequence, which 
we still denote by $\{\mu_n\}$, such that $\mu_n$ converges weakly to $\tilde{\mu}$.
On the other hand, it follows from \eqref{0401:02} that for any $\nu\in B_R$,
\begin{equation*}
P_{s,t}^*\nu(\|\cdot\|^p)\leq \nu(\|\cdot\|^p){\rm e}^{-\lambda'(t-s)}+C
\leq R+C.
\end{equation*}
Hence, by H\"older's inequality, \eqref{0401:02} and Chebyshev's inequality,
one sees that
\begin{align*}
\lim_{r\rightarrow\infty}\limsup_{n\rightarrow\infty}
\int_{\|x\|\geq r}\|x\|^2\mu_n(\rmd x)
&
\leq\lim_{r\rightarrow\infty}\limsup_{n\rightarrow\infty}
\left(\int_H\|x\|^p\mu_n(\rmd x)\right)^{\frac{2}{p}}
\left(\int_{\|x\|>r}\mu_n(\rmd x)\right)^{\frac{p-2}{p}}=0,
\end{align*}
which along with \cite[Definition 6.8]{Vil09} implies that 
$\lim_{n\rightarrow\infty}\cW_2(\mu_n,\tilde{\mu})=0$.
\end{proof}

\begin{lem}\label{lem0414}
There exists a $\theta_{1}$-invariant set $\Omega_0\subset \Omega$ of full $\mP$-measure such that

(i) for any $i=1,..,d$, $p\geq1$ and $\omega\in \Omega_0$,
\begin{equation}\label{0421:02}
\lim_{t\rightarrow\pm\infty}\frac{|z_i^I(t,\omega)|^p}{|t|}=0;
\end{equation}

(ii)
for all $\omega\in\Omega_0$, $s<t$, $1\leq i\leq d$ and $p\geq2$,
\begin{equation}\label{0414:01}
\lim_{s\rightarrow-\infty}\frac{1}{t-s}\int_s^t|z_i^I(r,\omega)|^p
\rmd r=\mE|z_i^I(0)|^p
\leq(2\eta)^{-\frac{p}{2}}.
\end{equation}
\end{lem}
\begin{proof}
(i)
By Burkholder-Davis-Gundy's inequality, we have for any $t\in\R$ and $p\geq1$,
\begin{align}\label{0421:03}
\mE|z_i^I(t)|^p
=\mE|z_i^I(0)|^p
=\mE\left|\int_{-\infty}^0{\rm e}^{\eta\tau}\rmd W_i(\tau)\right|^p
\leq \left(\int_{-\infty}^0{\rm e}^{2\eta\tau}\rmd\tau\right)^{p/2}\leq (2\eta)^{-\frac{p}{2}}.
\end{align}
By \eqref{0421:03} and Burkholder-Davis-Gundy's inequality again, one sees that for any $p\geq1$,
\begin{align*}
\mE\left(\sup_{t\in[0,1]}|z_i^I(t)|^p\right)
&
\leq 3^{p-1}\left(\mE|z_i^I(0)|^p
+\mE\left(\sup_{t\in[0,1]}\left|\int_0^t-\eta z_i^I(r)\rmd r\right|^p\right)
+\mE\left(\sup_{t\in[0,1]}|W_i(t)|^p\right)\right)\nonumber\\
&
\leq C\left(1+\eta^p\int_0^1\mE|z_i^I(r)|^p\rmd r\right)\leq C,
\end{align*}
which along with the dichotomy for linear growth of stationary process
(see \cite[Proposition 4.1.3]{Arnold}) implies \eqref{0421:02}.

(ii)
We complete the proof by the Birkhoff ergodic theorem and \eqref{0421:03}. 
\end{proof}

Henceforth, we identify $\Omega_0$ with $\Omega$ and denote both simply by $\Omega$.

\begin{lem}\label{lem0911}
Assume that {\bf(H1)}, {\bf(H2)}, {\bf(H2$''$)}, {\bf(H3)} and {\bf(H4)} hold.
Then for each $\omega\in\Omega$, $\mu\in\cP_p(H)$, $s\in\R$, $T>0$ and $z_0\in H$, 
there exists a unique solution 
$Z_\cdot^{\mu}(\omega)\in C\left([s,s+T];H\right)\cap L^2\left([s,s+T];V\right)$ 
to the following equation:
\begin{equation*}
\begin{cases}
\rmd Z_t^\mu(\omega)=\left(AZ_t^\mu(\omega)
+B(Z_t^\mu(\omega)+z_t^I(\omega))+F(Z_t^\mu(\omega)+z_t^I(\omega),P_{s,t}^*\mu)
+\eta z_t^I(\omega)+Az_t^I(\omega)\right)\rmd t,\\
Z_s=z_0.
\end{cases}
\end{equation*}
\end{lem}
\begin{proof}
For any $\omega\in\Omega$ and $\mu\in\cP_p(H)$,
define 
$$
A_\omega^\mu(t,u):=Au+B(u+z_t^I(\omega))+F(u+z_t^I(\omega),P_{s,t}^*\mu)
+\eta z_t^I(\omega)+Az_t^I(\omega).
$$
By Theorem \ref{EUsolution}, it suffices to show that for any $\omega\in\Omega$ and $\mu\in\cP_p(H)$, $A_\omega^\mu$ satisfies conditions {\bf(A1)}--{\bf(A4)}.
Note that {\bf(A1)} is from {\bf(H1)} and the linearity of $A$ and $B$.
In view of \eqref{0424:02}, {\bf(H2)} and \eqref{0401:02},
we have for any $u_1,u_2\in V$ and $t\in[s,s+T]$,
\begin{align*}
&
\,_{V^*}\langle A_\omega^\mu(t,u_1)-A_\omega^\mu(t,u_2),u_1-u_2\rangle_V\\
&
=\,_{V^*}\langle A(u_1-u_2)+B(u_1+z_t^I(\omega))-B(u_2+z_t^I(\omega)),u_1-u_2\rangle_V\\
&\quad
+\,_{V^*}\langle F(u_1+z_t^I(\omega),P_{s,t}^*\mu)
-F(u_2+z_t^I(\omega),P_{s,t}^*\mu),u_1-u_2\rangle_V\\
&
\leq -\frac{\nu_c}{2}\|u_1-u_2\|_V^2
+C\left[\left(1+\|u_2+z_t^I(\omega)\|_V^2\right)\left(1+\|u_2+z_t^I(\omega)\|^2\right)
+P_{s,t}^*\mu\left(\|\cdot\|^\kappa\right)\right]\|u_1-u_2\|^2\\
&
\leq-\frac{\nu_c}{2}\|u_1-u_2\|_V^2
+C\left[\left(1+\|z_t^I(\omega)\|_V^2+\|u_2\|_V^2\right)
\left(1+\|z_t^I(\omega)\|^2+\|u_2\|^2\right)
+\mu(\|\cdot\|^p)\right]\|u_1-u_2\|^2\\
&
\leq-\frac{\nu_c}{2}\|u_1-u_2\|_V^2
+C\left(1+\|z_t^I(\omega)\|_V^4+\sup_{t\in[s,s+T]}\|z_t^I(\omega)\|^2
+\mu(\|\cdot\|^p)\right)\|u_1-u_2\|^2\\
&\quad
+C\left(1+\|u_2\|_V^2\right)\left(1+\|u_2\|^4\right)\|u_1-u_2\|^2 
\end{align*}
which along with the property of $z_\cdot(\omega)\in L_{loc}^\infty(\R;V)$
implies that {\bf(A2)} holds.

Since $\phi_i$ is Lipschitz continuous for any $i=1,...,d$, 
there exists $C>0$ such that for any $u\in V$ and $t\in\R$, 
\begin{align}\label{0409:03}
-\langle B(u+z_t^I(\omega)), z_t^I(\omega)\rangle
&
=\langle B(u+z_t^I(\omega),z_t^I(\omega)), u+z_t^I(\omega)\rangle\nonumber\\
&
=\sum_{i=1}^d\langle B(u+z_t^I(\omega),\phi_i), u+z_t^I(\omega)\rangle z_i^I(t,\omega)\nonumber\\
&
\leq C\|u+z_t^I(\omega)\|^2\sum_{i=1}^d|z_i^I(t,\omega)|\nonumber\\
&
\leq C\sum_{i=1}^d|z_i^I(t,\omega)|\left(\|u\|^2+\|z_t^I(\omega)\|^2\right).
\end{align}
By Young's inequality, one sees that for any $\varepsilon_1\in(0,1)$, 
there exists $C_{\varepsilon_1}>0$ such that for any $u\in V$ and $t\in[s,s+T]$, 
\begin{align}\label{0409:04}
2~_{V^*}\langle A(u+z_t^I(\omega)),z_t^I(\omega)\rangle_V
&
\leq2\|A(u+z_t^I(\omega))\|_{V^*}\|z_t^I(\omega)\|_V\nonumber\\
&
\leq C\|u+z_t^I(\omega)\|_{V}\|z_t^I(\omega)\|_V\nonumber\\
&
\leq \varepsilon_1\|u+z_t^I(\omega)\|_V^2+C_{\varepsilon_1}\|z_t^I(\omega)\|_V^2.
\end{align}
And by {\bf(H4)} and Young's inequality, 
we obtain that for any $\varepsilon_2\in(0,1)$ there exists $C_{\varepsilon_2}>0$
such that for any $u\in V$, $\mu\in\cP_p(H)$ and $t\in R$,
\begin{align}\label{0409:05}
&
2_{V^*}\langle F(u+z_t^I(\omega),\mu),z_t^I(\omega)\rangle_V\nonumber\\
&
\leq 2\|F(u+z_t^I(\omega),\mu)\|_{V^*}\|z_t^I(\omega)\|_V\nonumber\\
&
\leq C\left[(1+\|u+z_t^I(\omega)\|_V+(\mu(\|\cdot\|^\kappa))^{1/2})
(1+\|u+z_t^I(\omega)\|+(\mu(\|\cdot\|^\kappa))^{1/2})\right]\|z_t^I(\omega)\|_V\nonumber\\
&
\leq \varepsilon_2\|u+z_t^I(\omega)\|_V^2+\varepsilon_2\mu(\|\cdot\|^\kappa)+\varepsilon_2
+C_{\varepsilon_2}\left(1+\|u+z_t^I(\omega)\|^2+\mu(\|\cdot\|^\kappa)\right)
\|z_t^I(\omega)\|_V^2\nonumber\\
&
\leq \varepsilon_2\|u+z_t^I(\omega)\|_V^2
+\left(\varepsilon_2+C_{\varepsilon_2}\|z_r\|_V^2\right)\mu(\|\cdot\|^\kappa)
+C_{\varepsilon_2}\|z_t^I(\omega)\|_V^2\|u\|^2\nonumber\\
&\quad
+C_{\varepsilon_2}\left(\|z_t^I(\omega)\|^2\|z_t^I(\omega)\|_V^2+\|z_t^I(\omega)\|_V^2+1\right).
\end{align}
Combining \eqref{0401:01}, \eqref{0409:04}, \eqref{0409:03} and \eqref{0409:05}, 
we have for any $u\in V$, $\mu\in\cP_p(H)$ and $\varepsilon_3\in(0,1)$,
\begin{align}\label{0904:04}
2\,_{V^*}\langle A_\omega^\mu(t,u), u\rangle_V 
&
=2\,_{V^*}\langle A(u+z_t^I(\omega))+B(u+z_t^I(\omega))+F(u+z_t^I(\omega),\mu)+\eta z_r, u\rangle_V \nonumber\\
&
=2\,_{V^*}\langle A(u+z_t^I(\omega))+B(u+z_t^I(\omega))+F(u+z_t^I(\omega),\mu), 
u+z_t^I(\omega)\rangle_V\nonumber\\
&\quad
+2\,_{V^*}\langle\eta z_t^I(\omega),u\rangle_V-2\,_{V^*}\langle A(u+z_t^I(\omega))+B(u+z_t^I(\omega))+F(u+z_t^I(\omega),\mu), 
z_t^I(\omega)\rangle_V\nonumber\\
&
\leq -(\lambda_1'-\varepsilon_1-\varepsilon_2)\|u+z_t^I(\omega)\|_V^2
-\lambda_2'\|u+z_t^I(\omega)\|^2+\varepsilon_3\|u\|_V^2\nonumber\\
&\quad
+\left(C\sum_{i=1}^{d}|z_i^I(t,\omega)|
+C_{\varepsilon_2}\|z_t^I(\omega)\|_V^2\right)\|u\|^2
+\left(\varepsilon_2+\lambda_2+C_{\varepsilon_2}\|z_t^I(\omega)\|_V^2\right)
P_{s,t}^*\mu(\|\cdot\|^\kappa)\nonumber\\
&\quad
+\left(C_{\varepsilon_3}+C\sum_{i=1}^{d}|z_i^I(t,\omega)|\right)\|z_t^I(\omega)\|^2
+C_{\varepsilon_1,\varepsilon_2}\left(1+\|z_t^I(\omega)\|^2\right)\|z_t^I(\omega)\|_V^2
+C_{\varepsilon_2}.
\end{align}
We note that  for any $p\geq 1$ and $a,b\in\R_+$,
\begin{equation}\label{0418:02}
(a+b)^p\geq 2^{1-p}a^p-b^p.
\end{equation}
Hence, in view of \eqref{0904:04} and \eqref{0418:02}, we obtain that 
\begin{align}\label{0409:06}
&
2_{V^*}\langle A_\omega^\mu(t,u),u\rangle_V \nonumber\\
&
\leq -\left(\frac{\left(\lambda_1'-\varepsilon_1-\varepsilon_2\right)}{2}
-\varepsilon_3\right)\|u\|_V^2
+\left(-\frac{\lambda_2'}{2}+C\sum_{i=1}^d|z_i(t,\omega)|
+C_{\varepsilon_2}\|z_t^I(\omega)\|_V^2\right)\|u\|^2
\nonumber\\
&\quad
+\left(\varepsilon_2+\lambda_2+C_{\varepsilon_2}\|z_t^I(\omega)\|_V^2\right)
P_{s,t}^*\mu(\|\cdot\|^\kappa)\nonumber\\
&\quad
+C_{\varepsilon_1,\varepsilon_2,\varepsilon_3}\left[(1+\|z_t^I(\omega)\|^2)\|z_t^I(\omega)\|_V^2
+(1+C\sum_{i=1}^{d}|z_i(t,\omega)|)\|z_t^I(\omega)\|^2+1\right]
\nonumber\\
&
\leq -\left(\frac{\left(\lambda_1'-\varepsilon_1-\varepsilon_2\right)}{2}-\varepsilon_3\right)\|u\|_V^2
+\left(-\frac12\lambda_2'+f_1(t,\omega)\right)\|u\|^2
+f_2(t,\omega)P_{s,t}^*\mu(\|\cdot\|^\kappa)
+f_3(t,\omega),
\end{align}
where
$$
f_1(t,\omega):=C\sum_{i=1}^d|z_i(t,\omega)|
+C_{\varepsilon_2}\sum_{i}^{d}|z_i(t,\omega)|^2,
\quad
f_2(t,\omega):=C_{\varepsilon_2}\left(1+\sum_{i=1}^{d}|z_i(t,\omega)|^2\right)
$$
and
$$
f_3(t,\omega):=C_{\varepsilon_1,\varepsilon_2,\varepsilon_3}
\left[\left(1+\sum_{i=1}^{d}|z_i(t,\omega)|^2\right)
\sum_{j=1}^{d}|z_j(t,\omega)|^2+1\right].
$$
Thanks to \eqref{0409:06}, by letting $\varepsilon_1$, $\varepsilon_2$ and $\varepsilon_3$ small enough
so that $(\lambda_1'-\varepsilon_1-\varepsilon_2)/2-\varepsilon_3>\frac14\lambda_1'$,
we have
\begin{align}\label{0911:01}
2_{V^*}\langle A_\omega^\mu(t,u),u\rangle_V
\leq-\frac{\lambda_1'}{4}\|u\|_V^2
+\left(-\frac{\lambda_2'}{2}+f_1(t,\omega)\right)\|u\|^2
+f_2(t,\omega)P_{s,t}^*\mu(\|\cdot\|^\kappa)
+f_3(t,\omega).
\end{align}
Therefore, with the help of \eqref{0911:01} and \eqref{0401:02} we have
\begin{align*}
\,_{V^*}\langle A_\omega^\mu(t,u),u\rangle_V
&
\leq-\frac{\lambda_1'}{8}\|u\|_V^2
+\left(-\frac{\lambda_2'}{4}+\frac12f_1(t,\omega)\right)\|u\|^2
+\frac12f_2(t,\omega)P_{s,t}^*\mu(\|\cdot\|^\kappa)
+\frac12f_3(t,\omega)\\
&
\leq-\frac{\lambda_1'}{8}\|u\|_V^2
+\left(-\frac{\lambda_2'}{4}+\frac12f_1(t,\omega)\right)\|u\|^2
+C\left(f_2(t,\omega)\mu(\|\cdot\|^p)
+f_3(t,\omega)\right),
\end{align*}
which implies that {\bf(A3)} holds. By \eqref{0911:04}, {\bf(H4)} and \eqref{0401:02}, 
we have
\begin{align}\label{0911:05}
\left\|A_\omega^\mu(t,u)\right\|_{V^*}^2
&
\leq 2\left(\left\|A(u+z_t^I(\omega))\right\|_{V^*}^2
+\left\|B(u+z_t^I(\omega))\right\|_{V^*}^2
+\left\|F(u+z_t^I(\omega),P_{s,t}^*\mu)\right\|_{V^*}^2
+\left\|\eta z_r\right\|_{V^*}^2\right)\nonumber\\
&
\leq C\left(1+\|u+z_t^I(\omega)\|_V^2+P_{s,t}^*\mu(\|\cdot\|^\kappa)\right)
\left(1+\|u+z_t^I(\omega)\|^2+P_{s,t}^*\mu(\|\cdot\|^\kappa)\right),
\end{align}
which implies that
\begin{align*}
\left\|A_\omega^\mu(t,u)\right\|_{V^*}^2
\leq C\left(1+\sup_{t\in[s,s+T]}\|z_t^I(\omega)\|^2+\mu(\|\cdot\|^p)\right)
\left(1+\|z_t^I(\omega)\|_V^2+\mu(\|\cdot\|^p)\right)\left(1+\|u\|^2\right).
\end{align*}
Hence, {\bf(A4)} holds.
We complete the proof by Lemma \ref{EUsolution}.
\end{proof}

\begin{lem}
Assume that {\bf(H1)}, {\bf(H2)}, {\bf(H2$''$)}, {\bf(H3)} and {\bf(H4)} hold.
Let $s<t$ and $\omega\in\Omega$. For any $x_n,x\in H$, $\mu\in\cP_p(H)$ and
$\{\mu_n\}\subset\cP_p(H)$,
if 
$\lim_{n\rightarrow\infty}\left(\|x_n-x\|+\cW_2(\mu_n,\mu)\right)=0$
then 
$$\lim_{n\rightarrow\infty}\|Z_{s,t}^{x_n,\mu_n}(\omega)-Z_{s,t}^{x,\mu}(\omega)\|=0.$$
\end{lem}
\begin{proof}
By \eqref{0424:02}, {\bf(H2)} and \eqref{0401:02}, 
one sees that
for any $\mu_n,\mu\in\cP_p(H)$ and $x_n,x\in H$,
\begin{align*}
&
\|Z_{s,t}^{x_n,\mu_n}(\omega)-Z_{s,t}^{x,\mu}(\omega)\|^2\\
&
=\|x_n-x\|^2+\int_s^t2_{V^*}\langle A\left(Z_{s,r}^{x_n,\mu_n}(\omega)-Z_{s,r}^{x,\mu}(\omega)\right),
Z_{s,r}^{x_n,\mu_n}(\omega)-Z_{s,r}^{x,\mu}(\omega)\rangle\rmd r\\
&\quad
+\int_s^t2_{V^*}\langle B(X_{s,r}^{x_n,\mu_n})-B(X_{s,r}^{x,\mu}),
Z_{s,r}^{x_n,\mu_n}(\omega)-Z_{s,r}^{x,\mu}(\omega)\rangle\rmd r\\
&\quad
+\int_s^t2_{V^*}\langle F(X_{s,r}^{x_n,\mu_n},\sL_{Y_r^{\mu_n}})
-F(X_{s,r}^{x,\mu},\sL_{Y_r^{\mu}}),
Z_{s,r}^{x_n,\mu_n}(\omega)-Z_{s,r}^{x,\mu}(\omega)\rangle\rmd r\\
&
\leq \|x_n-x\|^2+\int_s^tC\left[\left(\|X_{s,r}^{x,\mu}\|_V^2+1\right)\left(1+\|X_{s,r}^{x,\mu}\|^2
\right)+\mE\|Y_r^{\mu}\|^\kappa\right]\|Z_{s,r}^{x_n,\mu_n}(\omega)-Z_{s,r}^{x,\mu}(\omega)\|^2\rmd r\\
&\quad
+\int_s^t C\left(1+\mE\|Y_r^{\mu}\|^\kappa\right)
\cW_2(\sL_{Y_r^{\mu_n}},\sL_{Y_r^\mu})^2\rmd r\\
&
\leq \|x_n-x\|^2+\int_s^tC\left[\left(\|X_{s,r}^{x,\mu}\|_V^2+1\right)\left(1+\|X_{s,r}^{x,\mu}\|^2
\right)+\mu(\|\cdot\|^p)\right]\|Z_{s,r}^{x_n,\mu_n}(\omega)-Z_{s,r}^{x,\mu}(\omega)\|^2\rmd r\\
&\quad
+C_{s,t}\left(1+\mu(\|\cdot\|^p)\right)\sup_{r\in[s,t]}
\cW_2(P_{r-s}^*\mu_n,P_{r-s}^*\mu)^2\\
&
\leq \|x_n-x\|^2+\int_s^tC\left[\left(\|Z_{s,r}^{x,\mu}(\omega)\|_V^2
+\|z_r^I(\omega)\|_V^2+1\right)
\left(1+\|Z_{s,r}^{x,\mu}(\omega)\|^2+\|z_r^I(\omega)\|^2\right)
+\mu(\|\cdot\|^p)\right]\\
&\qquad
\times\|Z_{s,r}^{x_n,\mu_n}(\omega)-Z_{s,r}^{x,\mu}(\omega)\|^2\rmd r
+C_{s,t}\left(1+\mu(\|\cdot\|^p)\right)\sup_{r\in[s,t]}
\cW_2(P_{r-s}^*\mu_n,P_{r-s}^*\mu)^2,
\end{align*}
which along with the Gronwall inequality implies that
\begin{align}\label{0425:01}
\|Z_{s,t}^{x_n,\mu_n}(\omega)-Z_{s,t}^{x,\mu}(\omega)\|^2
\leq {\rm e}^{\int_s^tf(r)\rmd r}\left(\|x_n-x\|^2
+C_{s,t}\left(1+\mu(\|\cdot\|^p)\right)\sup_{r\in[s,t]}
\cW_2(P_{r-s}^*\mu_n,P_{r-s}^*\mu)^2\right),
\end{align}
where
$$
f(r):=C\left[\left(\|Z_{s,r}^{x,\mu}(\omega)\|_V^2
+\|z_r^I(\omega)\|_V^2+1\right)
\left(1+\|Z_{s,r}^{x,\mu}(\omega)\|^2+\|z_r^I(\omega)\|^2\right)
+\mu(\|\cdot\|^p)\right].
$$
Note that
\begin{align}\label{0425:02}
&
\int_s^t\left[\left(\|Z_{s,r}^{x,\mu}(\omega)\|_V^2
+\|z_r^I(\omega)\|_V^2+1\right)
\left(1+\|Z_{s,r}^{x,\mu}(\omega)\|^2
+\|z_r^I(\omega)\|^2\right)+\mu(\|\cdot\|^p)\right]\rmd r\nonumber\\
&
\leq \int_s^t\bigg(\|Z_{s,r}^{x,\mu}(\omega)\|_V^2\|Z_{s,r}^{x,\mu}(\omega)\|^2
+\|Z_{s,r}^{x,\mu}(\omega)\|_V^2\left(1+\|z_r^I(\omega)\|^2\right)
+\|Z_{s,r}^{x,\mu}(\omega)\|^2\left(1+\|z_r^I(\omega)\|_V^2\right)\nonumber\\
&\qquad
+\left(1+\|z_r^I(\omega)\|^2\right)\left(1+\|z_r^I(\omega)\|_V^2\right)
+\mu(\|\cdot\|^p)\bigg)\rmd r
\leq C_\omega.
\end{align}
Then we complete the proof by \eqref{0425:01}, \eqref{0425:02} and \eqref{0423:01}.
\end{proof}

\begin{lem}
Assume that {\bf(H1)}, {\bf(H2)}, {\bf(H2$''$)}, {\bf(H3)} and {\bf(H4)} hold.
Suppose further that for all $i=1,...,d$, 
$\phi_i$ is Lipschitz continuous in $\mT^2$.
Then for any $p>4\vee\kappa$ and 
$\lambda'\in(0,\frac{\lambda_2' p}{16\kappa}\wedge \frac{p(\lambda_2'-\lambda_2)}{2})$
there exists a constant $C>0$ such that for all $x\in H$, $\mu\in \cP_p(H)$ and $s<t$,
\begin{align}\label{0418:05}
\|Z_{s,t}^{x,\mu}(\omega)\|^2
&
\leq 2\|x\|^2{\rm e}^{\int_{s}^{t}(-\frac{\lambda_2'}{8}+g_1(r,\omega))\rmd r}
+(\mu(\|\cdot\|^p))^{\frac{2\kappa}{p}}{\rm e}^{-\frac{2\lambda'\kappa}{p}(t-s)}
\int_{s}^{t}{\rm e}^{\int_r^t\left(-\frac{\lambda_2'}{8}+g_1(\tau,\omega)\right)\rmd \tau}\rmd r
\nonumber\\
&\quad
+\|z_{s}^I(\omega)\|^2
{\rm e}^{\int_{s}^{t}(-\frac{\lambda_2'}{8}+g_1(r,\omega))\rmd r}
+\int_{s}^{t}g_2(r,\omega)
{\rm e}^{\int_r^{t}(-\frac{\lambda_2'}{8}+g_1(\tau,\omega))\rmd\tau}\rmd r.
\end{align}
where
$$
g_1(r,\omega):=C\sum_{i=1}^d|z_i^I(r,\omega)|\left(|z_i^I(r,\omega)|+1\right),
\quad
g_2(r,\omega):=C\left(1+\sum_{i=1}^{d}|z_i^I(r,\omega)|^4\right).
$$

In particular, fix $t_0\geq0$.
For any $\omega\in\Omega$, $R'>0$, 
there exists $T(\omega,R')>0$ such that
for all 
$t>T(\omega)$, $\|x\|\leq R'$ and $\mu(\|\cdot\|^p)<R'$,
\begin{equation}\label{1027:01}
\|Z_{-t,-t_0}^{x,\mu}(\omega)\|^2\leq R^2(t_0,\omega),
\end{equation}
where 
$$
R^2(t_0,\omega):=2+2\sup_{s\leq-t_0}\|z_{s}^I(\omega)\|^2{\rm e}^{\int_{s}^{-t_0}\left(-\frac{\lambda_2'}{8}+g_1(r,\omega)\right)\rmd r}
+\int_{-\infty}^{-t_0}{\rm e}^{\int_{r}^{-t_0}\left(-\frac{\lambda_2'}{8}+g_1(r,\omega)\right)\rmd \tau}
g_2(r,\omega)\rmd r.
$$
\end{lem}
\begin{proof}

By \eqref{0911:01}, one sees that for any $\tilde{p}\geq2$ and $s\leq \tau<t$,
\begin{align*}
\|Z_{s,t}^{x,\mu}(\omega)\|^{\tilde{p}}
&
=\|Z_{s,\tau}^{x,\mu}(\omega)\|^{\tilde{p}}+\int_\tau^t\frac{\tilde{p}}{2}
\|Z_{s,r}^{x,\mu}(\omega)\|^{\tilde{p}-2}\bigg(2_{V^*}\langle
A(Z_{s,r}^{x,\mu}(\omega)+z_r^I)+B(Z_{s,r}^{x,\mu}(\omega)+z_r^I),Z_{s,r}^{x,\mu}(\omega)\rangle_V\\
&\qquad
+2_{V^*}\langle F(Z_{s,r}^{x,\mu}(\omega)+z_r^I,\sL_{Y_r})
+\eta z_r^I,Z_{s,r}^{x,\mu}(\omega)\rangle_V \bigg)\rmd r\\
&
\leq \|Z_{s,\tau}^{x,\mu}(\omega)\|^{\tilde{p}}
+\int_\tau^t\frac{\tilde{p}}{2}\|Z_{s,r}^{x,\mu}(\omega)\|^{\tilde{p}-2}
\bigg(-\frac14\lambda_1'\|Z_{s,r}^{x,\mu}(\omega)\|_V^2
+\left(-\frac12\lambda_2'+f_1(r,\omega)\right)\|Z_{s,r}^{x,\mu}(\omega)\|^2\\
&\qquad
+f_2(r,\omega)P_{s,r}^*(\|\cdot\|^\kappa)
+f_3(r,\omega)\bigg)\rmd r,
\end{align*}
which along with \eqref{0401:02} and Young's inequality 
implies that for any $p>4\vee\kappa$ and $\tilde{p}\geq2$,
\begin{align}\label{0409:07}
&
\|Z_{s,t}^{x,\mu}(\omega)\|^{\tilde{p}}
+\frac{\lambda_1' {\tilde{p}}}{8}\int_\tau^t\|Z_{s,r}^{x,\mu}(\omega)\|_V^2\|Z_{s,r}^{x,\mu}(\omega)\|^{{\tilde{p}}-2}\rmd r\nonumber\\
&
\leq 
\|Z_{s,\tau}^{x,\mu}(\omega)\|^{\tilde{p}}
+\int_\tau^t\left(-\frac{\lambda_2'\tilde{p}}{4}
+\frac{\tilde{p}}{2}f_1(r,\omega)\right)\|Z_{s,r}^{x,\mu}(\omega)\|^{\tilde{p}}
\rmd r\nonumber\\
&\quad
+\int_\tau^t\frac{\tilde{p}}{2}\left[f_2(r,\omega)
P_{s,t}^*(\|\cdot\|^\kappa)
+f_3(r,\omega)\right]\|Z_{s,r}^{x,\mu}(\omega)\|^{{\tilde{p}}-2}
\rmd r\nonumber\\
&
\leq \|Z_{s,\tau}^{x,\mu}(\omega)\|^{\tilde{p}}
+\int_\tau^t\left(-\frac{\lambda_2'\tilde{p}}{8}+\frac{\tilde{p}}{2}f_1(r,\omega)\right)
\|Z_{s,r}^{x,\mu}(\omega)\|^{\tilde{p}}\rmd r\nonumber\\
&\quad
+\int_\tau^tC_{\tilde{p}}\left(f_2(r,\omega)^{{\tilde{p}}/2}
(\mE\|Y_r\|^p)^{\frac{\kappa\tilde{p}}{2p}}
+f_3(r,\omega)^{{\tilde{p}}/2}\right)\rmd r\nonumber\\
&
\leq\|Z_{s,\tau}^{x,\mu}(\omega)\|^{\tilde{p}}
+\int_\tau^t\left(-\frac{\lambda_2'\tilde{p}}{8}+\frac{\tilde{p}}{2}f_1(r,\omega)\right)
\|Z_{s,r}^{x,\mu}(\omega)\|^{\tilde{p}}\rmd r\nonumber\\
&\quad
+\int_\tau^t C_{\tilde{p}}\left(f_2(r,\omega)^{{\tilde{p}}/2}
(\mE\|Y_s\|^p{\rm e}^{-\lambda'(r-s)}+C)^{\frac{\kappa\tilde{p}}{2p}}
+f_3(r,\omega)^{{\tilde{p}}/2}\right)\rmd r\nonumber\\
&
\leq \|Z_{s,\tau}^{x,\mu}(\omega)\|^{\tilde{p}}
+\int_\tau^t\left(-\frac{\lambda_2'\tilde{p}}{8}+\frac{\tilde{p}}{2}f_1(r,\omega)\right)
\|Z_{s,r}^{x,\mu}(\omega)\|^{\tilde{p}}\rmd r\nonumber\\
&\quad 
+\int_\tau^t\left[(\mu(\|\cdot\|^p))^{\frac{\kappa\tilde{p}}{p}}
{\rm e}^{-\frac{\lambda'\kappa\tilde{p}}{p}(r-s)}
+C_{\tilde{p}}f_2(r,\omega)^{\tilde{p}}
+C_{\tilde{p}}f_3(r,\omega)^{{\tilde{p}}/2}\right]\rmd r.
\end{align}

In view of \eqref{0409:07} and Gronwall's inequality, we have for any 
$p>4\vee\kappa$ and $\tilde{p}=2$,
\begin{align}\label{0418:03}
\|Z_{s,t}^{x,\mu}(\omega)\|^{2}
&
\leq \|Z_{s,\tau}^{x,\mu}(\omega)\|^{2}
{\rm e}^{\int_\tau^t(-\frac{\lambda_2'}{4}+ f_1(r,\omega))\rmd r}
+\int_\tau^t(\mu(\|\cdot\|^p))^{\frac{2\kappa}{p}}{\rm e}^{-\frac{2\lambda'\kappa}{p}(r-s)}
{\rm e}^{\int_r^t(-\frac{\lambda_2'}{4}+ f_1(\tau',\omega))
\rmd\tau'}\rmd r\nonumber\\
&\quad
+\int_\tau^tC\left(f_2(r,\omega)^{2}+f_3(r,\omega)\right)
{\rm e}^{\int_r^t(-\frac{\lambda_2'}{4}+ f_1(\tau',\omega))\rmd\tau'}\rmd r
\nonumber\\
&
\leq \|Z_{s,\tau}^{x,\mu}(\omega)\|^{2}
{\rm e}^{\int_\tau^t(-\frac{\lambda_2'}{4}+ f_1(r,\omega))\rmd r}
+(\mu(\|\cdot\|^p))^{\frac{2\kappa}{p}}
{\rm e}^{-\frac{\lambda'\kappa2}{p}(t-s)}
\int_\tau^t
{\rm e}^{\int_r^t(\frac{2\lambda'\kappa}{p}-\frac{\lambda_2'}{4}
+f_1(\tau',\omega))\rmd\tau'}\rmd r\nonumber\\
&\quad
+\int_\tau^tC\left(f_2(r,\omega)^{2}+f_3(r,\omega)\right)
{\rm e}^{\int_r^t(-\frac{\lambda_2'}{4}+ f_1(\tau',\omega))\rmd\tau'}\rmd r.
\end{align}
By letting 
$\lambda'\in(0,\frac{\lambda_2' p}{16\kappa}\wedge\frac{p(\lambda_2'-\lambda_2)}{2})$
in \eqref{0418:03}, one sees that for any
$x\in H$ and $\mu\in \cP_p(H)$, 
\begin{align}\label{0906:05}
\|Z_{s,t}^{x,\mu}(\omega)\|^2
&
\leq 2\|x\|^2{\rm e}^{\int_{s}^{t}(-\frac{\lambda_2'}{8}+g_1(r,\omega))\rmd r}
+(\mu(\|\cdot\|^p))^{\frac{2\kappa}{p}}{\rm e}^{-\frac{2\lambda'\kappa}{p}(t-s)}
\int_{s}^{t}{\rm e}^{\int_r^t\left(-\frac{\lambda_2'}{8}+g_1(\tau,\omega)\right)\rmd \tau}\rmd r\nonumber\\
&\quad
+2\|z_{s}^I(\omega)\|^2
{\rm e}^{\int_{s}^{t}(-\frac{\lambda_2'}{8}+g_1(r,\omega))\rmd r}
+\int_{s}^{t}g_2(r,\omega)
{\rm e}^{\int_r^{t}(-\frac{\lambda_2'}{8}+g_1(\tau',\omega))\rmd\tau'}\rmd r.
\end{align}

Now we show that $R(t_0,\omega)^2<\infty$, $\mP$-a.s. Indeed, 
by \eqref{0414:01}, one sees that there exists $\eta>0$ such that for $\mP$-a.s. $\omega$
there exists $T_0(\omega)>0$ such that for any $t>T_0(\omega)$,
\begin{equation}\label{0421:05}
\int_{-t}^{-t_0}\left(-\frac{\lambda_2'}{8}+g_1(r,\omega)\right)\rmd r
\leq -\frac{\lambda_2'}{16}(-t_0+t),
\end{equation}
which along with \eqref{0421:02} implies that for any $t_0>0$,
$$
2\sup_{s\leq-t_0}\|z_{s}^I(\omega)\|^2{\rm e}^{\int_{s}^{-t_0}(-\frac{\lambda_2'}{8}+g_1(r,\omega))\rmd r}
+\int_{-\infty}^{-t_0}{\rm e}^{\int_{r}^{-t_0}\left(-\frac{\lambda_2'}{8}+g_1(\tau,\omega)\right)\rmd \tau}
g_2(r,\omega)\rmd r<\infty,\quad \mP-{\text{a.s.}}
$$
It follows from \eqref{0421:05} that for any $R'>0$, there exists $T_1(\omega)>T_0(\omega)$ so that
for all $t>T_1(\omega)$ and $\|x\|<R'$, 
$$
2\|x\|^2{\rm e}^{\int_{-t}^{-t_0}(-\frac{\lambda_2'}{8}+g_1(r,\omega))\rmd r}\leq 1.
$$
Since $\lambda'>0$, there exists $T_2>0$ such that for all $t>T_2$ and $\mu\in\cP_p(H)$
with $\mu(\|\cdot\|^p)<R'$,
$$
(\mu(\|\cdot\|^p))^{\frac{2\kappa}{p}}{\rm e}^{-\frac{2\lambda'\kappa}{p}(t_0-t)}\leq1.
$$
Then for any $R'>0$, there exists $T(\omega):=T_1(\omega)\vee T_2$ such that
for all $t>T(\omega)$, $\|x\|\leq R'$ and $\mu(\|\cdot\|^p)<R'$,
\begin{equation*}
\|Z_{-t,-t_0}^{x,\mu}(\omega)\|^2\leq R^2(t_0,\omega)<\infty.
\end{equation*}
\end{proof}

\begin{lem}\label{lem1027}
Assume that {\bf(H1)}, {\bf(H2)}, {\bf(H2$''$)}, {\bf(H3)} and {\bf(H4)} hold.
Suppose further that for all $i=1,...,d$, 
$\phi_i$ is Lipschitz continuous in $\mT^2$.
Then $\varphi$ is $\cD$-pullback asymptotically compact
in the sense of Definition \ref{dePAC}.
\end{lem}
\begin{proof}

Let $\omega\in\Omega$, $B_H\subset H$ bounded, $B_{\cP}\subset \cP_p(H)$ bounded and $s<t$.
We first show that for any $\{x_n\}\subset B_H$ and $\{\mu_n\}\subset B_{\cP}$,
there exists a convergent subsequence of 
$\{Z_{s,t}^{x_n,\mu_n}\}$. 
By \eqref{0409:07}, we obtain that 
for any $x\in H$, $p>4\vee\kappa$ and $\tilde{p}\geq2$,
\begin{align}\label{0412:01}
\frac{\lambda_1'\tilde{p}}{8}\int_s^t\|Z_{s,r}^{x,\mu}(\omega)\|_V^2\|Z_{s,r}^{x,\mu}(\omega)\|^{\tilde{p}-2}\rmd r
&
\leq 
\|x+z_s^I(\omega)\|^{\tilde{p}}
+\int_s^t\left(-\frac{\lambda_2'\tilde{p}}{8}+\frac{\tilde{p}}{2}g_1(r,\omega)\right)
\|Z_{s,r}^{x,\mu}(\omega)\|^{\tilde{p}}\rmd r\nonumber\\
&\quad
+\int_s^t\left[(\mu(\|\cdot\|^p))^{\frac{\kappa\tilde{p}}{p}}{\rm e}^{-\frac{\lambda'\kappa\tilde{p}}{p}(r-s)}
+Cf_2(r,\omega)^{\tilde{p}}
+Cf_3(r,\omega)^{{\tilde{p}}/2}\right]\rmd r,
\end{align}
which along with $\tilde{p}=2$, \eqref{0421:03}, \eqref{0418:05} and the property 
$z_i^I(\cdot,\omega)\in L_{loc}^\infty(\R;\R)$ implies that 
for any $x_n\in B_H$ and $\mu_n\in B_{\cP}$,
\begin{align}\label{0920:03}
\int_{s}^t\|Z_{s,r}^{x_n,\mu_n}(\omega)\|_V^2\rmd r
&
\leq C\left(\|x_n\|^2+\|z^I_s(\omega)\|^2
+\int_s^t\left(-\frac{\lambda_2'}{4}+f_1(r,\omega)\right)
\|Z_{s,r}^{x_n,\mu_n}(\omega)\|^{2}\rmd r\right)
\nonumber\\
&\quad
+C\int_s^t\left[(\mu_n(\|\cdot\|^p))^{\frac{2\kappa}{p}}{\rm e}^{-\frac{2\lambda'\kappa}{p}(r-s)}
+f_2(r,\omega)^{2}+f_3(r,\omega)\right]\rmd r\nonumber\\
&
< C_\omega.
\end{align}
Note that for any $r\geq s$,
\begin{align*}
Z_{s,r}^{x_n,\mu_n}(\omega)=x_n-z_s^I(\omega)
+\int_s^r A_\omega^{\mu_n}\left(\tau,Z_{s,\tau}^{x_n,\mu_n}(\omega)\right)\rmd\tau,
\end{align*}
which along with \cite[Theorem 1.6]{Sh97} implies that
for a.e. $r\in[s,t]$, $Z_{s,r}^{x_n,\mu_n}(\omega)$ is differentiable with
\begin{equation*}
\frac{\rmd}{\rmd r}Z_{s,r}^{x_n,\mu_n}(\omega)=
A_\omega^{\mu_n}(r,Z_{s,r}^{x_n,\mu_n}(\omega)).
\end{equation*}
In view of \eqref{0911:05}, \eqref{0401:02}, \eqref{0920:03} and \eqref{0412:01} with $\tilde{p}=4$, we have
\begin{align}\label{0920:01}
\int_s^t\left\|\frac{\rmd}{\rmd r}Z_{s,r}^{x_n,\mu_n}(\omega)
\right\|_{V^*}^2\rmd r
&
=\int_s^t\left\|A_\omega^{\mu_n}(r,Z_{s,r}^{x_n,\mu_n}(\omega))
\right\|_{V^*}^2\rmd r\nonumber\\
&
\leq C\int_s^t\left(\|Z_{s,r}^{x_n,\mu_n}(\omega)+z_r^I(\omega)\|_V^2
+P_{s,r}^*\mu_n(\|\cdot\|^\kappa)+1\right)\nonumber\\
&\qquad\quad
\times\left(\|Z_{s,r}^{x_n,\mu_n}(\omega)+z_r^I(\omega)\|^2
+P_{s,r}^*\mu_n(\|\cdot\|^\kappa)+1\right)\rmd r\nonumber\\
&
\leq C\int_s^t\Big(
\|Z_{s,r}^{x_n,\mu_n}(\omega)\|_V^2
\|Z_{s,r}^{x_n,\mu_n}(\omega)\|^2
+\left(1+\|z_r^I(\omega)\|^2\right)
\|Z_{s,r}^{x_n,\mu_n}(\omega)\|_V^2
\nonumber\\
&\qquad\quad
+\left(1+\|z_r^I(\omega)\|_V^2\right)
\|Z_{s,r}^{x_n,\mu_n}(\omega)\|^2
+\|z_r^I(\omega)\|_V^4
+1\Big)\rmd r<C.
\end{align}

Define
$$
W:=\left\{v\in L^2([s,t];V),
\frac{\rmd}{\rmd r}v\in L^{2}([s,t];V^*):
\|v\|_W<\infty\right\},
$$
where
$$
\|v\|_W:=\|v\|_{L^2([s,t];V)}
+\left\|\frac{\rmd}{\rmd r}v\right\|_{L^{2}([s,t];V^*)}.
$$
It follows from Aubin-Lions' lemma that $W\subset L^2([s,t];H)$ is compact.
Define $Z_n:=Z_{s,t}^{x_n,\mu_n}(\omega)$.
Hence, by \eqref{0920:03} and \eqref{0920:01}, we obtain that there exist subsequences
$\{x_{n_k}\}$ and $\{\mu_{n_k}\}$ such that $Z_{n_k}$ converges to $Z_0$ in $L^2([s,t];H)$.
For almost everywhere $r\in[s,t]$, we also have
$Z_{n_k}(r)$ converges to $Z_0(r)$ in $H$. 
On the other hand, by \eqref{0421:01}, one sees that there exists a subsequence such that
$P_{s,r}^*\mu_{n_k}$ converges to $\mu_0$.
Therefore, by the continuity of $Z$,
we get
$$
\lim_{k\rightarrow\infty} Z_{s,t}^{x_{n_k},\mu_{n_k}}(\omega)
=\lim_{k\rightarrow\infty}Z_{r,t}^{Z_{s,r}^{x_{n_k},\mu_{n_k}}(\omega)+z_r^I(\omega),
P_{s,r}^*\mu_{n_k}}(\omega)=Z_{r,t}^{Z_0(r)+z_r^I(\omega),\mu_0}(\omega).
$$
Now we show that $\varphi$ is $\cD$-pullback asymptotically compact. 
Indeed, 
for any $\omega\in\Omega$, $D\in\cD$, $t_n\rightarrow+\infty$, $p_n\in\cP^*$,
$q_n\in\Theta_{2,-t_n}p_n$ and $x_n\in D(\theta_{1,-t_n}\omega)$, we have
\begin{align*}
\varphi(t_n,\theta_{1,-t_n}\omega,q_n,x_n)
&
=Z_{0,t_n}^{x_n,q_n}(\theta_{1,-t_n}\omega)+z_{t_n}^I(\theta_{1,-t_n}\omega)\\
&
=Z_{-t_n,0}^{x_n,q_n}(\omega)+z_0^I(\omega)
=Z_{-1,0}^{Z_{-t_n,-1}^{x_n,q_n}(\omega),P_{-t_n,-1}^*q_n}(\omega)
+z_0^I(\omega).
\end{align*}
Define 
$$
\tilde{x}_n:=Z_{-t_n,-1}^{x_n,q_n}(\omega),\quad \tilde{\mu}_n:=P_{-t_n,-1}^*q_n.
$$
By \eqref{1027:01} and \eqref{0401:02}, we obtain that 
$\{\tilde{x}_n\}$ and $\{\tilde{\mu}_n\}$ are bounded.
Hence, $\{\varphi(t_n,\theta_{1,-t_n}\omega,q_n,x_n)\}$ has a convergent subsequence.
\end{proof}

We now turn to the proof of Theorem \ref{thmNSeq2}.

\begin{proof}[Proof of Theorem \ref{thmNSeq2}]
(i) It follows from Lemma \ref{thmNSeq01}.

(ii)--(iii)
By \eqref{0906:05}, one sees that for any $x\in D\in\cD_{\lambda_0}$
and $\mu\in\cP^*$,
\begin{align*}
&
\|X(-t,0,\omega,\mu)x(\theta_{1,-t}\omega)\|^2\\
&
\leq 2\|Z_{-t,0}^{x,\mu}(\omega)\|^2+2\|z_{0}^I(\omega)\|^2\\
&
\leq 
4\|x(\theta_{1,-t}\omega)\|^2{\rm e}^{\int_{-t}^{0}\left(-\frac{1}{8}\lambda_2'+g_1(r,\omega)\right)\rmd r}
+C_\mu{\rm e}^{-\frac{2\lambda'\kappa}{p}t}
\int_{-t}^{0}{\rm e}^{\int_r^{0}\left(-\frac{1}{8}\lambda_2'+g_1(\tau,\omega)\right)\rmd\tau}\rmd r\\
&\quad
+2\|z_{-t}^I(\omega)\|^2{\rm e}^{\int_{-t}^{0}\left(-\frac{1}{8}\lambda_2'+g_1(r,\omega)\right)\rmd r}
+2\int_{-t}^{0}Cg_2(r,\omega){\rm e}^{\int_r^{0}\left(-\frac{1}{8}\lambda_2'+g_1(\tau,\omega)\right)\rmd\tau}
+2\|z_{0}^I(\omega)\|^2.
\end{align*}
By Lemma \ref{lem0414}, we can choose $\eta$ large enough so that
for any $\lambda_0<\frac{\lambda_2'}{8}$,
there exists $t_0(D,\omega)$ such that for any 
$t>t_0(D,\omega)$ and $x\in D$,
\begin{equation*}
4\|x(\theta_{1,-t}\omega)\|^2{\rm e}^{\int_{-t}^{0}
\left(-\frac{1}{8}\lambda_2'+g_1(r,\omega)\right)\rmd r}\leq \frac{1}{3},
\quad
C_\mu{\rm e}^{-\frac{2\lambda'\kappa}{p}t}\leq \frac{1}{3},
\end{equation*}
and
\begin{equation*}
2\|z_{-t}^I(\omega)\|^2{\rm e}^{\int_{-t}^{0}
\left(-\frac{1}{8}\lambda_2'+g_1(r,\omega)\right)\rmd r}\leq\frac{1}{3}.
\end{equation*}
Define 
$$
R(\omega):=1+2\|z_{0}^I(\omega)\|^2
+\int_{-\infty}^{0}Cg_2(r,\omega)
{\rm e}^{\int_r^{0}\left(-\frac{1}{8}\lambda_2'+g_1(\tau,\omega)\right)\rmd\tau}\rmd r
$$
and 
$$
B_R(\omega):=\left\{x\in H: \, \|x\|^2\leq R(\omega)\right\},\quad
B_R:=\left\{B_R(\omega):\omega\in\Omega\right\}.
$$
Hence, $B_R$ is the closed absorbing set.
By Lemma \ref{lem0414}, one sees that 
\begin{align*}
&
\lim_{t\rightarrow\infty}{\rm e}^{-\lambda_0t}\|B_R(\theta_{1,-t}\omega)\|^2\\
&
=\lim_{t\rightarrow\infty}{\rm e}^{-\lambda_0t}R(\theta_{1,-t}\omega)\\
&
=\lim_{t\rightarrow\infty}{\rm e}^{-\lambda_0t}\left(2\|z_{0}^I(\theta_{1,-t}\omega)\|^2
+\int_{-\infty}^{0}g_2(r,\theta_{1,-t}\omega)
{\rm e}^{\int_r^{0}\left(-\frac{1}{8}\lambda_2'
+g_1(\tau,\theta_{1,-t}\omega)\right)\rmd\tau}\rmd r
\right)=0,
\end{align*}
which implies that $B_R\in\cD_{\lambda_0}$. 
Combining Lemma \ref{lem1027}, Proposition \ref{prop0117}
 and Theorem \ref{thm0330}, 
we complete the proof.
\end{proof}

\appendix

\section{}
In this section, we introduce some definitions and results of dynamical systems and
attractors; see e.g. \cite{Arnold, CJLL13, CLL12, CF, Hale88, KR2011}
for more details.

\begin{de}\label{de1103}
Let $\sS$ be a collection of some closed subset of a metric space $(\cX,d)$. 
We say that $\sS$ is {\it inclusion closed}, if for any $D_1\in\sS$ and closed 
$D_2\subset\cX$ with $D_2\subset D_1$ then $D_2\in\sS$. 
\end{de}

\begin{de}
A nonempty compact set $A\in\sS$ is called a {\it $\sS$-global attractor} 
of a semi-DS $\phi$ on the metric space $(\cX,d)$ if it is $\phi$-invariant and attracts all element in $\sS$, 
i.e. for any $D\in\sS$,
$$
\lim_{t\rightarrow+\infty}{\rm {dist}}(\phi(t,D),A)=0.
$$
\end{de}

\begin{lem}\label{EGA}
Assume that a semi-DS $\phi$ on metric space $\cX$ has a nonempty compact absorbing set $B\in\sS$, 
that is to say,  for any $D\in\sS$, there exists $T=T(D)>0$ such that $\phi(t,D)\subset B$ for all $t\geq T$.
If for any $t\geq0$, $\phi(t,\cdot):\cX\rightarrow\cX$ is continuous over 
any element in $\sS$, 
then $\phi$ has a $\sS$-global attractor $A$ given by
$$
A=\Omega(B):=\bigcap_{s\geq0}\overline{\bigcup_{t\geq s}\phi(t,B)}.
$$ 
\end{lem}
\begin{rem}\label{Rem0607}
Note that $x\in \Omega(B)$ if and only if there exists $t_n\rightarrow+\infty$
and $\{x_n\}\subset B$ such that 
$\lim_{n\rightarrow\infty}\phi(t_n,x_n)=x$.
\end{rem}
\begin{proof}
First, we show that $\Omega(B)$ is nonempty compact set.
Since $B$ is nonempty, for any $s\geq0$, $\bigcup_{t\geq s}\phi(t,B)$ is nonempty
and decreases as $s$ increases. 
By the absorbing property of $B$, one sees that there 
exists $t_0>0$ such that 
$$
\bigcap_{s\geq0}\overline{\bigcup_{t\geq s}\phi(t,B)}
\subset\overline{\bigcup_{t\geq t_0}\phi(t,B)}\subset B,
$$
which implies that $\Omega(B)$ is compact. 
For any $t_n\to+\infty$ and $x_n\in B$, it follows from the absorption of $B$ that $\{\phi(t_n,x_n)\}_{t_n>t_B}\subset B$. 
The compactness of $B$ implies that $\{\phi(t_n,x_n)\}_{t_n>t_B}$ possesses a convergent subsequence. 
By Remark \ref{Rem0607}, $\Omega(B)$ is nonempty.

Now we show that $\Omega(B)$ is invariant. 
It follows from Remark \ref{Rem0607} that for any $x\in\Omega(B)$, there exist $t_n\rightarrow+\infty$ and $\{x_n\}\subset B$ such that 
$\lim_{n\rightarrow\infty}\phi(t_n,x_n)=x$.
On the one hand,
due to the absorption of $B$ and $t_n\to+\infty$, 
$\{\phi(t_n,x_n)\}_{t_n>t_B}\subset B$.
Therefore, in view of the semigroup property and the continuity over $B$ 
of $\phi$, we have for any $t\geq 0$, 
\begin{equation*}
\lim_{n\rightarrow\infty}\phi(t+t_n,x_n)=\lim_{n\rightarrow\infty}\phi(t,\phi(t_n,x_n))
=\phi(t,x),
\end{equation*}
which implies that $\phi(t,x)\in\Omega(B)$. 
By the arbitrary of $x\in\Omega(B)$, we obtain 
that for any $t\geq0$, $\phi(t,\Omega(B))\subset\Omega(B)$.
On the other hand, $t_n\to+\infty$ and the absorption and compactness of $B$ imply that for all $t\geq0$ there exists a
subsequence $\{n_k\}$ such that $t_{n_1}-t>t_B$ and 
$\phi(t_{n_k}-t,x_{n_k})\rightarrow y\in \Omega(B)$.
By again the semigroup property and the continuity over $B$
of $\phi$, one sees that
\begin{equation*}
\phi(t,y)=\lim_{k\rightarrow\infty}\phi(t,\phi(t_{n_k}-t,x_{n_k}))
=\lim_{k\rightarrow\infty}\phi(t_{n_k},x_{n_k})=x.
\end{equation*}
Hence, for any $t\geq0$, $\Omega(B)\subset\phi(t,\Omega(B))$.

Finally, we prove that $\Omega(B)$ attracts all element in $\sS$.
By the absorption of $B$, it suffices to show that $\Omega(B)$ attracts $B$.
Assume that $\Omega(B)$ does not attract $B$, i.e.,
there exist $t_n\rightarrow+\infty$, $\{x_n\}\subset B$ and $\varepsilon_0>0$ such that 
for all $n\geq1$,
\begin{equation}\label{0607:01}
{\rm{dist}}(\phi(t_n,x_n),\Omega(B))\geq \varepsilon_0.
\end{equation}
Since $B$ is absorbing and compact, there exists a subsequence $\{n_k\}$ such that
for all $k\geq1$, 
$$
\{\phi(t_{n_k},x_{n_k})\}\subset B {\text{ and }} 
\lim_{k\rightarrow\infty}\phi(t_{n_k},x_{n_k})=:\tilde{x}\in\Omega(B).
$$
Hence, by the continuity over $B$ of $\phi$, one sees that
\begin{equation*}
0=\lim_{k\rightarrow\infty}d(\phi(t_{n_k},x_{n_k}),\tilde{x})
\geq \lim_{k\rightarrow\infty}{\rm{dist}}(\phi(t_{n_k},x_{n_k}),\Omega(B)),
\end{equation*}
which contradicts \eqref{0607:01}. We complete the proof.
\end{proof}

\begin{de}\label{defMDS}
We say that $(\Omega,\cF,\mathbb{P},(\theta_{1,t})_{t\in\R})$ is a 
{\it metric dynamical system} 
if the following conditions are satisfied:
\begin{itemize}
\item[(i)] $(t,\omega)\mapsto\theta_{1,t}(\omega)$ is $\cB(\R)\otimes\cF/\cF$ measurable.
\item[(ii)] $\theta_{1,0}=Id_{\Omega}$ and for all $t,s\in\R$,
$\theta_{1,t+s}=\theta_{1,t}\circ\theta_{1,s}$. 
\item[(iii)] For any $t\in\R$, $\theta_{1,t}$ is $\mathbb{P}$-preserving. 
\end{itemize}
\end{de}

\begin{de}\label{defRDS}
A {\it random dynamical system} with time $\R$ on a metric space 
$(\cX,d)$ with Borel $\sigma$-algebra $\cB$ over $\theta_{1,t}$ on $(\Omega,\cF,\mP)$ 
is a measurable map
$$
\varphi:\R\times\cX\times\Omega\rightarrow\cX,
\quad (t,x,\omega)\mapsto\varphi(t,\omega)x
$$
such that $\varphi(0,\omega)=Id_{\cX}$ and $\varphi$ satisfies the cocycle property, i.e.
for all $t,s\in\R$ and $\omega\in\Omega$,
$$
\varphi(t+s,\omega)=\varphi(t,\theta_{1,s}\omega)\circ\varphi(s,\omega).
$$  

Moreover, an RDS is said to be continuous if $\varphi(t,\omega):\cX\rightarrow\cX$
is continuous for all $t\in\R$ outside a $\mP$-nullset.
\end{de}

For simplicity, we suppress $\theta_{1,t}$ and $(\Omega,\cF,\mP)$,
and simply refer to an RDS $\varphi$.

\begin{de}\label{de0916}
Let $2^\cX$ be the collection of all subsets of $\cX$.
A set-valued map $D:\Omega\rightarrow 2^{\cX}$ taking values in the closed/compact subsets of $\cX$ is said to be measurable 
if the mapping $\omega\mapsto d(x,D(\omega))$ is measurable for any $x\in\cX$.
\end{de}

We recall the following lemma, which is from \cite{CLL12}.
\begin{lem}\label{lem0827}
Let $\cX$ be Polish space. Then $D:\Omega\rightarrow 2^{\cX}$ is compact and measurable if and only if $D(\omega)$ is compact for every $\omega\in\Omega$
and the set $\{\omega: D(\omega)\cap C\neq \emptyset\}$ is measurable for any
closed set $C\subset \cX$.
\end{lem}

We use $\cD$ to denote a collection of some families of nonempty subsets of $\cX$:
$$
\cD:=\left\{D=\left\{\emptyset\neq D(\omega)\subset\cX:\omega\in\Omega\right\}:
D ~\text{satisfies some conditions}\right\}.
$$

\begin{de}\label{DefRA}
A compact random set $\cA\in\cD$ is said to be a {\it $\cD$-random attractor} for an RDS $\varphi$,
if it satisfies $\mP$-a.s.
\begin{enumerate}
\item[(i)] $\cA$ is invariant, i.e. $\varphi(t,\omega)A(\omega)=A(\theta_t\omega)$
for all $t>0$.
\item[(ii)] $\cA$ attracts every member of $\cD$, i.e. for every $B\in\cD$ and $\omega\in\Omega$,
    $$
    \lim_{t\rightarrow+\infty}\dist \left(\varphi(t,\theta_{-t}\omega)B(\theta_{-t}\omega),
    \cA(\omega)\right)=0.
    $$
\end{enumerate}
\end{de}

\begin{de}\label{Duniverse}
We say that $\cD$ is {\it neighborhood closed} if for every $D\in\cD$, there exists a 
constant $\epsilon>0$ depending on $D$ such that the family
$$
\left\{B(\omega):B(\omega)\text{~is a nonempty subset of}~\cN_\epsilon(D(\omega)),
\forall \omega\in \Omega\right\}
$$
also belongs to $\cD$, where
$$
\cN_\epsilon(B):=\left\{x\in\cX: d(x,B)<\epsilon\right\}
$$
for $B\subset \cX$. 
\end{de}

\begin{rem}
Note that the neighborhood closedness of $\cD$ implies that $\cD$ is set inclusion, i.e.
if $D_1\in\cD$ and $D_2(\omega)\subset D_1(\omega)$ for all $\omega\in\Omega$,
then $D_2\in\cD$.
\end{rem}

By \cite[Theorem 7]{Caraballo2006} and \cite[Theorem 2.23]{WangBX12}, we have the following lemma.
\begin{lem}\label{lemAR}
Let $\cD$ be a neighborhood closed collection of some families of nonempty subsets
of $\cX$. Assume that $\varphi$ is a continuous RDS on $\cX$. 
If $\varphi$ is $\cD$-pullback asymptotically compact in $\cX$
and there exists a closed random set $K\in\cD$ such that for all 
$\omega\in\Omega$ and $D\in\cD$, there exists $T(\omega,D)>0$ so that
\begin{equation}\label{1025:02}
\varphi(t,\theta_{-t}\omega)D(\theta_{-t}\omega)\subset
K(\omega),\quad \forall t\geq T(\omega,D),
\end{equation}
then there exists
a $\cD$-random attractor $\cA$ in $\cD$, defined by
$$
\cA(\omega):=\Omega(K(\omega))=\bigcup_{B\in\cD}\Omega(B(\omega)).
$$
\end{lem}

Finally we recall the definition of conjugation mapping and some results 
concerning conjugated stochastic flows; see e.g. \cite{IL01}.
\begin{de}
Let $(\cX,d_1)$ and $(\cY,d_2)$ be two metric spaces.
\begin{enumerate}
  \item A family of homeomorphisms $\cT:=\{\cT(\omega):\cX\rightarrow\cY\}_{\omega\in\Omega}$
  such that the maps $\omega\mapsto \cT(\omega)x$ and $\omega\mapsto \cT^{-1}(\omega)y$ are measurable
  for all $x\in\cX$ and $y\in\cY$, is called a stationary conjugation mapping. 
  Let $\cT(t,\omega)=\cT(\theta_t\omega)$.
  \item Let $S_1(t,s,\omega)$ and $S_2(t,s,\omega)$ be stochastic flows. We say that
  $S_1(t,s,\omega)$ and $S_2(t,s,\omega)$ are (stationary) conjugated, if there exists a (stationary) conjugation
  mapping $\cT$ such that
  \begin{equation*}
  S_2(t,s,\omega)=\cT(t,\omega)\circ S_1(t,s,\omega)\circ \cT^{-1}(s,\omega).
  \end{equation*}
\end{enumerate}
\end{de}

\begin{lem}
Let $\cT$ be a conjugated mapping and $S_1(t,s,\omega)$ be a (measurable and continuous)
stochastic flow. Then
\begin{equation*}
S_2(t,s,\omega):=\cT(t,\omega)\circ S_1(t,s,\omega)\circ \cT^{-1}(s,\omega)
\end{equation*}
defines a conjugated (measurable and continuous) stochastic flow. If $\cT$ is
stationary and $S_1(t,s,\omega)$ is a cocycle then $S_2(t,s,\omega)$ is a cocycle.
\end{lem}

\begin{lem}\label{lem0928}
Let $S_1(t,s,\omega)$ and $S_2(t,s,\omega)$ be stochastic flows conjugated by
a conjugation mapping $\cT$ consisting of uniformly continuous mappings
$\cT(t,\omega):\cX\rightarrow\cX$. Assume that there is a $\widetilde{D}$-attractor
$\widetilde{A}$ for $S_1(t,s,\omega)$. Then
$\cA(t,\omega):=\cT(t,\omega)\widetilde{A}(t,\omega)$ is a random $\cD$-attractor
for $S_2(t,s,\omega)$, where
$$
\cD:=\left\{\{\cT(t,\omega)\widetilde{D}(t,\omega)\}_{t\in\R,\omega\in\Omega}:
\widetilde{D}\in\widetilde{\cD}\right\}.
$$
Moreover, if $\widetilde{A}$ is measurable then so is $\cA$.
\end{lem}

\section{}

Let $(H,\langle\cdot,\cdot\rangle_H)$ be a separable Hilbert space with dual space $H^*$. 
Let $(V,\|\cdot\|_V)$ be a reflexive Banach space such that $V\subset H$
continuously and densely. By identifying $H$ with its dual $H^*$ via the Riesz isomorphism,
we then have the following Gelfand triple
$$
V\subseteq H\subseteq V^*.
$$
For any $T>0$, consider the nonlinear evolution equation
\begin{equation}\label{0608:01}
\begin{cases}
u'(t)=A(t,u(t)),\quad 0<t<T,\\
u(0)=u_0\in H,
\end{cases}
\end{equation}
where $u'$ is the generalized derivative of $u$ on (0,T) and 
$A:[0,T]\times V\rightarrow V^*$ is restrictedly measurable, i.e. for each $\rmd t$-version
of $u\in L^1([0,T];V)$, $t\mapsto A(t,u(t))$ is $V^*$-measurable on $[0,T]$.

Let us introduce the following conditions.
Suppose that for fixed $\alpha>1$ and $\beta\geq0$ there exist constants $\delta>0$,
$C\in\R$ and a positive function $f,g\in L^1([0,T];\R)$ such that for all $t\in[0,T]$
and $v,v_1,v_2\in V$:
\begin{itemize}
\item[(A1)] (Hemicontinuity)
The map $\theta\mapsto ~_{V^*}\langle A(t,v_1+\theta v_2),v\rangle_V$ is continuous 
on $\R$. 
\item[(A2)] (Local monotonicity)
$$
~_{V^*}\langle A(t,v_1)-A(t,v_2),v_1-v_2\rangle_V
\leq \left(f(t)+\rho(v_1)+\eta(v_2)\right)\|v_1-v_2\|_H^2,
$$
where $\rho,\eta:V\rightarrow [0,+\infty)$ are measurable and 
satisfy that there exist $C,\gamma\in[0,+\infty)$ such that 
$$
\rho(v)+\eta(v)\leq C\left(1+\|v\|_V^\alpha\right)\left(1+\|v\|_H^\gamma\right).
$$
\item[(A3)] (Coercivity)
$$
~_{V^*}\langle A(t,v),v\rangle_V\leq-\delta\|v\|_V^\alpha+g(t)\|v\|_H^2+f(t).
$$
\item[(A4)] (Growth)
$$
\|A(t,v)\|_{V^*}^{\frac{\alpha}{\alpha-1}}\leq \left(f(t)+C\|v\|_V^\alpha \right)
\left(1+\|v\|_H^\beta\right).
$$
\end{itemize}

\begin{thm}\label{EUsolution}
Assume that $V\subseteq H$ is compact and (A1)-(A4) hold.
Then for any $u_0\in H$, there exists a unique solution $u$ to \eqref{0608:01} on $[0,T]$,
i.e. 
\begin{equation*}
u\in L^\alpha([0,T];V)\cap C([0,T];H),\quad 
u'\in L^{\frac{\alpha}{\alpha-1}}([0,T];V^*)
\end{equation*}
and for all $t\in[0,T]$ and $v\in V$,
\begin{equation*}
\langle u(t),v\rangle_H=\langle u_0,v\rangle_H+\int_0^t
~_{V^*}\langle A(s,u(s)),v\rangle_V\rmd s.
\end{equation*}
\end{thm}

On a complete probability space $(\Omega, \mathcal{F}, \mathbb{P})$ equipped with a right-continuous and complete filtration $\mathbb{F} = (\mathcal{F}_t)_{t \geq 0}$, we now consider an $\mathbb{R}^d$-valued It\^o process:
\begin{equation}\label{0604--1}
\rmd X_t = b_t \rmd t + \sigma_t \rmd W_t, \quad X_0 \in L^2(\Omega, \mathcal{F}, \mathbb{P}),
\end{equation}
where $(W_t)_{t \geq 0}$ is an $\mathbb{F}$-Brownian motion with values in $\mathbb{R}^d$, and $(b_t)_{t \geq 0}$ and $(\sigma_t)_{t \geq 0}$ are $\mathbb{F}$-progressively measurable processes with values in $\mathbb{R}^d$ and $\mathbb{R}^{d \times d}$ respectively. We assume that
\begin{equation}\label{0604--2}
\forall ~T > 0, \quad \mE \left[ \int_0^T (|b_t|^2 + |\sigma_t|^4) \rmd t \right] < +\infty.
\end{equation}
Now we recall the following It\^o formula; see \cite[Chapter 5]{Carmona1} for more details.
\begin{prop}\label{20250526-1}
Fix $u\in\Xi$. Let $\mu_t = \sL_{X_t}$ for $t \in [0, T]$ for an It\^o process $(X_t)_{0 \leq t \leq T}$ of the form \eqref{0604--1} satisfying \eqref{0604--2}. 
If $(\xi_t)_{t \in [0, T]}$ is another $d$-dimensional It\^o process on the same filtered probability space $(\Omega, \mathcal{F}, \mathbb{P})$ with similar dynamics $\rmd\xi_t = \eta_t \rmd t + \gamma_t \rmd W_t$, for two $\mathbb{F}$-progressively measurable processes $(\eta_t)_{0 \leq t \leq T}$ and $(\gamma_t)_{0 \leq t \leq T}$ with values in $\mathbb{R}^d$ and $\mathbb{R}^{d \times d}$ respectively such that
\[
\mathbb{P} \left[ \int_0^T (|\eta_t| + |\gamma_t|^2) \rmd t < \infty \right] = 1,
\]
then, $\mathbb{P}$-almost surely for all $t \in [0, T]$, it holds
\begin{equation}
\begin{aligned}
u(\xi_t, \mu_t) = u(\xi_0, \mu_0) &+ \int_0^t  \partial_x u(\xi_s, \mu_s) \cdot \eta_s \rmd s + \int_0^t \partial_x u(\xi_s, \mu_s) \cdot (\gamma_s \rmd W_s) \\
&+ \frac{1}{2} \int_0^t \text{tr} \left[ \partial^2_{xx} u(\xi_s, \mu_s) \gamma_s \gamma_s^\top \right] \rmd s \\
&+ \int_0^t \tilde{\mE} \left[ \partial_{\mu} u(\xi_s, \mu_s)(\tilde{X}_s) \cdot \tilde{b}_s \right] \rmd s \\
&+ \frac{1}{2} \int_0^t \tilde{\mE} \left[ \text{tr} \left( \partial_v \partial_\mu u( \xi_s, \mu_s)(\tilde{X}_s) \tilde{\sigma}_s \tilde{\sigma}_s^\top \right) \right] \rmd s,
\end{aligned}
\end{equation}
where the process $(\tilde{X}_s, \tilde{b}_s, \tilde{\sigma}_s)_{0 \leq s \leq t}$ is a copy of the process $(X_s, b_s, \sigma_s)_{0 \leq s \leq t}$ on a copy $(\tilde{\Omega}, \tilde{\mathcal{F}}, \tilde{\mathbb{P}})$ of the probability space $(\Omega, \mathcal{F}, \mathbb{P})$.
\end{prop}

We recall the following stochastic Gronwall's lemma, which is from \cite[Lemma 5.3]{GZ09}.
\begin{lem}\label{sGlem}
Let $T>0$. Assume that $X,Y,Z,R:[0,T)\times\Omega\rightarrow\R$ are real-valued, nonnegative stochastic
processes. Suppose further that for some stopping time $\tau<T$ there exists a constant $c_1>0$
such that
$$
\int_0^\tau R(s)\rmd s<c_1 \quad a.s.
$$
and
$$
\mE\int_0^\tau\left(R(s)X(s)+Z(s)\right)\rmd s<\infty.
$$
If there exists constant $c_2>0$ such that for all stopping times $0\leq \tau_1\leq \tau_2\leq\tau$,
\begin{equation*}
\mE\left(\sup_{t\in[\tau_1,\tau_2]}X(t)+\int_{\tau_1}^{\tau_2}Y(s)\rmd s\right)
\leq c_2\mE\left(X(\tau_1)+\int_{\tau_1}^{\tau_2}(R(s)X(s)+Z(s))\rmd s\right),
\end{equation*}
then 
\begin{equation*}
\mE\left(\sup_{t\in[0,\tau]}X(t)+\int_{0}^\tau Y(s)\rmd s\right)
\leq \tilde{c}\mE\left(X(0)+\int_0^\tau Z(s)\rmd s\right),
\end{equation*}
where $\tilde{c}:=C_{c_2,T,c_1}$.
\end{lem}

\section*{Acknowledgments}

This work is supported by 
National Key R\&D Program of China (No. 2023YFA1009200), NSFC (Grants 12531009, 12301223, 124B2009 and 11925102), and Liaoning Revitalization Talents Program (Grant XLYC2202042).

We are grateful to Professor Michael R\"ockner for his valuable comments on the first draft of this manuscript.

\end{document}